%% file: Part1_v1_arxiv.tex
\documentclass[12pt, reqno, oneside]{amsart}

\usepackage{amsfonts, amstext, amsmath, amsthm, amscd, amssymb, enumitem, url}
\usepackage{mathabx}
\usepackage{tikz-cd}
\usetikzlibrary{shapes.geometric, arrows}
\usepackage{xcolor}
\usepackage{svg}
\usepackage{pdfpages}
\usepackage{pdflscape}
\usepackage{spverbatim}

\usepackage{geometry}
\usepackage{todonotes}
\usepackage{microtype}
\usepackage[parfill]{parskip}

\usepackage[colorlinks,citecolor=cyan,linkcolor=magenta]{hyperref}
\usepackage[nameinlink]{cleveref}

\usepackage{array}
\newcolumntype{C}[1]{>{\centering\let\newline\\\arraybackslash\hspace{0pt}}m{#1}}
\usepackage{multicol}
\usepackage{multirow}
\usepackage{makecell}
\usepackage{booktabs}

\newtheorem{thm}{Theorem}[section]
\newtheorem{cor}[thm]{Corollary}
\newtheorem{lemma}[thm]{Lemma}
\newtheorem{prop}[thm]{Proposition}

\newtheorem{defn}[thm]{Definition}
\newtheorem{rmk}[thm]{Remark}

\newtheorem{eg}[thm]{Example}

\numberwithin{equation}{section}

\newcommand{\wh}{\widehat}

\renewcommand{\epsilon}{\varepsilon}

\usepackage{stmaryrd}

\renewcommand{\top}{\mathrm{top}}
\renewcommand{\bot}{\mathrm{bot}}

\newcommand{\intr}{\mathrm{int}}

\newcommand{\brloc}{\mathrm{brloc}}

\newcommand{\boa}{\boldsymbol{\alpha}}
\newcommand{\bob}{\boldsymbol{\beta}}
\newcommand{\x}{\mathbf{x}}
\newcommand{\y}{\mathbf{y}}
\newcommand{\z}{\mathbf{z}}
\newcommand{\s}{\mathfrak{s}}
\newcommand{\xt}{\mathbf{x}^\text{top}}
\newcommand{\xb}{\mathbf{x}^\text{bot}}
\newcommand{\sym}{\text{Sym}}
\newcommand{\Z}{\mathbb{Z}}
\newcommand{\R}{\mathbb{R}}
\newcommand{\T}{\mathbb{T}}
\newcommand{\C}{\mathbb{C}}
\newcommand{\Fo}{\mathcal{F}}

\newcommand{\even}{\mathrm{even}}
\newcommand{\odd}{\mathrm{odd}}

\DeclareMathOperator{\Fit}{\mathrm{Fit}}
\DeclareMathOperator{\coker}{\mathrm{coker}}
\DeclareMathOperator{\rank}{\mathrm{rank}}
\DeclareMathOperator{\spn}{\mathrm{span}}
\DeclareMathOperator{\sgn}{\mathrm{sign}}
\DeclareMathOperator{\tr}{\mathrm{tr}}
\DeclareMathOperator{\len}{\mathrm{length}}
\DeclareMathOperator{\spinc}{\text{spin}^\text{c}}
\DeclareMathOperator{\Spinc}{\text{Spin}^\text{c}}

\begin{document}

\title[Heegaard Floer theory and pseudo-Anosov flows I]{Heegaard Floer theory and pseudo-Anosov flows I: Generators and categorification \\ of the zeta function}

\author{Antonio Alfieri}
\address{department of mathematics\\
University of Georgia at Athens \\
1023 D. W. Brooks Drive, Athens, GA 30605}
\email{alfieriantonio90@gmail.com}

\author{Chi Cheuk Tsang}
\address{Département de mathématiques \\
Université du Québec à Montréal \\
201 President Kennedy Avenue \\
Montréal, QC, Canada H2X 3Y7}
\email{tsang.chi\_cheuk@uqam.ca}

\begin{abstract} 
We bring to light a new connection between dynamics and Heegaard Floer homology. On a closed 3-manifold $Y$ we consider a pseudo-Anosov flow $\phi$  with no perfect fits with respect to its singularity locus $L \subset Y$, or perhaps a larger collection of closed orbits.
Using work of Agol and Guéritaud on veering branched surfaces we produce a chain complex computing the link Floer homology of $L$ in the framing specified by the degeneracy curves of the flow.
Using work of Landry, Minsky, and Taylor we show that the generators of the chain complex correspond to certain closed multi-orbits of $\phi$.
We prove that two canonical generators $\mathbf{x}^\mathrm{top}$ and $\mathbf{x}^\mathrm{bot}$ determine non-trivial homology classes located in the $\spinc$-grading of the flow, and its opposite.
Finally, we observe that our specific model of the chain complex for link Floer homology naturally supports a grading with dynamical significance. This grading, a modification of the regular Maslov grading, is shown to categorify a suitable normalization of the zeta function associated to $\phi$.
\end{abstract}

\maketitle

\setcounter{tocdepth}{3}
\makeatletter
\def\l@subsection{\@tocline{2}{0pt}{2.5pc}{5pc}{}}
\def\l@subsubsection{\@tocline{2}{0pt}{5pc}{7.5pc}{}}
\makeatother

\section{Introduction}

\thispagestyle{empty}

This paper concerns the study of certain three-dimensional dynamical systems with hyperbolic dynamics called pseudo-Anosov flows. 
A flow $\phi$ on a closed oriented 3-manifold $Y$ is \textit{pseudo-Anosov} if there exists a pair of singular two-dimensional foliations $\Fo^s$ and $\Fo^u$, called the \emph{stable} and \emph{unstable foliations} respectively, such that:
\begin{itemize}
    \item The singularity locus of $\Fo^s$ and $\Fo^u$ coincide and is equal to a finite union of closed orbits of $\phi$, called the \emph{singular orbits} of the flow. 
    Close to these orbits the two foliations look as in \Cref{fig:paflow} left.
    \item Away from the singular orbits, the leaves of $\Fo^s$ and $\Fo^u$ intersect transversely in the orbits of $\phi$. In particular each leaf of $\Fo^s$ and of $\Fo^u$ is foliated by orbits of $\phi$.
    \item The orbits in each leaf of $\Fo^s$ are forward asymptotic, and the orbits in each leaf of $\Fo^u$ are backward asymptotic.
\end{itemize}
 A flow is called \emph{Anosov} if the stable and unstable foliations are non-singular.

Basic examples of pseudo Anosov flows are the geodesics flow on the unit cotangent bundle of a closed hyperbolic surface $S$, and the suspension flow of a pseudo Anosov map $f:S\to S$. The first example that was not the geodesics flow or a suspension flow was described by Handel and Thurston \cite{HT80}. Further examples can be produced by performing Goodman-Fried surgery on a pre-existing flow \cite{Goo83, Fri83}. See also work of Salmoiraghi and the second author \cite{Sal21,Tsa24} for some more recent results.

Despite the amount of work that has been done in the study of the topology and dynamics of pseudo Anosov flows, there is still a lot we do not know about these objects. 
It is expected that each closed 3-manifold admit only finitely many inequivalent (pseudo) Anosov flows, and it was conjectured by Ghys that all Anosov flows are obtained by a finite sequence of Goodman-Fried surgeries from the suspension flow of an Anosov map $f: T^2 \to T^2$. 

Furthermore, if the stable and unstable foliations $\mathcal{F}^s$ and $\mathcal{F}^u$ of a pseudo-Anosov flow are coorientable, they give rise to coorientable taut foliations. Thus the study of these flows is related to the $L$-space conjecture. 
Indeed, in \cite{CD03} Calegari and Dunfield showed that pseudo-Anosov flows gives rise to  circular orderings on the fundamental group $\pi_1(Y)$ of the underlying manifold, and in \cite{BGH24} Boyer, Gordon and Hu used an argument based on pseudo-Anosov flows to confirm the $L$-space conjecture in the case of certain 3-manifolds arising as cyclic branched covers over links. 

It would be interesting to know which 3-manifolds  admitting  taut foliations have an Anosov flow. 
For example we know that all surgeries on the figure-eight knot  admit taut foliations \cite{Rob01}, but it is shown in \cite{Yu23} that only integer surgeries admit Anosov flows. More interestingly, large surgeries on the pretzel knot $P(-2, 3, 7)$ give rise to a Anosov flows despite $P(-2, 3, 7)$ being an $L$-space knot. 
And so it is meaningful to ask for a classification of $L$-spaces admitting Anosov flows, even though these would have non-orientable foliations.

When allowing singularities the situation becomes more general and even more intriguing. Empirically, pseudo-Anosov flows are very prevalent among 3-manifolds. Every irreducible atoroidal 3-manifold with $b_1 \geq 1$ has a pseudo-Anosov flow \cite{Mos96}, and it is a conjecture of Thurston that every irreducible atoroidal 3-manifold admitting a cooriented taut foliation admits also a pseudo-Anosov flow.

Finally, there is an unexplored line of investigation regarding the possible isotopy classes of closed orbits. 
We know, for example, that if a knot $K \subset Y$ is a closed orbit of a pseudo-Anosov flow with orientable stable and unstable foliations, then $K$ is \emph{persistently foliar}, that is, for all but one boundary slope there is a co-oriented taut foliation meeting the boundary of the knot complement transversely in a foliation by curves of that slope. In the Heegaard Floer homology jargon one would say: $L$-space knots are not closed orbits of pseudo-Anosov flows with orientable foliations. 

Similarly, one can ask what topological restrictions a link $L\subset Y$ must satisfy in order for it to be the singularity locus of a pseudo-Anosov flow.
One of the main motivations for our work is to approach this question from the perspective of link Floer homology.

\subsection{Dynamics and Floer theory}

Since Poincaré, a recurring theme in topology is that the dynamics of a flow $\phi : \R \times Y \to Y$ carries topological information about the manifold $Y$.
This is a pervasive idea in symplectic and contact topology:

\begin{itemize}
\item When $\phi$ is a Hamiltonian flow on a closed sympletic manifold $M$, Floer defined a chian comlex  $HF(M, \phi)$ generated by the closed orbits of $\phi$. This is used to show that the number of closed orbits of an Hamiltonian system is bounded from below by the sum of the betti numbers (Arnold conjecture).
\item When $\phi$ is a (nondegenerate) Reeb flow on a closed 3-manifold, Hutchings \cite{Hut14} defined the embedded contact homology $ECH(Y,\phi)$ of $\phi$:
The chain complex $ECC(Y,\phi)$ is generated by certain multi-sets of closed orbits, and the differential is computed by counting pseudo-holomorphic curves. The homology of this chain complex $ECH(Y,\phi)$ turns out to be a topological invariant of the 3-manifold $Y$ \cite{Tau10}. 
\item Similar ideas where pioneered by Eliashberg, Givental and Hofer for Reeb flows in arbitrary dimensions in a set up known as Symplectic Field Theory \cite{EGH00}. 
\end{itemize}

In recent years, work has been put into adapting these methods to the case when $\phi$ is an Anosov flow. 
Recently Cieliebak-Lazarev-Massoni-Moreno \cite{CLMM22}, using the bi-contact machinery developed by Hozoori \cite{Hoz24}, defined an $A_\infty$ category whose objects are the closed orbits.

Furthermore, we have been informed by Zung \cite{Zun25} that is possible to define  a chain complex $CA(\phi)$ that is generated by multi-sets of closed orbits of an Anosov flow, and whose differential is computed by counting Fried pants. The homology of this chain complex $HA(\phi)$ is in general an infinitely generated $\mathbb{F}[\omega]$-module.

In this paper we describe a new approach to this program based on   Heegaard Floer homology, developed by Ozsv\' ath and Szab\' o in 2001.
This has a few advantages. Firstly, it works in general for \textit{pseudo-Anosov flows}, which are much more common than Anosov flows. 
Secondly, the Heegaard Floer homology package, via its theory of domains, allows us to use \emph{combinatorial methods}. This is a huge advantage which we will build on the second paper of the series \cite{AT25b}. 
Thirdly, the \emph{topological meaning} of our  invariant is known \cite{Juh06} \cite{OS08}.
This is in contrast to the aforementioned attempts for Anosov flows, where the output of the theory, at the moment of writing, is not known to be topological invariants of the underlying 3-manifold, or of the singularity link.

\subsection{Summary of the paper} 

We start with a pseudo-Anosov flow $\phi$ on a closed oriented 3-manifold $Y$.
We shall denote with $L\subset Y$ the singularity locus of the pseudo-Anosov flow $\phi$. 

Given a finite collection $\mathcal{C}$ of closed orbits, one defines the \textit{blown-up flow} $\phi^\sharp$ on the link complement $Y^\sharp = Y \backslash \nu(\mathcal{C})$. The restriction of $\phi^\sharp$ to the interior of $M$ can be identified with the restriction of $\phi$ to $Y \backslash \mathcal{C}$.  Meanwhile, each boundary component $T_\gamma$ of $M$ corresponds to a closed orbit $\gamma \in \mathcal{C}$. 
The restriction of $\phi^\sharp$ to $T_\gamma$ contains attracting closed orbit along the circles of intersection between $T_\gamma$ and the leaf of $\Fo^u$ that contains $\gamma$. We refer to each attracting closed orbit as a \textit{degeneracy curve}.

We say that $\phi$ has \textit{no perfect fits} relative to $\mathcal{C}$ if $\phi^\sharp$ does not admit two closed orbits $\gamma_1$ and $\gamma_2$ such that $\gamma_1$ is homotopic to $\gamma_2^{-1}$ in $Y^\sharp$.
We shall work under the technical hypothesis that: (1) $\phi$ has no perfects fits relative  to the singularity link, or (2) more generally, $\phi$ has no perfect fits relative to some possibly larger finite collection of closed orbits $\mathcal{C} \supset L$. This second  hypothesis is less restrictive and is automatically satisfied if the pseudo-Anosov flow is \emph{transitive}, that is, it has a dense orbit.  

In this situation we consider the link Floer homology of $\mathcal{C}\subset Y$ framed by the degeneracy locus of the blown-up flow $\phi^\sharp$. That is we shall consider the total space $Y^\sharp = Y \backslash \nu(\mathcal{C})$ of $\phi^\sharp$ as a sutured manifold with two sutures for each degeneracy curve of $\phi^\sharp$. This has an associated sutured Floer homology group $SFH(Y^\sharp)$ in the sense of Juh\' asz \cite{Juh06} that can be interpreted as a framed version of the link Floer homology of $(Y, \mathcal{C})$. (To compute the regular link Floer homology as in \cite{OS08} one would have to place as sutures two meridians on each boundary component of $Y \backslash \nu(\mathcal{C})$.)

\begin{thm} \label{thm:introinformal}
Let $\phi$ be a pseudo-Anosov flow on a closed oriented 3-manifold $Y$.
Let $\mathcal{C}$ be a finite collection of closed orbits relative to which $\phi$ has no perfect fits.
Then there exists a chain complex $SFC(\phi,\mathcal{C})$ whose generators correspond to certain closed multi-orbits of $\phi^\sharp$ and whose homology equals the sutured Heegaard Floer homology $SFH(Y^\sharp)$ of the total space of the blown-up flow.
\end{thm}

Our starting point to construct the model chain complex $SFC(\phi,\mathcal{C})$ is the work of Agol-Guéritaud and Schleimer-Segerman \cite{SS19} putting in one-to-one correspondence pairs $(\phi, \mathcal{C})$ as in \Cref{thm:introinformal} and veering branched surfaces (see \Cref{defn:vbs}).  
Given a veering branched surface $B$, we construct in \Cref{sec:vbstosfh} a canonical Heegaard diagram $(\Sigma, \boa, \bob)$ associated to $B$, and we define $SFC(\phi,\mathcal{C}) = CF(\Sigma, \boa, \bob)$. We then use work of Landry-Minsky-Taylor \cite{LMT23} to associate to each generator of $SFC(\phi,\mathcal{C})$ a multi-orbit of $\phi^\sharp$. 

In sutured Heegaard Floer homology, every generator of the chain complex has an associated $\spinc$-structure. This gives rise to the $\spinc$-grading on the group $SFH(Y^\sharp)$ we study in this paper.
Here, recall that a \textit{$\spinc$-structure} is an equivalence class of vector fields. 
For example, a vector field directing the \textit{blown-up flow} $\phi^\sharp$ pins down a $\spinc$-structure on $Y^\sharp$ which we denote by $\s_{\phi^\sharp}$. Note that the set $\Spinc(Y^\sharp)$ of all $\spinc$-structures is an affine space over $H^2(Y^\sharp, \partial Y^\sharp)\cong H_1(Y^\sharp)$, and so the $\spinc$-grading can be thought as a $H_1(Y^\sharp)$-grading. This $\spinc$-grading admits a neat description in our chain complex $SFC(\phi, \mathcal{C})$.

\begin{thm} \label{thm:introspinc}
The homology class of the multi-orbit $\gamma_\x$ associated to a generator $\x$ of the model chain complex $SFC(\phi,\mathcal{C})$ determines its $\spinc$-grading via the formula:
\[\s(\x)= \overline{\s_{\phi^\sharp}} + PD[\gamma_\x] \ ,\]  
where $PD: H_1(Y^\sharp) \to H^2(Y^\sharp, \partial Y^\sharp)$ denotes Poincaré duality.
\end{thm}

As a consequence of the special combinatorics of veering branched surfaces, we can pick out two special generators of the chain complex $SFC(\phi, \mathcal{C})$.

\begin{thm} \label{thm:introtopbotgenerator}
There is a canonically defined generator $\xb$, corresponding to the empty multi-orbit, whose associated $\spinc$-structure $\mathfrak{s}(\xb)$ equals $\overline{\s_{\phi^\sharp}}$.
The generator $\xb$ has bottom $\spinc$-grading in the sense that for every other generator $\y$, the element $\mathfrak{s}(\y)-\mathfrak{s}(\xb)$ is the homology class of a nonempty closed multi-orbit.

Symmetrically, there is a canonically defined generator $\xt$, whose associated $\spinc$ structure equals $\s_{\phi^\sharp}$. 
The generator $\xt$ has top $\spinc$-grading in the sense that for every other generator $\y$, the element $\mathfrak{s}(\xt) - \mathfrak{s}(\y)$ is the homology class of a nonempty closed multi-orbit.
\end{thm}

We then show that there is no holomorphic disk connecting the special generators to any other generator of the Heegaard Floer chain complex.

\begin{thm} \label{thm:introtopbotnontrivial}
The generators $\xt$ and $\xb$ are cycles determining non-trivial homology classes in the sutured Floer homology $SFH(Y^\sharp)$. 
\end{thm}

\Cref{thm:introtopbotnontrivial} implies the following corollary.

\begin{cor} \label{cor:introdimgeq2}
Let $(Y,\phi)$ be a pseudo-Anosov flow on a closed 3-manifold $Y$, $\mathcal{C}$ be a finite collection of closed orbits relative to which $\phi$ has no perfect fits, and $(Y^\sharp, \phi^\sharp)$ the associated blown-up flow. Then $\dim SFH(Y^\sharp,\s_{\phi^\sharp}) \geq 1$ and $\dim SFH(Y^\sharp,\overline{\s_{\phi^\sharp}}) \geq 1$.
\end{cor}

We also recover the easier direction of the fiberedness detection theorem \cite{Juh08}.

\begin{cor} \label{cor:introfibered}
Suppose that $(Y,\phi)$ is the suspension flow of a pseudo-Anosov mapping $f:S \to S$ of a closed surface $S$. Let $\mathcal{C}$ be any finite collection of closed orbit containing the singular orbits.
We define $e = \sum_{x \in S \cap \mathcal{C}} p_x$, where the sum is taken over the periodic points $x \in S \cap \mathcal{C}$, and we suppose each $x$ is $p_x$-pronged.

Then under the identification $\Spinc(Y^\sharp) \cong H_1(Q)$ that sends $\overline{\s_{\phi^\sharp}}$ to $0$,
$$\dim SFH(Y^\sharp,n) := \dim \bigoplus_{\langle \mathfrak{s}, [S^\sharp] \rangle = n} SFH(Y^\sharp,\mathfrak{s}) =
\begin{cases}
0 & \text{if $n > -\frac{3e}{2}$} \\
1 & \text{if $n = -\frac{3e}{2}$} \\
1 & \text{if $n = 0$} \\
0 & \text{if $n < 0$}
\end{cases}.$$
\end{cor}

In the sequel \cite{AT25b}, we will upgrade \Cref{cor:introfibered} by showing that the dimensions of the next nonzero gradings count the number of periodic points of $\phi^\sharp$ of the least period. In particular, the second-to-top grading counts the number of fixed points of $\phi^\sharp$.

In this paper, we address the  question of counting the dimension `with signs', i.e. we study a decategorification of $SFH(Y^\sharp)$.
Let $G$ be the finitely generated free abelian group $H_1(Y^\sharp)/\text{Torsion}$.
Recall that given a $\mathbb{Z}/2$-grading $\nu$ on $SFH(Y^\sharp)$ (and an identification $\Spinc(Y^\sharp) \cong H_1(Y^\sharp)$), we can form the polynomial invariant 
$$\chi_\nu (SFH(Y^\sharp)) = \sum_{\mathfrak{s} \in \Spinc(Y^\sharp)} (\dim (SFC_{\nu=0}(Y^\sharp,\mathfrak{s})) - \dim (SFC_{\nu=1}(Y^\sharp,\mathfrak{s}))) \cdot \mathfrak{s} \in \mathbb{Z}[G].$$
In \cite{FJR11}, Friedl-Juhász-Rasmussen showed that if $\nu$ is the Maslov grading, then $\chi_\nu (SFH(Y^\sharp))$ is a type of Turaev torsion of $Y^\sharp$. In particular it depends only on the algebraic data of the homomorphisms $\pi_1(R_\pm(Y^\sharp)) \to \pi_1(Y^\sharp)$.

We show that under a different grading, $\chi_\nu(SFH(Y^\sharp))$ recovers a dynamical invariant: the \textit{zeta function} of $\phi^\sharp$, defined as 
$$\zeta_{\phi^\sharp} = \prod_{\text{primitive $\gamma$}} (1-[\gamma])^{-1} \in \mathbb{Z}[[G]].$$

\begin{thm} \label{thm:introzeta}
Under the same setting as \Cref{thm:introinformal}, if $b_1(Y^\sharp) \geq 2$, then there is a $\mathbb{Z}/2$-grading $\nu$ on $SFH(Y^\sharp)$ such that
$$\chi_\nu (SFH(Y^\sharp)) = \zeta_{\phi^\sharp}^{-1}.$$
\end{thm}

The equality in \Cref{thm:introzeta} holds with some additional factors if instead $b_1(Y^\sharp)=1$.
We refer to \Cref{prop:zetablowupantiveering} and \Cref{thm:SFHcategorifiesantiveering} for detailed statements.

\subsection{Outline of the paper}
The paper is organized as follows.
The first part consists of \Cref{sec:paflowvbs} and \Cref{sec:suturedfloer} and covers preliminaries on pseudo-Anosov flows, and sutured Floer homology respectively. Since this paper addresses problems and ideas studied by researchers working in different subareas of mathematics, we opted for a lengthier exposition. We expect \Cref{sec:paflowvbs} to be more beneficial to people with a background in Floer theory, and \Cref{sec:suturedfloer} to people with a stronger background in pseudo-Anosov dynamics.

The second part includes \Cref{sec:vbstosfh} and \Cref{sec:chaincomplexcombin}.
\Cref{sec:vbstosfh} explains the main construction of the paper. More specifically, we explain how to associate to a veering branched surface $B$ on a 3-manifold with torus boundary $M$ a canonical balanced Heegaard diagram $(\Sigma, \boa, \bob)$ for $M$ with sutured structure $\Gamma$ induced by $B$ on $\partial M$.
We study the generators of the Heegaard Floer chain complex $CF(\Sigma, \boa, \bob)$ and show that if $B$ corresponds to a pseudo-Anosov flow $\phi$ and a collection of closed orbits $\mathcal{C}$ then the generators of $CF(\Sigma, \boa, \bob)$ correspond to certain multi-orbits of $\phi^\sharp$. 

In \Cref{sec:chaincomplexcombin} we go deeper in the understanding of the combinatorics of the Heegaard diagram $(\Sigma, \boa, \bob)$. We record the structure of the diagram and use this structure to show that the diagram is admissible, as well as to show that the top and bottom generators $\xt$ and $\xb$ determine non-vanishing homology classes.

The third part runs from \Cref{sec:polyinv} to \Cref{sec:sfhcategorifies}, and is aimed at showing \Cref{thm:introzeta}.
In \Cref{sec:polyinv}, we define some polynomial invariants of veering branched surfaces and discuss the relation between these and various zeta functions associated to $\phi$.
The proof of one of the results, \Cref{thm:tautantiveering}, is deferred to \Cref{sec:factorantiveerpoly} since the ideas involved do not play a role in the rest of the paper.
In \Cref{sec:sfhcategorifies}, we show that $SFH$ categorifies one of the polynomial invariants.

{\it {\bf Acknowledgments.} The first author would like to thank Gordana Mati\' c for many useful conversations that helped shape this manuscript.
The second author would like to thank Ian Agol and Jonathan Zung for many enlightening conversations on Floer theory, pseudo-Anosov flows, and veering triangulations, many of which influenced the directions and results of this paper. 
We would also like to thank Robert Lipshitz and Anna Parlak for helpful comments. 
The ideas we expose in this paper were developed when the two authors were postdoctoral fellows at CIRGET, the geometry and topology lab of Université du Québec à Montréal. We would like to thank Steven Boyer for creating the fertile scientific environment in which this collaboration originated.}

{\bf Notational conventions.} Throughout this paper:
\begin{itemize}
    \item We use the following graph theory terminology: Let $\Gamma$ be a directed graph. 
    \begin{itemize}
        \item A \textit{loop} in $\Gamma$ will mean a cyclic sequence of edges $(e_i)_{i \in \mathbb{Z}/N}$ where the terminal vertex of $e_i$ equals the initial vertex of $e_{i+1}$ for each $i$. We consider two loops to be the same if they are related by reindexing the edges by a cyclic permutation.
        \item An \textit{embedded loop} in $\Gamma$ is a loop that visits each vertex at most once. 
        A \textit{primitive loop} is a loop that is not a multiple of another loop. 
        For example, an embedded loop must be primitive. 
        When the context is clear, we will conflate an embedded loop $(e_i)_{i \in \mathbb{Z}/N}$ with the union $\bigcup_{i \in \mathbb{Z}/N} e_i$.
        \item A \textit{multi-loop} in $\Gamma$ is a finite collection of loops. An \textit{embedded multi-loop} is a finite collection of embedded loops that are pairwise disjoint. For example, the empty set is an embedded multi-loop.
        \item A \textit{cycle} in $\Gamma$ is a simplicial homology 1-cycle, i.e. an element of $H_1(\Gamma) \subset C_1(\Gamma)$. Each multi-loop determines a cycle, but multiple multi-loops can determine the same cycle. On the other hand, each cycle is determined by at most one embedded loop.
    \end{itemize}
    \item We use the following dynamics terminology:
    \begin{itemize}
        \item A \textit{flow} on a space $X$ is a continuous map $\phi: \mathbb{R} \times X \to X$ such that $\phi(0,x) = x$ and $\phi(s,\phi(t,x))=\phi(s+t,x)$ for all $s,t \in \mathbb{R}$ and all $x \in X$.
        \item A \textit{closed orbit} of $\phi$ is a continuous map $\alpha: \mathbb{R}/L\mathbb{Z} \to X$ of the form $\alpha(t) = \phi(t,x)$ for some $x \in X$. We consider two closed orbits to be the same if they are related by reparametrization.
        \item A \textit{primitive closed orbit} is a closed orbit that is not a multiple of another closed orbit. 
        When the context is clear, we will conflate a primitive closed orbit with its image.
        \item A \textit{closed multi-orbit} is a finite collection of closed orbits.
        For example, the empty set is a closed multi-orbit.
    \end{itemize}
\end{itemize}

\section{Pseudo-Anosov flows and veering branched surfaces} \label{sec:paflowvbs}

In this section, we recall some facts about pseudo-Anosov flows and veering branched surfaces, including the correspondence theorem between these objects. A general reference for this material is \cite[Chapters 1 and 2]{Tsathesis}.

\subsection{Pseudo-Anosov flows}

For the purposes of this paper, a \textit{pseudo-Anosov flow} on a closed 3-manifold $Y$ is a flow $\phi$ for which there is a pair of singular 2-dimensional foliations $(\Fo^s, \Fo^u)$ whose leaves intersect transversely in the flow lines, such that the flow lines in each leaf of the \textit{stable foliation} $\Fo^s$ are forward asymptotic, and the flow lines in each leaf of the \textit{unstable foliation} $\Fo^u$ are backward asymptotic.
The singularity locus of the stable and unstable foliations must coincide and be equal to a collection of primitive closed orbits of $\phi$. 
We refer to these as the \emph{singular} orbits of $\phi$.

We illustrate a local picture of $(\Fo^s, \Fo^u)$ near a 3-pronged singular orbit in \Cref{fig:paflow} left, and a local picture of $(\Fo^s, \Fo^u)$ away from the singular orbits in \Cref{fig:paflow} right.

\begin{figure}
    \centering
    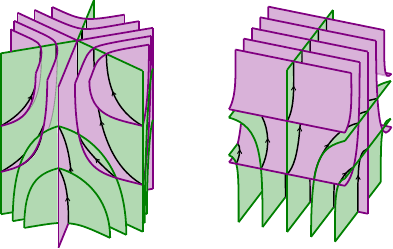
    \caption{Local picture of a pseudo-Anosov flow. Left: Near a singular orbit. Right: Away from a singular orbit.}
    \label{fig:paflow}
\end{figure}

The precise definition of a pseudo-Anosov flow involves some details on regularity and Markov partitions; we refer the reader to \cite[Section 4]{FM01}.
See also \cite[Section 3.1]{Mos96} or \cite[Section 5.4]{AT24} for some discussion on differing definitions in the literature.
These technicalities will not be important for this paper.

Two pseudo-Anosov flows $(Y_1,\phi_1)$ and $(Y_2,\phi_2)$  are \textit{orbit equivalent} if there exists a homeomorphism $h:Y_1 \to Y_2$ that sends the orbits of $\phi_1$ to those of $\phi_2$ in an orientation preserving way, but not necessarily preserving their parametrizations. 
We will consider orbit equivalent flows as being equivalent.

Given a finite collection of primitive closed orbits $\mathcal{C}$, one can define the \textit{blown-up flow} $\phi^\sharp$ on the complement of said collection of closed orbits $Y^\sharp= Y \backslash \nu(\mathcal{C})$.
The restriction of $\phi^\sharp$ to the interior of $Y^\sharp$ can be identified with the restriction of $\phi$ to $Y \backslash \mathcal{C}$. 
Meanwhile, each boundary component $T_\gamma$ of $Y^\sharp$ corresponds to a primitive closed orbit $\gamma \in \mathcal{C}$. 
The stable and unstable foliations of $\gamma$ intersect $T_\gamma$ in an equal number of circles. 
When restricted to $T_\gamma$, the blown-up flow $\phi^\sharp$ has repelling and attracting closed orbits exactly along these circles, respectively.
The \textit{degeneracy class} of $\phi$ on $T_\gamma$ is the isotopy class of the collection of primitive repelling closed orbits on $T_\gamma$, or equivalently, the isotopy class of the collection of primitive attracting closed orbits on $T_\gamma$.
We illustrate an example in \Cref{fig:paflowblowup} where we blow up along \Cref{fig:paflow} left. 

\begin{figure}
    \centering
    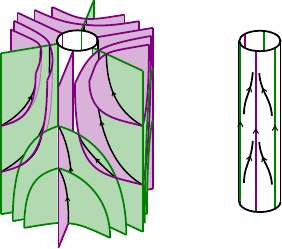
    \caption{Blowing up along \Cref{fig:paflow} left.}
    \label{fig:paflowblowup}
\end{figure}

Conversely, there is a map $\pi: Y^\sharp \to Y$ that collapses each boundary component down to a primitive closed orbit and sends the flow lines of $\phi^\sharp$ to the flow lines of $\phi$. 
More generally, given a collection of slopes $s=(s_T)$ on each boundary component $T$ of $Y^\sharp$ such that $|\langle s_T,l_T \rangle| \geq 2$, where $l_T$ is the degeneracy class of $\phi$ on $T$, there is a pseudo-Anosov flow $\phi^\sharp(s)$ on the closed manifold $Y^\sharp(s)$ obtained by Dehn filling $Y^\sharp$ along the slope $s_T$ for each boundary component $T$ and a map $\pi: Y^\sharp \to Y^\sharp(s)$ that collapses each boundary component $T$ down to a $|\langle s_T,l_T \rangle|$-pronged primitive closed orbit and sends the flow lines of $\phi^\sharp$ to the flow lines of $\phi^\sharp(s)$. 
We denote by $\mathcal{C}(s)$ the collection of primitive closed orbits that is the image of $\partial Y^\sharp$ under $\pi$.
We refer to this procedure as \textit{blowing down} $\phi^\sharp$ along the multi-slope $s$.

See \cite{LMT24b} for a more careful treatment of blow-ups and blow-downs.

Finally, we recall the notion of no perfect fits.
Let $\phi$ be a pseudo-Anosov flow on a closed 3-manifold $Y$. Consider the lifted flow $\widetilde{\phi}$ on the universal cover $\widetilde{Y}$. The \textit{orbit space} $\mathcal{O}$ is defined to be the set of flow lines of $\widetilde{\phi}$, equipped with the quotient topology. It is shown in \cite[Proposition 4.1]{FM01} that $\mathcal{O}$ is homeomorphic to the plane $\mathbb{R}^2$. 
The lifted 2-dimensional stable/unstable foliations $\widetilde{\Fo^{s/u}}$ on $\widetilde{Y}$ induce 1-dimensional foliations $\mathcal{O}^{s/u}$ on $\mathcal{O}$.
Let $\mathcal{C}$ be a nonempty finite collection of primitive closed orbits of $\phi$.
Then the preimage $\widetilde{\mathcal{C}}$ of $\mathcal{C}$ is a collection of flow lines of $\widetilde{\phi}$, thus determines a set of points in $\mathcal{O}$.

A \textit{perfect fit rectangle} is a properly embedded rectangle-with-one-ideal-vertex $R$ in $\mathcal{O}$ where the restrictions of $\mathcal{O}^{s/u}$ foliate $R$ as a product. In other words, $R$ is the image of a proper embedding $[0,1]^2 \backslash \{(1,1)\} \to \mathcal{O}$ so that the pullback of $\mathcal{O}^{s/u}$ to $[0,1]^2 \backslash \{(1,1)\}$ is the foliation by vertical/horizontal lines respectively. See \Cref{fig:perfectfitrect}.
We say that $\phi$ \textit{has no perfect fits relative to $\mathcal{C}$} if every perfect fit rectangle intersects $\widetilde{\mathcal{C}}$.
We say that $\phi$ \textit{has no perfect fits} if it has no perfect fits with respect to the collection of singular orbits.

\begin{figure}
    \centering
    \resizebox{!}{2.5cm}{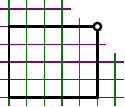}
    \caption{A perfect fit rectangle.}
    \label{fig:perfectfitrect}
\end{figure}

The dynamical significance of the no perfect fit condition is recorded by the following proposition.

\begin{prop} \label{prop:npfhtpyclassunique}
Let $\phi$ be a pseudo-Anosov flow with no perfect fits relative to a collection of closed orbits $\mathcal{C}$. Let $\phi^\sharp$ be the blow-up of $\phi$ along $\mathcal{C}$. Then two distinct closed orbits $\gamma_1$ and $\gamma_2$ of $\phi^\sharp$ are homotopic in $Y^\sharp$ only if they lie on the same boundary component.
\end{prop}
\begin{proof}
This is implied by the arguments in \cite[Lemma 2.13]{Tsa24a}. For completeness, we sketch a proof.

Suppose $\gamma_1$ and $\gamma_2$ are homotopic in $Y^\sharp$.
For each boundary component $T$ of $Y^\sharp$, we pick a slope $s_T$ such that $|\langle s_T,l_T \rangle| \geq 3$ where $l_T$ is the degeneracy class. Via transferring perfect fit rectangles between orbit spaces as in \cite[Proposition 2.7]{Tsa24a}, one can see that the blown-down flow $\phi^\sharp(s)$ has no perfect fits relative to $\mathcal{C}(s)$. 
Since all orbits in $\mathcal{C}(s)$ are singular, $\phi^\sharp(s)$ has no perfect fits.

Meanwhile, the images of $\gamma_1$ and $\gamma_2$ are homotopic in $Y^\sharp(s)$. Hence by \cite[Theorem 4.8]{Fen99}, these images are equal. Thus $\gamma_1$ and $\gamma_2$ must have been quotiented down from the same boundary component of $Y^\sharp$.
\end{proof}

We point out that the no perfect fits condition is non-restrictive for most pseudo-Anosov flows of interest: 
A pseudo-Anosov flow on a closed 3-manifold $Y$ is \textit{transitive} if the set of closed orbits of $\phi$ is dense in $Y$.
The arguments in \cite[Proposition 2.7]{Mos92a} imply that a pseudo-Anosov flow is intransitive if and only if there is a transverse separating torus.
In particular, a pseudo-Anosov flow on an atoroidal 3-manifold must be transitive.

\begin{prop}[{\cite[Proposition 2.7, Remark 2.11]{Tsa24a}}]
Let $\phi$ be a pseudo-Anosov flow. There exists a collection of closed orbits $\mathcal{C}$ relative to which $\phi$ has no perfect fits if and only if $\phi$ is transitive. In this case, one can in fact choose $\mathcal{C}$ to be the union of the singular orbits and one other closed orbit. \qed
\end{prop}

\begin{rmk} \label{rmk:perfectfitdefns}
There is an alternate characterization of no perfect fits that is increasingly favored in the literature:
A pseudo-Anosov flow $\phi$ has no perfect fits relative to a nonempty finite collection of closed orbits $\mathcal{C}$ if and only if the blown-up flow $\phi^\sharp$ does not admit two closed orbits $\gamma_1$ and $\gamma_2$ such that $\gamma_1$ is homotopic to $\gamma_2^{-1}$ in $Y^\sharp$.
This fact follows from the results in \cite{Fen16}, as explained in \cite[Remark 3.9]{LMT24b}.
\end{rmk}

\subsection{Veering branched surfaces}

Let $M$ be a compact oriented 3-manifold with torus boundary components.
A \textit{branched surface} in $M$ is a 2-complex $B \subset M$ where every point in $B$ has a neighborhood smoothly modeled on a point in \Cref{fig:bsdefn}.

\begin{figure}
    \centering
    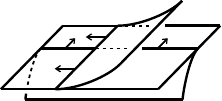
    \caption{The local model for a branched surface. The arrows indicate the maw coorientation.}
    \label{fig:bsdefn}
\end{figure}

The union of the non-manifold points of $B$ is called the \textit{branch locus} of $B$ and is denoted by $\brloc(B)$. There is a natural structure of a 4-valent graph on $\brloc(B)$. The vertices of $\brloc(B)$ are called the \textit{triple points} of $B$. 
The \textit{sectors} of $B$ are the complementary regions of $\brloc(B)$ in $B$.
Each edge of $\brloc(B)$ has a canonical coorientation on $B$, which we call the \textit{maw coorientation}, given by the direction from the side with more sectors to the side with less sectors, as indicated in \Cref{fig:bsdefn}.

A \textit{branch loop} of $B$ is a smoothly immersed closed curve in $\brloc(B)$. Each edge of $\brloc(B)$ lies in exactly one branch loop. In other words, $\brloc(B)$ admits a canonical decomposition as a union of branch loops.
Each sector of $B$ carries a natural structure of surfaces with corners. The corners are where the boundary switches from lying on one branch loop to another branch loop locally.

The complementary regions of $B$ in $M$ are 3-manifolds with boundary, along with the data of a disjoint collection of \textit{cusp curves}, i.e. closed curves along which the boundary fails to be smoothly immersed.
The cusp curves are in one-to-one correspondence with the branch loops of $B$.

\begin{defn} \label{defn:vbs}
A branched surface $B$ in $M$ is \textit{veering} if:
\begin{enumerate}
    \item Each sector of $B$ is homeomorphic to a disc.
    \item Each component of $M \backslash B$ is a \emph{cusped torus shell}, i.e. a thickened torus with a nonempty collection of cusp curves on one of its boundary components.
    \item There is a choice of orientation on each branch loop so that at each triple point, the orientation of each branch loop induces the maw coorientation on the other branch loop. See \Cref{fig:veeringcondition}.
\end{enumerate}
\end{defn}

\begin{figure}
    \centering
    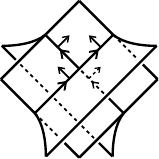
    \caption{A local picture around a triple point in a branched surface satisfying \Cref{defn:vbs}(3).}
    \label{fig:veeringcondition}
\end{figure}

The choice of orientations in (3) is unique if it exists. As such, given any veering branched surface $B$, we will implicitly adopt those orientations on the branch loops of $B$. In turn, the cusp curves of the complementary regions inherit natural orientations.

For each boundary component $T$ of $M$, there is a unique complementary region $C$ such that $T$ is a boundary component of $C$. The cusp curves on the other boundary component of $C$ determines an isotopy class of oriented curves on $T$. We refer to this isotopy class as the \textit{ladderpole class} on $T$.

The edges of $\brloc(B)$ also inherit orientations from the branch loops.
This upgrades the structure of $\brloc(B)$ as a 4-valent graph to a $(2,2)$-valent directed graph, i.e. a directed graph where every vertex has 2 incoming edges and 2 outgoing edges. 
It is customary to refer to $\brloc(B)$ with this upgraded structure as the \textit{dual graph} of $B$.
In this paper we denote it by $G$.

\subsection{Correspondence theorem}

We are ready to state the correspondence theorem between pseudo-Anosov flows and veering branched surfaces.

\begin{thm} \label{thm:vbspaflowcorr}
Let $\mathcal{F}$ denote the set of triples $(Y,\phi,\mathcal{C})$, where
\begin{itemize}
    \item $Y$ is a closed oriented 3-manifold,
    \item $\phi$ is a pseudo-Anosov flow on $Y$, and
    \item $\mathcal{C}$ is a nonempty finite collection of primitive closed orbits of $\phi$, containing the singular orbits, such that $\phi$ has no perfect fits relative to $\mathcal{C}$,
\end{itemize}
modulo orbit equivalence by homeomorphisms isotopic to identity.

Let $\mathcal{B}$ denote the set of triples $(M,B,s)$, where 
\begin{itemize}
    \item $M$ is a compact oriented 3-manifold with torus boundary components,
    \item $B$ is a veering branched surface on $M$, and
    \item $s=(s_T)$ is a collection of slopes on each boundary component $T$ of $M$ such that $|\langle s_T,l_T \rangle| \geq 2$ for each $T$, where $l_T$ is the ladderpole class of $B$ on $T$,
\end{itemize}
modulo isotopy.

Then:
\begin{enumerate}[label=(\arabic*)]
    \item There exists a function $\mathsf{V}:\mathcal{F} \to \mathcal{B}$ of the form
    $$\mathsf{V}(Y, \phi, \mathcal{C}) = (Y \backslash \nu(\mathcal{C}), B(\phi, \mathcal{C}), \text{meridians}).$$
    \item There exists a function $\mathsf{P}:\mathcal{B} \to \mathcal{F}$ of the form 
    $$\mathsf{P}(M, B, s) = (M(s), \phi(B, s), \text{cores of filling solid tori}).$$
    \item $\mathsf{V}$ and $\mathsf{P}$ are inverse to each other.
\end{enumerate}

Furthermore, under these bijections, if $(Y, \phi, \mathcal{C})$ and $(M, B, s)$ correspond to each other, then:
\begin{enumerate}[label=(\roman*)]
    \item $B$ carries the blown-up unstable foliation of $\phi$ in $M$. In particular, the ladderpole class of $B$ equals the degeneracy class of $\phi$ on each boundary component of $M$.
    \item The collection of homotopy classes of loops of the dual graph $G$ of $B$ equals the collection of homotopy classes of closed orbits of the blown-up flow $\phi^\sharp$ in $M$.
\end{enumerate}
\end{thm}
\begin{proof}
We first explain how the theorem can be assembled from results in the literature: 
\cite[Proposition 2.15 and Proposition 3.2]{Tsa23} show that veering branched surfaces are the dual objects to \textit{veering triangulations}. 
These are a type of triangulations satisfying certain combinatorial conditions. Their exact definition can be found in \cite{FG13}, but it will not be of concern to us in this paper.
The important fact here is that the set $\mathcal{B}$ can be identified with the set $\mathcal{T}$ of triples $(M,\Delta,s)$, where 
\begin{itemize}
    \item $M$ is a compact oriented 3-manifold with torus boundary components,
    \item $\Delta$ is a veering triangulation on $M$, and
    \item $s=(s_T)$ is a collection of slopes on each boundary component $T$ of $M$ such that $|\langle s_T,l_T \rangle| \geq 2$ for each $T$, where $l_T$ is the ladderpole curve of $\Delta$ on $T$,
\end{itemize}
modulo isotopy.

With this understood, (1) and (i) follow from work of Agol-Guéritaud (see \cite{LMT23} or \cite{SS21}), (2) and (3) are shown by Schleimer-Segerman \cite{SS20}, \cite{SS21}, \cite{SS19}, \cite{SS23}, \cite{SSpart5}, and (ii) is shown by Landry-Minsky-Taylor \cite{LMT23}.

To aid the reader's intuition, we spend the remainder of this section giving an outline of a proof of (1), (i), and (ii), which are the items most relevant to this paper. 
More in-depth explanations of these ideas, as well as those appearing in the proofs of (2), (3), and more can be found in \cite[Chapter 2]{Tsathesis}.

Let $\phi$ be a pseudo-Anosov flow. Let $\mathcal{O}$ be the orbit space of $\phi$.
We define a \textit{maximal rectangle} to be a rectangle $R$ embedded in $\mathcal{O}$ where the restrictions of $\mathcal{O}^{s/u}$ foliate $R$ as a product, such that there are no elements of $\widetilde{\mathcal{C}}$ in the interior of $R$, but there is an element of $\widetilde{\mathcal{C}}$ in the interior of each edge of $R$.
The fact that $\phi$ has no perfect fits relative to $\mathcal{C}$ implies every point of $\mathcal{O}$ lies in the interior of some maximal rectangle.

We associate to each maximal rectangle $R$ a tetrahedron $t_R$. It is helpful to think of $t_R$ as lying over $R$, with the 4 vertices of $t_R$ lying over the 4 elements of $\widetilde{\mathcal{C}}$ that meet $R$. See \Cref{fig:maxrecttoveertet} first two rows, where the red/blue colors on the edges of $t_R$ can be ignored for the purposes of this paper.
We build a triangulation $\overline{\Delta}$ out of these tetrahedra by gluing $t_{R_1}$ and $t_{R_2}$ along a face whenever $R_1$ and $R_2$ meet 3 common elements of $\widetilde{\mathcal{C}}$, the face that is glued being the one with vertices along the 3 common elements. See \Cref{fig:maxrecttoveertet} bottom row.

\begin{figure}
    \centering
    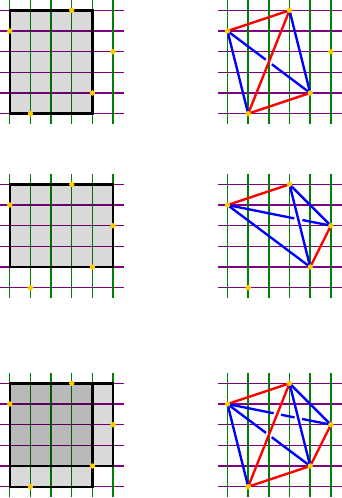
    \caption{Building a triangulation $\overline{\Delta}$ out of tetrahedra $t_R$ corresponding to maximal rectangles $R$.}
    \label{fig:maxrecttoveertet}
\end{figure}

More concretely, we can perform the construction of $\overline{\Delta}$ within $\widetilde{Y}$.
However, note that $\overline{\Delta}$ is not a triangulation of $\widetilde{Y}$, since each flow line in $\widetilde{\mathcal{C}}$ is `collapsed' down to a point in $\overline{\Delta}$. Instead, we can remove a neighborhood of each vertex to get an \textit{ideal triangulation} $\widetilde{\Delta}$ of $\widetilde{Y} \backslash \nu(\widetilde{\mathcal{C}})$.
Landry-Minsky-Taylor \cite{LMT23} show that one can further arrange for $\widetilde{\Delta}$ to be $\pi_1 Y$-invariant, and so that the faces of $\widetilde{\Delta}$ are transverse to the flow lines of $\widetilde{\phi}$, as suggested by the figures. 
These facts implies that one can quotient out the action of $\pi_1 Y$ to get an ideal triangulation $\Delta$ of $Y \backslash \nu(\mathcal{C})$, where the faces of $\Delta$ are transverse to the flow lines of $\phi$. 
(This is the veering triangulation referred to above.)

The latter fact also implies that the unstable foliation $\widetilde{\Fo^u}$ intersects each tetrahedron transversely. 
In fact, one can check that the restriction of $\widetilde{\Fo^u}$ to each tetrahedron, once blown up, consists of a band of quadrilaterals and two bands of triangles. See \Cref{fig:veertettovbs} top right.
Thus this blown-up foliation is carried by the branched surface given by the corresponding union of one quadrilateral and two triangles. See \Cref{fig:veertettovbs} bottom right.

\begin{figure}
    \centering
    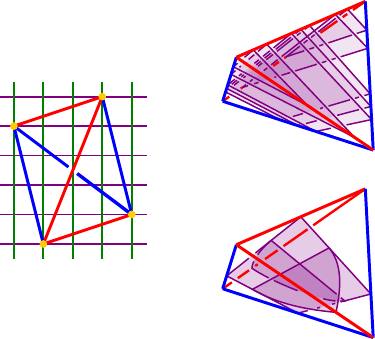
    \caption{Approximating $\widetilde{\Fo^u}$ within each tetrahedron by a branched surface.}
    \label{fig:veertettovbs}
\end{figure}

One can further show that the branched surfaces in each tetrahedron piece together to give a $\pi_1 Y$-invariant branched surface in $\widetilde{Y} \backslash \nu(\widetilde{\mathcal{C}})$. Taking the quotient over $\pi_1 Y$, we get a veering branched surface $B$ on $Y \backslash \nu(\mathcal{C})$ carrying the blown-up unstable foliation of $\phi$. This shows (1) and (i).

To show (ii), observe that $B$ is the dual 2-complex to $\Delta$. 
Given a closed orbit $\gamma$ of the blown-up flow $\phi^\sharp$, since the faces of $\Delta$ are transverse to the flow lines of $\phi$, thus to that of $\phi^\sharp$, $\gamma$ passes through a sequence of faces of $\Delta$. This determines a loop $c$ in the dual graph $G$ that is homotopic to $\gamma$. 

Conversely, given a loop $c$ in $\Gamma$, consider a lift $\widetilde{c}$ of $c$ in $\widetilde{Y}$. The sequence of vertices passed through by $\widetilde{c}$ corresponds to a sequence of tetrahedra of $\widetilde{\Delta}$, which in turn corresponds to a sequence of maximal rectangles in the orbit space $\mathcal{O}$. In the forward direction, the rectangles get longer in the $\mathcal{O}^s$ direction and shorter in the $\mathcal{O}^u$ direction, while in the backward direction, the rectangles get shorter in the $\mathcal{O}^s$ direction and longer in the $\mathcal{O}^u$ direction. This property implies that the intersection of all rectangles is a point in the orbit space. Such a point is a flow line $\widetilde{\gamma}$ of $\widetilde{\phi}$. One can show that $\widetilde{\gamma}$ is invariant under the deck transformation $[c] \in \pi_1 Y$, thus descends to a closed orbit $\gamma$ of $\phi$ homotopic to $c$ in $Y$. 

Running this argument in the universal cover of $\mathcal{O} \backslash \widetilde{\mathcal{C}}$, one can keep track of homotopy classes in $Y \backslash \nu(\mathcal{C})$ instead of $Y$. This way one can arrange for $\gamma$ to be a closed orbit of $\phi^\sharp$ homotopic to $c$ in $Y \backslash \nu(\mathcal{C})$. We refer to \cite{LMT23} for details.
\end{proof}

\section{Sutured Floer homology} \label{sec:suturedfloer}

In this section, we recall some background on sutured Floer homology. 
This theory was developed by Juhász \cite{Juh06}, building on ideas explored by Ozsváth and Szabó \cite{OS04c}. 
A good general reference for this material is the upcoming book by Ozsváth, Stipsicz, and Szabó \cite{OSS} (with the first few chapters, covering all the material needed for this paper, being available online at the time of writing). 

For the purposes of this paper, a \emph{sutured manifold} is a pair $(M,\Gamma)$ where $M$ is a compact oriented 3-manifold  whose boundary is decomposed into two subsurfaces $R_+$ and $R_-$ meeting along a collection of homotopically essential closed curves $\Gamma$, called the \emph{sutures}.

\begin{defn}[Balanced sutured manifold]
A sutured manifold $(M, \Gamma)$ is called \emph{balanced} if $R_+$ and $R_-$ have the same Euler characteristics: $\chi(R_+)=\chi(R_-)$, and each connected component of $\partial M$ contains at least one suture.
\end{defn}

The two main sources of sutured manifolds we shall consider in this paper are summarized below.

\begin{eg}[Neighbourhoods of branched surfaces] \label{eg:bstosut}
Suppose $B$ is a branched surface in a closed oriented 3-manifold $Y$. Then a tubular neighbourood $M=\mathcal{N}(B)$ of $B$ is a oriented 3-manifold with boundary. As shown in  \Cref{fig:bstosut},  $\partial M$ decomposes in two parts: the horizontal boundary $\partial_h M$, and the vertical boundary $\partial_v M$. 
We can put a sutured manifold structure on $M$ by declaring $R_+ = \partial_v M$ and $R_- = \partial_h M$.
\end{eg}

\begin{figure}
    \centering
    \selectfont\fontsize{8pt}{8pt}
    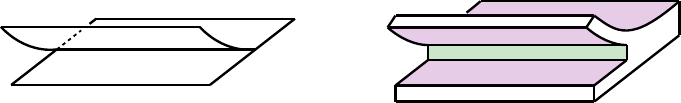
    \caption{The sutured manifold structure on a tubular neighborhood of a branched surface.}
    \label{fig:bstosut}
\end{figure}

\begin{eg}[Blown-up pseudo-Anosov flows] \label{eg:paflowtosut}
Suppose $(Y,\phi)$ is a pseudo-Anosov flow, and that  $(Y^\sharp, \phi^\sharp)$ is the result of blowing up $(Y, \phi)$  along a collection of primitive closed orbits $\mathcal{C}$. Then $Y^\sharp$ admits a sutured manifold structure where the sutures are the primitive closed orbits of the blown-up flow $\phi^\sharp$ on $\partial Y^\sharp$. 
\end{eg}

If a veering branched surface $B$ corresponds to a pseudo-Anosov flow $\phi$  and a collection of closed orbits $\mathcal{C}$ as in \Cref{thm:vbspaflowcorr}, then since $B$ carries the blown-up unstable foliation of $\phi$, the sutured manifold structure on $Y^\sharp$ induced by $B$ agrees with the one induced by $\phi^\sharp$.

\begin{rmk}
Note that these sutured structures are balanced since $R_+$ and $R_-$ are  union of annuli.   
\end{rmk}

\subsection{Heegaard diagrams} \label{subsec:heegaarddiagram}

We can describe sutured manifolds combinatorially using some finite two-dimensional diagrams.

\begin{defn}[Sutured Heegaard diagram]
A \emph{sutured Heegaard diagram} is a triple $(\Sigma, \boa, \bob)$ consisting of a compact, oriented surface with boundary $\Sigma$, and  two multi-curves $\boa=\{\alpha_1, \dots ,\alpha_n\}$, and $\bob=\{\beta_1, \dots ,\beta_m\}$ contained in the interior of $\Sigma$.
\end{defn}

Note that in a sutured Heegaard diagram $(\Sigma, \boa, \bob)$ we can pinch the boundary components of the surface $\Sigma$ to turn it into a closed surface $\overline{\Sigma}$ with a collection of marked points $\z=\{z_1, \dots, z_\ell\}$ contained in the complement of the $\alpha$- and the $\beta$-curves. The quadruple $(\overline{\Sigma}, \boa, \bob, \z)$ is called a \emph{multi-pointed Heegaard diagram}. 
Conversely, given a multi-pointed Heegaard diagram $(\overline{\Sigma}, \boa, \bob, \z)$ we can remove a small disk centered at each basepoint $z_i$ and get a sutured Heegaard diagram with as many boundary components as basepoints.

Starting from a sutured Heegaard diagram $(\Sigma, \boa, \bob)$ with $n$ $\alpha$-curves and $m$ $\beta$-curves one can construct a sutured manifold $(M, \Gamma)$ as follows:
\begin{enumerate}
    \item Start with a thickening $\Sigma \times [-1,1]$ of the surface $\Sigma$.
    \item Attach a three-dimensional $2$-handle with core $\alpha_i \times \{-1\}$ for $i=1, \dots , n$. 
    \item Attach a three-dimensional $2$-handle with core $\beta_i \times \{-1\}$  for $i=1, \dots , m$ and call the resulting manifold $M$.
    \item Finally, take the sutures to be $\Gamma= \partial \Sigma \times \{0\}$, and take $R_\pm$ to be the component of $\Sigma \backslash \Gamma$ containing $\partial \Sigma \times \{\pm 1\}$.
\end{enumerate}

In \cite[Proposition 2.9]{Juh06} it is shown that a sutured Heegaard diagram $(\Sigma, \boa, \bob)$ defines a balanced sutured manifold if and only if the following conditions are met:
\begin{itemize}
    \item $|\boa|=|\bob|$, that is, there are as many $\alpha$-curves as many $\beta$-curves,
    \item every connected component of 
    $\Sigma\  \setminus \  \bigcup_{i=1}^d \alpha_i$ contains at least one component of $\partial \Sigma$, and every connected component of 
    $\Sigma\  \setminus \  \bigcup_{i=1}^d \beta_i$ contains at least one component of $\partial \Sigma$.
\end{itemize}
Sutured diagrams satisfying these conditions are called \emph{balanced}. 
Conversely, it is shown in \cite[Proposition 2.13]{Juh06} that every balanced sutured manifold can be described by a balanced Heegaard diagram.

\subsection{Floer chain complex} \label{subsec:floerchaincomplex}

Suppose that $(\Sigma, \boa, \bob)$ is a balanced sutured Heegaard diagram with $d$ $\alpha$-curves and $d$ $\beta$-curves.
Pinch the boundary components of the Heegaard surface to turn it into a multi-pointed Heegaard diagram $(\overline{\Sigma},\boa,\bob,\z)$ with basepoints $\z=\{z_1, \dots , z_k\}$. 
Furthermore fix a volume form $\nu$ on $\Sigma$ and a compatible complex structure $j$.

The \emph{$d$-fold symmetric product} $\text{Sym}^d(\overline{\Sigma})$ of the surface $\overline{\Sigma}$ is defined as the space of degree $d$ effective divisors on $\overline{\Sigma}$, that is the space of all formal linear combinations $\x=x_1+\dots +x_d$ with $x_i \in \overline{\Sigma}$. One can identify $\text{Sym}^d(\overline{\Sigma})$ with the  quotient space $\overline{\Sigma}^{\times d}/S_d$ where $\overline{\Sigma}^{\times d}$ denotes the $d$-fold cartesian product $\overline{\Sigma} \times \dots \times \overline{\Sigma}$, and $S_d$ the symmetric group on $d$ letters acting on $\overline{\Sigma}^{\times d}$ by permutation of the coordinates:
\[\sigma \cdot (x_1, \dots, x_d) = (x_{\sigma (1)}, \dots, x_{\sigma (d)}) \ .\]
Note that despite the action of the symmetric group not being free, $\text{Sym}^d(\overline{\Sigma})$ has the structure of a smooth complex manifold. This is because $\text{Sym}^d(\overline{\Sigma})$ is locally biholomorphic to $\text{Sym}^d(\mathbb{C})$, and $\text{Sym}^d(\mathbb{C})\cong \C^d$ via the fundamental theorem of algebra.

The cartesian product $\overline{\Sigma}^{\times d}$ is a symplectic manifold with symplectic form $\widetilde{\omega}=\nu^{\times d}$. To define a symplectic form on $\text{Sym}^d(\overline{\Sigma})$  note that the projection $\pi:  \overline{\Sigma}^{\times d} \to \text{Sym}^d(\overline{\Sigma})$ is a $d!$-fold branched covering with branching set at the \emph{fat diagonal}
\[\Delta = \{ x_1+\dots +x_d \in \text{Sym}^d(\overline{\Sigma}) : x_i=x_j \text{ for some } i \not=j  \}  \ .\]
In general, when there is a branched covering $\widetilde{X} \to X$ and a symplectic form $\widetilde{\omega}$ on the total space $\widetilde{X}$, Perutz \cite{Per08} showed that one  can find a symplectic form $\omega$ on the base  manifold so that the following conditions are met:
\begin{itemize}
    \item $\pi_*\widetilde{\omega}=\omega$ in the complement of a neighbourhood of the branching set $\Delta$, 
    \item $\pi_*[\widetilde{\omega}]=[\omega]$ in the de Rham cohomology $H^2(\text{Sym}^d(\overline{\Sigma}), \R)$.
\end{itemize}
We equip $\sym^d (\overline{\Sigma})$ with such a symplectic form $\omega$.

In the symmetric product $(\text{Sym}^d(\overline{\Sigma}), \omega)$ the projections of 
\[\alpha_1 \times \dots \times \alpha_d \subset \overline{\Sigma}^{\times d} \ \ \ \ \text{ and } \ \ \ \ \ \beta_1 \times \dots \times \beta_d \subset \overline{\Sigma}^{\times d} \ \] 
are two \textit{Lagrangian submanifolds}, i.e. two half-dimensional submanifolds $\mathbb{T}_\alpha$ and $\mathbb{T}_\beta$ with the property that $\omega|_{\mathbb{T}_\alpha} =\omega|_{\mathbb{T}_\beta} \equiv 0$. Note that $\mathbb{T}_\alpha \pitchfork \mathbb{T}_\beta$ provided that the $\alpha$- and the $\beta$-curves intersect transversely in the underlying Heegaard diagram $(\Sigma, \boa, \bob)$.  

Similarly, the basepoints $z_1, \dots, z_k\in \overline{\Sigma}$ give rise to codimension-two submanifolds of  $\text{Sym}^d(\overline{\Sigma})$. Indeed, one can take $V_{z_i} = \{z_i\} \times \text{Sym}^{d-1}(\overline{\Sigma})$ as the subspace of all degree $d$ effective divisors $\x= x_1+\dots + x_d$ such that $x_j=z_i$ for some index $j=1, \dots , d$.

Summarizing, we have a symplectic manifold  $(W,\omega)=\left( \text{Sym}^d(\overline{\Sigma})  ,  \omega \right)$
and two compact Lagrangian submanifolds $L_1=\mathbb{T}_\alpha$ and $L_2=\mathbb{T}_\beta$.
In this set-up Floer \cite{Flo88} fixes a generic path $J_s$ of compatible almost-complex structures, and studies, for any two intersection points $\x$ and $\y\in L_1 \cap L_2$ of the two Lagrangians, all maps $u: \R \times [-1, 1]\to W$ solving the (perturbed)  Cauchy-Riemann equation
  \[J_s\  \frac{\partial u}{\partial t} - \frac{\partial u}{\partial s}=0 \]
subject to the boundary conditions:
\begin{enumerate}
    \item\label{item:disk1} $u(\R\times \{-1\}) \subset L_1$ and $u(\R\times \{1\}) \subset L_2$
    \item\label{item:disk2} $\lim_{t\to +\infty} u(t, s)= \x $ and $\lim_{t\to -\infty} u(t, s)= \y $
\end{enumerate}
and the \emph{finite energy} condition:
\[E(u)=\int u^*\omega = \int \int \left|\frac{\partial u}{\partial s} \right| ds\  dt < + \infty.\]
Maps satisfying conditions \eqref{item:disk1} and \eqref{item:disk2} are called \emph{Whitney disks}. These form a subspace $B(\x,\y) \subset C^\infty(\R \times [-1, 1], W)$.  We shall denote by $\pi_2(\x, \y)= \pi_0(B(\x,\y))$ the set of its connected components, and by $\mathcal{M}(\phi)$ the set of pseudo-holomorphic representatives in a given homotopy class $\phi \in \pi_2(\x, \y)$ of strips.

To each homotopy class $\phi \in \pi_2(\x, \y)$ one can associate an integer $\mu(\phi)\in \Z$ called the \emph{Maslov index} \cite{Vit87}.
This can be interpreted as the expected dimension of the moduli space  $\mathcal{M}(\phi)$ of pseudo-holomorphic representatives.

\begin{thm}[Floer]
For a generic choice of the path of almost complex structures, $\mathcal{M}(\phi)$ is a smooth finite-dimensional manifold of dimension $\mu(\phi)$. \qed
\end{thm}

Inspired by this work of Floer, Ozsv\' ath and Szab\' o \cite{OS04c} studied the chain complex $C_*(\mathbb{T}_{\boa}, \mathbb{T}_{\bob}; J_s)$ freely generated over $\mathbb{F} = \mathbb{F}_2$ by the intersection points of the Lagrangian tori $\mathbb{T}_{\boa}$ and $\mathbb{T}_{\bob}$ produced by an Heegaard diagram
\[C_*(\mathbb{T}_{\boa}, \mathbb{T}_{\bob}; J_s)= \bigoplus_{\x\in \mathbb{T}_{\boa} \cap \mathbb{T}_{\bob}} \mathbb{F} \cdot  \x \ , \]
equipped with differential
\begin{equation} \label{eq:differential}
    \partial \x = \sum_{\y \in \mathbb{T}_{\boa} \cap \mathbb{T}_{\bob}} \ \ \sum_{\left\{\phi \in \pi_2(\x, \y) \mid \ \#|\phi\cap V_{z_i}|=0,\ \mu(\phi)=1\right\}} \ \# \left(\frac{\mathcal{M}(\phi)}{\R} \right) \cdot \y. 
\end{equation} 
They showed that, provided that one uses a special type of Heegaard diagrams called \emph{admissible}, the following holds. 

\begin{thm}[Ozsv\' ath-Szab\' o] \label{thm:cfwelldefined}
For an admissible diagram the differential $\partial$ displayed in \Cref{eq:differential} is well-defined, satisfies $\partial^2=0$, and the homology of the chain complex $C_*(\T_{\boa}, \T_{\bob}; J_s)$ is  independent  from the chosen path of almost-complex structures, and invariant under Hamiltonian isotopies of $\T_{\boa}$ and $\T_{\bob}$. \qed
\end{thm}

The definition of admissibility will be presented in \Cref{subsec:admissibility} below.
Returning to our sutured Heegaard diagram $(\Sigma,\boa,\bob)$ and assuming admissibility, we define
\[CF(\Sigma, \boa, \bob)=C_*(\T_{\boa}, \T_{\bob}; J_s) \ ,\] 
where $J_s$ is some generic path of almost-complex structures  on the $d$-fold symmetric product $\sym^d(\overline{\Sigma})$. 

\begin{thm}[Juhász] \label{thm:sfhindependence}
The homology of the chain complex $CF(\Sigma, \boa, \bob)$ only depends on the topology of the associated  sutured manifold. \qed
\end{thm} 

The homology of $CF(\Sigma, \boa, \bob)$ is called \textit{sutured Floer homology} and is typically denoted by $SFH(M, \Gamma)$ to stress the dependence of the homology on the topology of the sutured manifold $(M,\Gamma)$ associated to $(\Sigma, \boa, \bob)$, and hence the independence from the particular combinatorics of the diagram $(\Sigma, \boa, \bob)$.

\subsection{Generators} \label{subsec:generators}

The chain complex $CF(\Sigma, \boa, \bob)$ is freely generated by the intersection points of $\T_{\boa}$ and $\T_{\bob}$ inside $\sym^d(\Sigma)$, a high dimensional symplectic manifold.
Note that a point $ x_1+\dots + x_d\in \sym^d(\Sigma)$ is in the intersection $\T_{\boa}\cap\T_{\bob}$ iff there exists a permutation $\sigma\in S_d$ such that $x_i \in \alpha_i \cap \beta_{\sigma(i)}$ for $i=1, \dots, d$.
The observation implies that the generators of $CF(\Sigma, \boa, \bob)$ are in one-to-one correspondence with the following objects:

\begin{defn}[Heegaard states]
A \emph{Heegaard state} is a subset of $d$ distinct points $\x = \{x_1,\dots,x_d\} \subset \Sigma$ for which there exists a permutation $\sigma\in \mathfrak{S}_d$ such that $x_i \in \alpha_i \cap \beta_{\sigma(i)}$ for $i=1, \dots, d$. 
We denote by $\mathfrak{S}(\Sigma, \boa, \bob)$ the set of all states.
\end{defn}

For the rest of this paper, we will think of $CF(\Sigma, \boa, \bob)$ as being freely generated by the Heegaard states of $(\Sigma, \alpha, \beta)$.

\subsection{Domains} \label{subsec:domains}
On the other hand, understanding the differential of $CF(\Sigma, \boa, \bob)$ is an hard task. 
For this purpose Ozsv\' ath and Szab\' o developed the combinatorial theory of domains in \cite{OS04c} we presently review.

In what follows, $D_1, \dots, D_r$ denote the closure of the connected components of
 \[\overline{\Sigma} \setminus\left( \bigcup_{i=1}^d \alpha_i \cup \bigcup_{i=1}^d \beta_i  \right) \ .\]
We call $D_1, \dots, D_r$ the \emph{elementary domains} of the diagram.

\begin{defn}[Domains]
A \emph{domain} is a formal linear combination $D=\sum_i n_i D_i$  with integer coefficients $n_i \in \Z$. The coefficient $n_i$ is called the \emph{multiplicity} of the domain $D$ at $D_i$. 
Equivalently, given a point $z\in D_i$ we define the \emph{local multiplicity} of $D$ at $z$ as the multiplicity of $D$ at $D_i$. This is usually denoted by $n_z(D)$.
\end{defn}

The set of domains form a finitely generated free $\Z$-module of rank $r$, where $r$ is the number of elementary domains.

Associated to a homotopy class $\phi \in \pi_2(\x, \y)$ there is a domain
\[D(\phi)= \sum_{i=1}^r n_i(\phi)  D_i \ , \]
where $n_i(\phi)$ is defined as the intersection number $\#|\phi \cap (\{w\} \times \sym^{d-1}(\Sigma)) |$, 
for some $w \in D_i$. 
Here the intersection number  $n_i(\phi)$ does not depend on $w$, since different choices of $w$ in the same elementary domain do not change the isotopy class of $ \{w\} \times \sym^{d-1}(\Sigma)$ relative to $\T_\alpha \cup \T_\beta$.

Note that in the differential of $CF(\Sigma, \alpha, \beta)$, displayed in \Cref{eq:differential}, the condition that $\#|\phi \cap V_{z_i}| = 0$ for all the basepoints $z_i$ is equivalent to the condition that the domain $D(\phi)$ has local multiplicity $n_{z_i}(D(\phi))=0$ at all basepoints. 
Motivated by this, for the rest of this paper, we will write `$n_\z(D) = 0$' as a shorthand to mean that in the domain $D$ the coefficients $n_i$ is zero whenever $D_i$ contains a basepoint.

\begin{defn}[Connecting domains, initial/final state] \label{defn:initialfinalstate}
Let $\x=x_1+ \dots + x_d$ and $\y=y_1+ \dots + y_d$ be two Heegaard states, and $D=\sum_i n_i D_i$ a domain. 
Thinking of $D$ as a $2$-chain in $C_2(\overline{\Sigma})$, $\partial D$ is a $1$-cycle in the graph $\boa \cup \bob$. 
Since $C_1(\boa \cup \bob)= C_1(\boa) \oplus C_1(\bob)$, the differential $\partial D$ naturally decomposes as the sum of two $1$-chains $\partial D=\partial_\alpha D + \partial_\beta D$, where $\partial_\alpha D= \partial D \cap \boa$ and $\partial_\beta D= \partial D \cap \bob$. 

We say that the domain \emph{$D$ connects $\x$ to $\y$} if $\partial_\alpha D=\y -\x$ and  $\partial_\beta D= \x -\y$. If $D$ connects $\x$ to $\y$ we call $\x$ the \emph{initial state} of $D$ and $\y$ its \emph{final state}.
\end{defn}

See \Cref{fig:indexeg} for some examples of domains connecting a state $\x$ to a state $\y$.

\begin{figure}
    \centering
    \selectfont\fontsize{6pt}{6pt}
    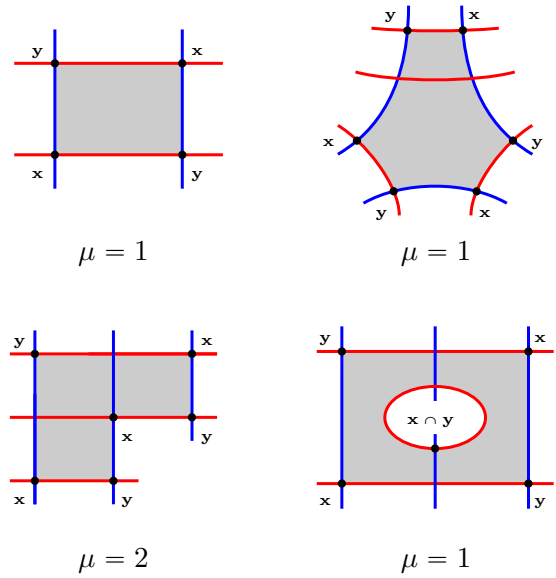
    \caption{Some examples of domains that connect a state $\x$ to a state $\y$. Here the coefficient of each shaded effective domain is $1$. The Maslov indices of these domain can be computed using Equation \eqref{eq:lipshitzindex}.}
    \label{fig:indexeg}
\end{figure}

\begin{lemma}\label{equations}
Suppose $D$ is a domain connecting a state $\x$ to a state $\y$. Let $p\in \alpha \cap \beta$ be an intersection point between an $\alpha$- and a $\beta$-curve.
Then we have that
\begin{equation} \label{eq:initialfinallocal}
n_i-n_j+n_k-n_l = 
\begin{cases}
1 & \text{if $p \in \x \backslash \y$} \\
-1 & \text{if $p \in \y \backslash \x$} \\
0 & \text{otherwise} \\
\end{cases}
\end{equation}
where $n_i$, $n_j$, $n_k$, $n_l$ denote the local multiplicities of the four elementary domains $D_i$, $D_j$, $D_k$, $D_l$ meeting at $p$ ordered counterclockwise as shown in \Cref{fig:initialfinallocal}. Vice versa, any domain $D$ satisfying these equation at each intersection point $p\in \alpha \cap \beta$ connects $\x$ to $\y$.
\end{lemma}
\begin{proof}
The lemma follows from comparing the coefficient of $p$ in $\partial \partial_\alpha D = \y - \x$, together the fact that the coefficient is $-1$ for $D=D_i,D_k$ and $1$ for $D=D_j,D_l$.
\end{proof}

\begin{figure}
    \centering
    \selectfont\fontsize{10pt}{10pt}
    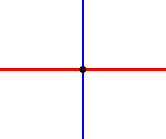
    \caption{A pictorial key for \Cref{eq:initialfinallocal}.}
    \label{fig:initialfinallocal}
\end{figure}

In \cite[Proposition 2.15]{OS04c} it is shown that the map $\phi \mapsto D(\phi)$ describes a bijection between $\pi_2(\x, \y)$ and the set of domains connecting $\x$ to $\y$.
So we shall make no distinction between Whitney disks and domains.

There is still the question of whether a homotopy class of Whitney disks contains no pseudo-holomorphic representative. The criterion of effectiveness gives a simple first obstruction.

\begin{defn}[Effective domain]
A domain $D=\sum_{i=1}^r n_i D_i$ is \emph{effective} if $n_i\geq 0$ for $i=1, \dots, r$. 
In this case we write $D\geq0$.
\end{defn}

It follows from positivity of intersections that only effective domains can have holomorphic representatives:  

\begin{lemma} \label{lemma:holdisctodomain}
If $\mathcal{M}(\phi)$ is non-empty, then $D(\phi)\geq 0$. \qed
\end{lemma}

Finally, there is the concern of checking $\mu(\phi)=1$ in order to determine whether $\phi$ contributes to the differential.
This task can be accomplished using the index formula of Lipshitz. 
Before stating the formula in \Cref{thm:lipshitzindex}, let us set up some definitions.

Let $S$ be a surface with corners. The \textit{Euler measure} of $S$ is defined to be 
$$e(S) = \chi(S) - \frac{1}{4} \text{\# corners}$$
where $\chi$ is the Euler characteristic of the underlying surface with boundary (i.e. forgetting the data of the corners).
In particular each elementary domain $D_i$, having a natural structure as a surface with corners, has an Euler measure $e(D_i)$.
The \textit{Euler measure} of a domain $D = \sum_i n_i D_i$ is defined to be $e(D) = \sum_i n_i e(D_i)$.

Suppose $x$ is a point in $\alpha \cap \beta$. Suppose $D_i, D_j, D_k, D_l$ are the elementary domains with a corner at $x$. The \textit{average multiplicity} of a domain $D = \sum_i n_i D_i$ at $x$ is defined to be $n_x(D) = \frac{1}{4}(n_i+n_j+n_k+n_l)$.
The \textit{average multiplicity} of a domain $D$ at a state $\x = \{x_1,\dots,x_d\}$ is defined to be $n_\x(D) = \sum_{i=1}^d n_{x_i}(D)$.

\begin{thm}[Lipshitz {\cite[Proposition 4.8]{Lip06}}] \label{thm:lipshitzindex}
Let $\x$ and $\y$ be two Heegaard states. Suppose $\phi \in \pi_2(\x,\y)$. Then 
\begin{equation} \label{eq:lipshitzindex}
\mu(\phi) = e(D(\phi))+n_\x(D(\phi))+n_\y(D(\phi)).
\end{equation} \qed
\end{thm}

See \Cref{fig:indexeg} for some examples.

Motivated by \Cref{thm:lipshitzindex}, we will define the \textit{index} of a domain $D$ connecting a state $\x$ to another state $\y$ to be $\mu(D) = e(D)+n_\x(D)+n_\y(D)$.
Note that $\mu(D)$ is an abuse of notation since this quantity depends on $\x$ and $\y$ as well. 
However, in practice $\x$ and $\y$ will be clear from context.

\subsection{$\Spinc$-grading} \label{subsec:spincgrading}

To decide if two generators are connected by some domain, one can also use $\spinc$-structures.  

Let $(M,\Gamma)$ be a sutured manifold. Fix a regular neighborhood of the sutures $A \cong \Gamma \times I$. Fix a non-vanishing vector field $X_0$ defined in a neighbourhood of $\partial M$ pointing into $M$ along $R_-(\Gamma) \backslash A$, out of $M$ along $R_+(\Gamma) \backslash A$, and agreeing on $A$ with the gradient of the projection $\Gamma \times I \to I$.
Note that the space of such vector fields is contractible. 

A \emph{$\spinc$-structure} on $(M,\Gamma)$ is a non-vanishing global vector field $X$ agreeing with $X_0$ in a neighbourhood of the boundary, modulo the relation that two $\spinc$ structures $X$ and $X'$ are equivalent if there exists a ball $B \subset \intr(M)$ such that $X|_{M\setminus B}$ is homotopic to $X'|_{M\setminus B}$ rel $\partial M$.

We denote with $\Spinc(M,\Gamma)$ the set of all $\spinc$-structures on $(M,\Gamma)$. Some elementary obstruction theory shows that $\Spinc(M,\Gamma) \neq \varnothing$, and that there is a free and transitive action of $H^2(M, \partial M)$ on $\Spinc(M,\Gamma)$, that is:  $\Spinc(M,\Gamma)$ is an affine space over the group $H^2(M, \partial M) \cong H_1(M)$.  

Now consider a balanced Heegaard diagram $(\Sigma, \boa, \bob)$ for $(M,\Gamma)$ with $d$ many $\alpha$- and $\beta$-curves.
In \cite{Juh06} Juhász associates to each generator $\x\in \T_{\boa} \cap \T_{\bob}$ a $\spinc$ structure $\s(\x)$ as follows. 

Let $f:M \to [0, 3]$ be a Morse function \emph{compatible} with the given balanced diagram, that is, such that:
\begin{itemize}
   \item  $f$ agrees with the projection $\Gamma \times I \to I$ in a neigbourhood $A$ of the sutures, $f^{-1}(0)=R_-(\Gamma) \backslash A$, and  $f^{-1}(3)=R_+(\Gamma) \backslash A$ (in particular $\nabla f$ agrees with the vector field $X_0$ in a neighbourhood of $\partial M$)
    \item $f$ has no index-zero and no index-three critical points.
    \item $f$ has $d$ index-one critical points $P_1, \dots , P_d$, and  $d$ index-two critical points $Q_1, \dots , Q_d$. Furthermore, $f^{-1}(3/2)=\Sigma$,  $f(P_i)=1$, and $f(Q_i)=2$.
   \item The $\alpha$-curves are the points on $\Sigma$ flowing into the index-one critical points via $-\nabla f$, that is, $\alpha_i = W^u(P_i) \cap \Sigma$, for $i=1, \dots, d$, where $W^u(P_i)$ denotes the unstable manifold of the index-one critical points.
    \item The $\beta$-curves are the points on $\Sigma$ flowing into the index-two critical points via $\nabla f$, that is, $\beta_i = W^s(Q_i) \cap \Sigma$, for $i=1, \dots, d$, where $W^s(P_i)$ denotes the stable manifold of the index-two critical points.
\end{itemize}

\begin{figure}
    \centering
    \selectfont\fontsize{8pt}{8pt}
    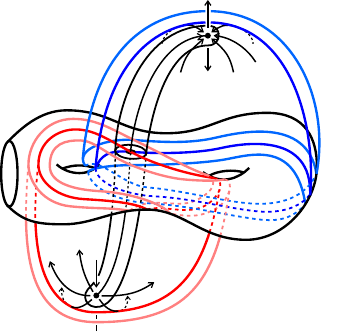
    \caption{Constructing the $\spinc$-structure $\mathfrak{s}(x)$ associated to a Heegaard state $\x$.}
    \label{fig:morsespindefn}
\end{figure}

Given an Heegaard state $\x$ of  $(\Sigma, \boa, \bob)$ we can find a  permutation  $\sigma \in \mathfrak{S}_d$ and intersection points $x_i \in \alpha_i \cap \beta_{\sigma(i)} $ such that $\x= x_1+ \dots + x_d$. Each $x_i$ gives rise to a trajectory $\gamma_{x_i}$ of the gradient flow $\nabla f$ such that
\[\lim_{t\to -\infty} \gamma_{x_i}(t)=P_i \ \ \  \text{ and } \ \ \  \lim_{t\to +\infty} \gamma_{x_i}(t)=Q_{\sigma(i)}  \ .\]
Choosing a collection of pairwise disjoint balls $B_1, \dots, B_d$ containing these trajectories we obtain a non-vanishing vector field $V_\x= \nabla f|_{M \setminus (B_1\cup  \dots \cup B_d)}$ agreeing with $X_0$ near $\partial M$. 
Since each ball $B_i$ contains a canceling pair of critical points, the vector field $V_\x$ can be extended to a non-vanishing vector field over $B_1\cup  \dots \cup B_d$. 
We define $\mathfrak{s}(\x)\in \Spinc(M, \Gamma)$ as the homotopy class of ${V}_\x$.
See \Cref{fig:morsespindefn} for a pictorial summary.

\begin{lemma}[{\cite[Corollary 4.8]{Juh06}}]
Suppose that $\x$ and $\y$ are two Heegaard states. 
If $\mathfrak{s}(\x) \neq \mathfrak{s}(\x)$ then $\pi_2(\x,\y)= \varnothing$. 
If $d>1$ the converse is also true: If $\mathfrak{s}(\x)= \mathfrak{s}(\x)$ then $\pi_2(\x,\y) \neq \varnothing$. \qed
\end{lemma}

The first statement implies that if two Heegaard states have different $\spinc$ structures then there is no differential between them.
Thus we can split $CF(\Sigma, \boa, \bob)$ into a direct sum of smaller chain complexes
\[CF(\Sigma, \boa, \bob) = \bigoplus_{\s \in  \Spinc(Q)} CF(\Sigma, \boa, \bob, \s) \ ,\]
where $CF(\Sigma, \boa, \bob, \s)$ denotes the sub-vector space generated by all states with $\spinc$-structure $\s$, that is:
\[CF(\Sigma, \boa, \bob, \s) = \bigoplus_{\s(\x)=\s } \mathbb{F} \cdot \x  \ . \]
This splitting has a conceptual meaning, since it says that  the homology splits in a direct sum along $\spinc$ structures 
$$SFH(M, \Gamma)=\bigoplus_{\s \in \Spinc(M, \Gamma)} SFH(M, \Gamma, \s) \ ,$$
and a practical meaning, since it allows us to apply a divide and conquer strategy in the computation of the Heegaard Floer homology differential.

\section{From veering branched surface to Heegaard diagram} \label{sec:vbstosfh}

In this section, we will present the main construction of the paper. 
This is a construction of a canonical sutured Heegaard diagram for a sutured manifold $(M, \Gamma)$ induced from a veering branched surface $B$ (see \Cref{eg:bstosut}).

\subsection{The Heegaard diagram} \label{subsec:vbstoheegaarddiagram}

Let $B$ be a veering branched surface on a 3-manifold with torus boundary $M$.
We shall denote by $v_1, \dots , v_n$ the triple points of $B$, and by $C_1, \dots, C_k$, the complementary regions of $B$. 
We denote by $G$ the dual graph of $B$. Recall that this is the 1-skeleton of $B$ with along with the data of orientations on the edges.
Finally, we denote by $S_1,\dots,S_m$ be the sectors of $B$.

The goal of this section is to describe a sutured Heegaard diagram  $(\Sigma, \boa, \bob)$ for  $M$ equipped with the sutured structure $\Gamma$ induced by $B$ on $\partial M$ as in \Cref{eg:paflowtosut}.

Let $U=\mathcal{N}(G)$ be the boundary of a tubular neighbourhood of $G$. 

\begin{lemma} \label{lemma:Sigma0genus}
$U$ is a genus $n+1$ handlebody.
\end{lemma} 
\begin{proof}
Recall that $G$ is a 4-valent graph.
If $e$ denotes the number of edges of $G$, then by the Handshake Lemma $2e = \sum_i \deg(v_i)= 4n$. 
Thus if $g$ denotes the genus of $U$, then $1-g = \dim H_0(U)-\dim H_1(U)= \chi(U)=\chi(G)=n-e= n-2n =-n$.
\end{proof}

Let $\Sigma_0=\partial U$ be the boundary of $U$. By \Cref{lemma:Sigma0genus}, $\Sigma_0$ is a surface of genus $n+1$. In $\Sigma_0$ there are two multi-curves $\boa=\{\alpha_1, \dots, \alpha_n\}$ and $\bob=\{\beta_1, \dots, \beta_m\}$: The $\alpha$-curves are placed in correspondence of the triple points of $B$ as suggested in \Cref{fig:vbstoheegaardtriplepoint}. 
(In particular $\alpha_i=\partial D_i$ for some disk $D_i\subset M$ with interior lying in $M\setminus \Sigma$.) 
The $\beta$-curves are defined as the boundary of the sectors of $B$. More precisely, $\beta_i=S_i\cap \Sigma_0$ for $i=1, \dots, m$. 

\begin{figure}
    \centering
    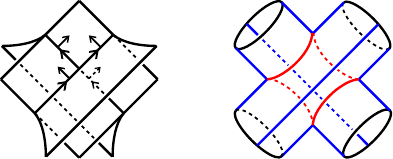
    \caption{The $\alpha$-curves are placed in correspondence of the triple points of $B$.}
    \label{fig:vbstoheegaardtriplepoint}
\end{figure}

Note that $(\Sigma_0, \boa, \bob)$ is \emph{not} an Heegaard diagram for $M$. Indeed, after compressing the $\alpha$-curves along the disks $D_i$, and the $\beta$-curves along the sectors we get a 3-manifold $M_0\subset M$ with complement $M \setminus M_0$ a collection of solid tori $\{T_1, \dots, T_\ell\}$, one for each cusp curve on $\partial M$. 

Note that the core $b_i$  of each solid torus $T_i\cong b_i\times D^2$ runs parallel to a branch loop, and it co-bounds an annulus $A_i \subset M$ with some cusp curve on $\partial M$. See the green annuli in \Cref{fig:vbstoheegaardcutannulus} middle. 

Let $M'_0\subset M_0 \subset M$ be the 3-manifold obtained from $M_0$ by cutting along this collection of annuli. In the manifold $M_0'$ the surface $\Sigma_0$ descends to an embedded surface $\Sigma \subset M_0'$ with boundary $\partial \Sigma \subset \partial M_0'$. 

\begin{figure}
    \centering
    \selectfont\fontsize{12pt}{12pt}
    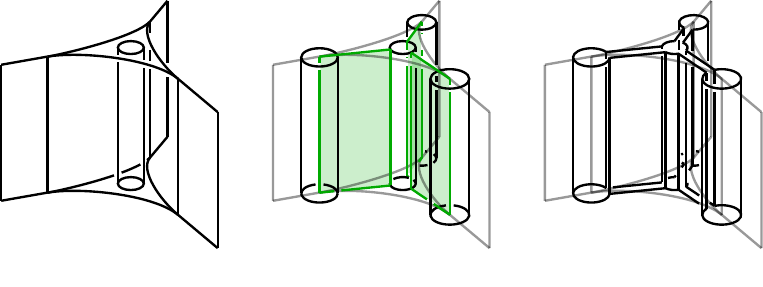
    \caption{The 3-manifold obtained from $(\Sigma_0, \boa, \bob)$ is the submanifold $M_0 \subset M$ with complement $M \setminus M_0$ a collection of solid tori $\{T_1, \dots, T_\ell\}$, one for each cusp curve on $\partial M$. In $M_0$, we have a collection of annuli $A_i$, each with one boundary component along the core of $T_i$ and the other boundary component along $\partial M$.}
   \label{fig:vbstoheegaardcutannulus}
\end{figure}

Furthermore, we observe that the annuli $A_i$ intersect $\Sigma_0$ away from the $\alpha$- and $\beta$-curves. 
See \Cref{fig:vbstoheegaardbranchloop}.
Thus we can consider the diagram $(\Sigma, \boa, \bob)$.

\begin{figure}
    \centering
    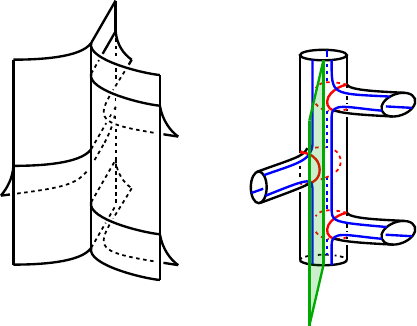
    \caption{The annuli $A_i$ intersect $\Sigma_0$ away from the $\alpha$- and the $\beta$-curves.}
    \label{fig:vbstoheegaardbranchloop}
\end{figure}

\begin{prop}  
$(\Sigma, \boa, \bob)$ is a sutured Heegaard diagram for $(M, \Gamma)$.
\end{prop}
\begin{proof}
By construction the 3-manifold obtained by compressing the $\alpha$- and $\beta$-curves in $(\Sigma, \boa, \bob)$ is $M'_0$.
Observe that $M'_0$ is naturally homeomorphic to $M$. Under this homeomorphism, there are two components of $\partial \Sigma_0$ for each cusp curve. Hence the sutured manifold structure on $M$ is $\Gamma$.
\end{proof}

\begin{prop} \label{prop:diagrambalanced} 
The diagram $(\Sigma, \boa, \bob)$ is balanced.
\end{prop}
\begin{proof}
The $\alpha$-curves are in one-to-one correspondence with the triple points $v_1, \dots,v_n$, and the $\beta$-curves with the sectors $S_1, \dots, S_m$. Since $M$ has torus boundary we have that $0=\chi(M)=\chi(B)= \chi(G)+m$. On the other hand, as shown in \Cref{lemma:Sigma0genus} we have that $\chi(G)=-n$. Thus $0=\chi(\Gamma)+m=m-n$. 
\end{proof}

\subsection{Heegaard states} \label{subsec:vbsstates}

Summarizing, we can canonically associate to each veering branched surface $B$ a balanced Heegaard diagram $(\Sigma, \boa, \bob)$.
The goal of this subsection is to obtain a description of the Heegaard states of this diagram in terms of $B$.
The entire discussion shall be based on the following proposition, and the consequential definition.

\begin{figure}
    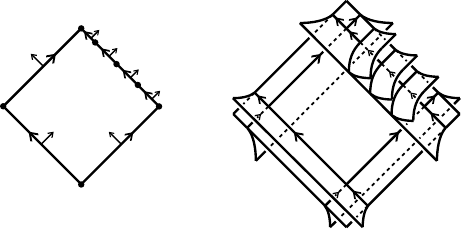
    \caption{Each sector of a veering branched surface is a diamond.}
    \label{fig:vbssector}
\end{figure}

\begin{lemma}[Diamonds] \label{prop:vbssector}
Each sector $S$ of a veering branched surface is a diamond, i.e. a disc with four corners, where the orientations on the sides all point from one corner to the opposite corner. 
\end{lemma} 

See \Cref{fig:vbssector}, where the black dots are the triple points, and we have indicated the orientation and coorientation of each edge in the boundary of $S$. 

\begin{proof}[Proof of \Cref{prop:vbssector}]
A complete proof can be found in \cite[Lemma 3.3]{Tsa23} (but note that the convention for orienting the branch loops is opposite in \cite{Tsa23}). Here we sketch the proof to aid the reader's intuition.

Let $S$ be a sector. As one goes along the sides of a boundary component of $S$, \Cref{defn:vbs}(3) implies that the maw coorientation flips between pointing into $S$ and pointing out of $S$ every two sides. Together with \Cref{defn:vbs}(1), we conclude that $S$ must be a disc with $4n$ corners for some $n \geq 0$.

If $n=0$ and the maw coorientation is pointing into $S$ everywhere, then the sector `around' $S$ contradicts \Cref{defn:vbs}(1). 
If $n=0$ and the maw coorientation is pointing out of $S$ everywhere, then we contradict \Cref{defn:vbs}(2).
Hence we must have $n \geq 1$.

Note that the Euler measure of $S$ is $1-n$. 
This is a nonpositive integer, and is zero if and only if $n=1$.

Fix a complementary region $C$ of $B$. The boundary of $C$ is a union of sectors that is homeomorphic to a torus.
Using the fact that the Euler measure is additive, the sum of the Euler measures of the sectors in $\partial C$ is the Euler characteristic of the torus, which is zero. Thus the index of each sector in $\partial C$ is zero, and they must all be diamonds.
\end{proof}

\begin{defn}[Corners] \label{defn:corners}
As suggested by \Cref{fig:vbssector}, we refer to the corner of a sector $S$ which all sides are oriented towards as the \emph{top corner} of $S$. We refer to the opposite corner as the \emph{bottom corner}.
We call the remaining two corners of $S$  \emph{side corners}. 
Finally we shall refer to the two sides of $S$ that meet the top corner as the \emph{top sides}, and the remaining two sides that meet the bottom corner as the \emph{bottom sides}. 
\end{defn}

Note that the top sides of each sector may consist of multiple edges of the branch locus, whereas the bottom sides must only consist of a single edge. 
We will describe this phenomenon in more detail in \Cref{prop:vbstogglefan}.

A consequence of \Cref{prop:vbssector} is that each $\beta$-curve $\beta_i$ has exactly four intersection points with the $\alpha$-curves. The four intersection points correspond to the four corners of the sector $S_i$.
See \Cref{fig:vbstoheegaardsector}.

\begin{figure}
    \centering
    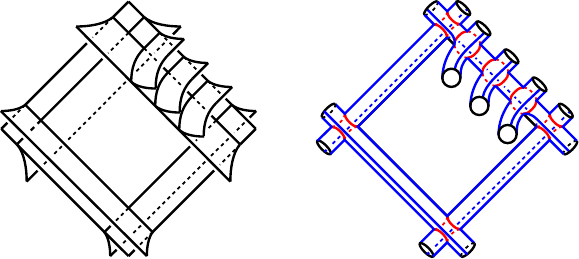
    \caption{Each $\beta$-curve $\beta_i$ has exactly four intersection points with the $\alpha$-curves, corresponding to the four corners of the sector $S_i$.}
    \label{fig:vbstoheegaardsector}
\end{figure}

We can thus characterize the Heegaard states as follows:

\begin{prop} \label{prop:generators}
Let $B$ a veering branched surface with sectors $S_1, \dots, S_n$ and triple points $v_1, \dots , v_n$.  The Heegaard states $\mathfrak{S}(\Sigma, \boa , \bob)$ of the Heegaard diagram $(\Sigma, \boa , \bob)$ associated to $B$ are in natural  correspondence with ways of assigning to each sector one of its four corners, so that each triple point is picked exactly once.
\end{prop}
\begin{proof}
Given such an assignment, we can pair the $\beta$-curve $\beta_i$ corresponding to $S_i$ to the $\alpha$-curve that corresponds to the corner that is assigned to $S_i$.
Conversely, given a state, we can assign a sector $S_i$ to the corner corresponding to the point on $\beta_i$ in the state.  
\end{proof}

\subsection{Top and bottom states} \label{subsec:tbstates}

As observed in \Cref{prop:vbssector} each sector of a veering branched surface has the shape of a \emph{diamond}.
In \Cref{defn:corners} we distinguished two particular corners which we called the top and bottom corners. 

\begin{defn}[Top/bottom states]
The \emph{top state} $\xt$ is the state defined by assigning to each sector its top corner.
The \emph{bottom state} $\xb$ is the stated defined by assigning to each sector its bottom corner.
\end{defn}

\begin{prop}
The assignments $\xt$ and $\xb$ define two Heegaard states.
\end{prop}
\begin{proof} We should verify that there is no triple point that is picked up twice by $\xt$. In \Cref{prop:diagrambalanced} we verified that there are as many triple points as many sectors. Thus, by pigeon hole principle, if a triple point was simultaneously the top  corner of two different sectors there would be some triple point that is not the top corner of any sector. On the other hand the local picture \Cref{fig:veeringcondition} shows that each triple point is the top corner of some sector, a contradiction. A similar argument applies to $\xb$.
\end{proof}

\subsection{Relation to closed orbits} \label{subsec:statestoorbits}

Suppose now that the sutured manifold $(M, \Gamma)$ is the blow-up $Y^\sharp$ of a pseudo-Anosov flow $(Y, \phi)$ along a collection of closed orbits $\mathcal{C}$ and that $B$ is associated to $(\phi, \mathcal{C})$ as in \Cref{thm:vbspaflowcorr}. In this section we consider the Heegaard diagram $(\Sigma, \boa, \bob)$ associated to $B$  in relation to the flow $\phi$. We explain how to assign to each Heegaard state $\x$ of $(\Sigma, \boa, \bob)$   a closed multi-orbit $\gamma_\x$ of the flow $\phi^\sharp$.

\begin{defn}[Augmented dual graph] \label{defn:augmenteddualgraph}
Let $B$ a veering branched surface. The \emph{augmented dual graph}  $G_+$ of $B$ is defined as the directed graph obtained from the $1$-skeleton $G$ by adding an extra edge connecting the bottom to the top vertex of each sector.
\end{defn}

Let $\x$ be a Heegaard state. We consider $\x$ to be a way of assigning each sector to one of its corner as in \Cref{prop:generators}.
For each sector $S_i$, let $\x(S_i)$ be the triple point assigned to $S_i$.
\begin{itemize}
    \item If $\x(S_i)$ is the bottom corner of $S_i$, take $e_{\x,i}$ to be the empty set.
    \item If $\x(S_i)$ is not the bottom corner of $S_i$, take $e_{\x,i}$ to be the edge of $G_+$ connecting the bottom corner of $S_i$ to $\x(S_i)$. (See \Cref{fig:statemultiloopdefn}.)
\end{itemize}

\begin{figure}
    \centering
    \selectfont\fontsize{12pt}{12pt}
    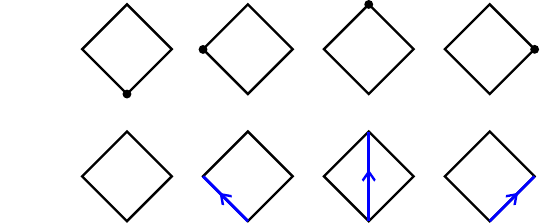
    \caption{Picking the edges $e_{\x,i}$, the union of which is the multi-loop $\mu_x$ associated to $\x$.}
    \label{fig:statemultiloopdefn}
\end{figure}

\begin{lemma} \label{lemma:stateloopisloop}
The union of edges $\bigcup_i e_{\x,i}$ is an embedded multi-loop of $G_+$.
\end{lemma}
\begin{proof}
Enumerate the triple points $v_1, \dots, v_n$ so that $v_i$ is the bottom vertex of the sector $S_i$. It suffices to show that at each $v_i$, either $v_i$ is disjoint from the chosen edges, or exactly one chosen edge enters $v_i$ and exactly one chosen edge exits $v_i$.

If $v_i = \x(S_i)$, then by definition none of the chosen edges exit $v_i$. At the same time, since $\x$ picks each triple point exactly once, $\x(S_j) \neq v_i$ for $j \neq i$, so none of the chosen edges enter $v_i$.

Otherwise $v_i = \x(S_j)$ for some $j \neq i$.
Since $\x$ picks each triple point exactly once, such a $j$ is unique. Hence exactly one chosen edge, namely $e_{\x,j}$ enters $v_i$.
In particular, $\x(S_i) \neq v_i$, so exactly one chosen edge, namely $e_{\x,i}$ exits $v_i$.
\end{proof}

\begin{defn}[Multi-loop $\mu_\x$]\label{multiloop}
Let $(\Sigma, \boa, \bob)$ be the Heegaard diagram of a veering branched surface. Define the multi-loop associated to a Heegaard state $\x\in \mathfrak{S}(\Sigma, \boa, \bob)$ to be the embedded multi-loop $\mu_\x = \bigcup_i e_{\x,i}$ of $G_+$. 
\end{defn}

In fact, the converse of \Cref{lemma:stateloopisloop} also holds.

\begin{prop} \label{prop:statesasembloops}
The map $\x \mapsto \mu_\x$ defines a bijection between the Heegaard states and the embedded multi-loops of $G_+$.
\end{prop}
\begin{proof}
Let $c$ be an embedded loop of $G_+$. We associate a Heegaard state $\x$ to $c$ as follows:
For each sector $S_i$, let $v_i$ be its bottom corner. 
If $c$ does not pass through $v_i$, then we set $\x(S_i)=v_i$. 
Otherwise, since $c$ is embedded, there is exactly one edge exiting $v_i$. We set $\x(S_i)$ to be the terminal vertex of this edge. By definition of $G_+$ this must be a corner of $S_i$.

The fact that $c$ is embedded implies that each triple point is picked exactly once by $\x$, so $\x$ is indeed a Heegaard state.

It is straightforward to verify that this is an inverse to the map $\x \mapsto \mu_\x$.
\end{proof}

\begin{eg} \label{eg:topbottomloop}
For the bottom state $\xb$ the edges $e_{\xb,i}$ are all empty sets, thus $\mu_{\xb}$ is the empty multi-loop.
For the top state $\xt$ the edges $e_{\xt,i}$ are all the vertical edges in the sectors, thus $\mu_{\xt}$ is the multi-loop that traverses all the vertical edges of $G_+$.
\end{eg}

Next, we wish to assign to $\mu_\x$ a closed multi-orbit of the blown-up flow $\phi^\sharp$.
Here, recall that $\phi$ is the pseudo-Anosov flow corresponding to the chosen veering branched surface $B$, and $\phi^\sharp$ is the blown-up flow, which is defined on $M$.

Each loop $c_+$ of $G_+$ can be homotoped to a loop $c$ of $G$ in $M$ by taking each edge in $c_+$ that does not lie in $G$, i.e. each edge that connects the bottom vertex to the top vertex of some sector $S$, and replacing it by the path in $G$ that traverses a bottom side and a top side of $S$.
We refer to $c$ as a \emph{strum} of $c_+$.
See \Cref{fig:strum}.

\begin{figure}
    \centering
    \selectfont\fontsize{8pt}{8pt}
    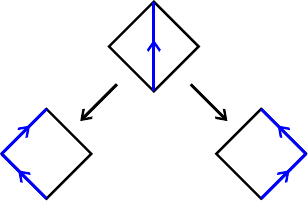
    \caption{Each edge of $G_+ \backslash G$ can be strum into a path of $G$ in two ways.}
    \label{fig:strum}
\end{figure}

More generally, we say that two loops in $G_+$ are \emph{sweep-equivalent} if they are related by a finite sequence of strums. 
More precisely, if we have a sequence of loops $c_0,\dots,c_n$ in $G_+$ where for each $i$, either $c_{i-1}$ is a strum of $c_i$ or $c_i$ is a strum of $c_{i-1}$, then we say that $c_0$ and $c_n$ are sweep-equivalent.

With this terminology in place, we have the following more precise version of \Cref{thm:vbspaflowcorr}(ii).

\begin{prop} \label{prop:looporbitcorr}
There is a bijection $\mathcal{F}$ from the sweep-equivalence classes of loops in $G_+$ to the closed orbits of $\phi^\sharp$, such that $\mathcal{F}(c)$ is homotopic to $c$ in $M$ for each loop $c$.
\end{prop}
\begin{proof}
The results of \cite{LMT23} imply that such a bijection $\mathcal{F}$ can be defined on the sweep-equivalence classes of loops in $G$. 
Indeed, \cite[Lemma 3.17]{LMT23} implies that the sweep-equivalence classes of loops are in bijection with dynamic (half-)planes, modulo deck transformation. By \Cref{thm:vbspaflowcorr}(i), these are in turn in bijection with periodic leaves of the blown-up unstable foliation, which are in bijection with the closed orbits of $\phi^\sharp$.

Meanwhile, each loop $c_+$ in $G_+$ is sweep-equivalent to a loop in $G$, namely any one of its strums, and all the strums of $c_+$ are sweep-equivalent. So the sweep-equivalence classes of loops in $G_+$ are the same as the sweep-equivalence classes of loops in $G$.
\end{proof}

\begin{defn}[Closed multi-orbit $\gamma_\x$]
Let $(\Sigma, \boa, \bob)$ be the Heegaard diagram of a veering branched surface. Define the closed multi-orbit associated to a Heegaard state $\x\in \mathfrak{S}(\Sigma, \boa, \bob)$ to be the closed multi-orbit corresponding to $\mu_\x$ under the bijection of \Cref{prop:looporbitcorr}.
\end{defn}

We record the following consequence of our definition.

\begin{prop} \label{prop:looporbithomotopic}
For each Heegaard state $\x$ the multi-loop $\mu_\x$ is homotopic to the multi-orbit $\gamma_\x$ in the 3-manifold $M$. \qed
\end{prop}

\subsection{$\Spinc$-grading of states} \label{subsec:vbsspincgrading}

The correspondence between Heegaard states and closed multi-orbits is not only conceptually pleasing, it also allows us to pin down the $\spinc$-grading of the generators.

We first recall Ozsv\' ath and Szab\' o's $\epsilon$ function:
Suppose we are given two states $\x=x_1+\dots+x_d$ and $\y = y_1+\dots+y_d$. 
We can take a collection of arcs $a$ in $\alpha_1 \cup \dots \cup \alpha_d \subset \Sigma$ with $\partial a=y_1 +\dots+y_d-x_1 -\dots-x_d$. 
Similarly, we can take a collection of arcs $b$ in $\beta_1\cup \dots \cup \beta_d \subset \Sigma$ with $\partial b=y_1 +\dots+y_d-x_1 -\dots-x_d$. 
Then the difference $a-b$ is a closed $1$-chain in $\Sigma$.
We define $\epsilon(\x,\y)$ to be its image in $H_1(M)$.

\begin{lemma} \label{lemma:epsilonrelspinc}
For each pair of Heegaard states $\x$ and $\y$ we have that 
 \[\s(\x)-\s(\y)= PD[\epsilon(\x,\y)]\] 
where $PD:  H_1(M) \to H^2(M, \partial M)$ denotes Poincar\' e duality.   \qed
\end{lemma}

In other words, the $\epsilon$ function allows one to compute the relative $\spinc$-grading of states.

\begin{lemma} \label{lemma:statesepsilon}
If $\x$ is an Heegaard state of the Heegaard diagram associated to a veering branched surface then  
\[\epsilon(\x, \xb)=[\mu_\x]=[\gamma_\x]\ .\]
\end{lemma}
\begin{proof}
Since $M$ retracts to $B$, we have $H_1(M) \cong H_1(B)$
We shall think $B$ as the $2$-complex obtained from its dual graph $\Gamma$ by attaching the sectors $S_i$ so that we have an identification 
\[\frac{H_1(G)}{\text{Span}\langle [\partial S_1 ], \dots , [\partial S_1] \rangle }\simeq H_1(M) \ . \]

To compute $\epsilon(\x,\xb)$ we have to choose a collection of arcs $a$ in $\alpha_1 \cup \dots \cup \alpha_d \subset \Sigma$ with $\partial a=\xb-\x$, and a collection of arcs $b$ in $\beta_1\cup \dots \cup \beta_d \subset \Sigma$ with $\partial b=\xb-\x$, so that $\epsilon(\xb, \x)$ is the image of $a-b$ in 
\[\frac{H_1(\Sigma)}{\text{Span}\langle[\alpha_1], \dots , [\alpha_d], [\beta_1], \dots , [\beta_d] \rangle }\simeq H_1(M) \ . \]

Meanwhile, we observe that there is a natural identification induced by a retraction:
\[\frac{H_1(\Sigma)}{\text{Span}\langle[\alpha_1], \dots , [\alpha_d]\rangle} = H_1(G)\]
and that under this identification the multi-arc $a$ correspond to $d$ constant paths positioned at the vertices. Thus we have only to choose the multi-arc $b$ so that $\epsilon(\x, \xb)$ is the image of $b$ under the projection 
\[\frac{H_1(\Sigma)}{\text{Span}\langle[\alpha_1], \dots , [\alpha_d]\rangle} = H_1(G) \to \frac{H_1(G)}{\text{Span}\langle [\partial S_1 ], \dots , [\partial S_1] \rangle } \cong H_1(M) \ . \]

To choose $b$, recall that $\mu_\x$ is the union of the edges $e_{\x,i}$. 
For each $i$, if $e_{\x,i}$ is not the vertical edge of $G_+$ in $S_i$, then $e_{\x,i}$ can be considered as a directed arc in $\partial S_i = \beta_i$. We include this directed arc in $b$. 
If $e_{\x,i}$ is the vertical edge of $G_+$ in $S_i$, then we first strum $e_{\x,i}$ by homotoping it into a path in $\partial S_i \subset G$ that transverses a bottom side and a top side of $S_i$. Then we include this directed arc in $b$.
Then by construction the image of $b$ in $H_1(M)$ equals $[\mu_\x]$.

Finally, by \Cref{prop:looporbithomotopic}, we have $[\mu_\x] = [\gamma_\x]$.
\end{proof}

With the relative gradings pinned down, it suffices to compute the (absolute) $\spinc$-grading of a single state.
We will do so in \Cref{prop:topbotspincgrading} below. 
In fact we will compute the $\spinc$-grading for two states.

Recall that $\phi^\sharp$ is the blown-up flow on $M$. 
We define $\s_{\phi^\sharp}$ to be the $\spinc$-structure determined by any vector field close to being tangent to the oriented flow lines of $\phi^\sharp$. 
(The reason for this awkward definition is because $\phi^\sharp$ may not be a smooth flow.)
We define $\overline{\s_{\phi^\sharp}}$ to be the $\spinc$-structure determined by any vector field close to being tangent to the reversed flow lines of $\phi^\sharp$. 

\begin{prop} \label{prop:topbotspincgrading}
The $\text{spin}^c$-grading of the top state $\xt$ is $\s_{\phi^\sharp}$. Similarly, the $\text{spin}^c$ grading of $\xb$ is $\overline{\s_{\phi^\sharp}}$.
\end{prop} 
\begin{proof}
By construction of $B$, we can fix a vector field $X_{\phi^\sharp}$ close to being tangent to the oriented flow lines of $\phi^\sharp$ that is tangent to the sectors of $B$ and positively transverse to its branch locus $G$. 
Observe that these properties characterize the homotopy class of $X_{\phi^\sharp}$, that is, any vector field $V$ with these two properties is homotopic to $X_{\phi^\sharp}$.

Choose a Morse function $f:M\to \R$ adapted to the Heegaard diagram as explained in \Cref{subsec:spincgrading}. 
Note that we can arrange things so that the Morse flow is tangent to the sectors (away from the boundary) and that there is one index-2 critical points $P_i$ at the center of each sector $S_i$, and one index-1 critical point $Q_i$ at the center of each compressing disk $D_i$ of the $\alpha$-curve corresponding to the vertex $v_i$.
A local model for the Morse flow is shown as the black arrows in \Cref{fig:topstatespinc}.

Each point $x_i\in \beta_i \cap \alpha_{\sigma(i)}$ of the state $\xt$ gives rise to a flow line $\gamma_{i}$ connecting the critical point $P_{\sigma(i)}$ to the critical point $Q_i$. After removing neighborhoods of $\gamma_{i}$ from $M$ we can rotate $V=\nabla f$ as suggested by the green arrows in \Cref{fig:topstatespinc} to make $V$ positively transverse to the singular locus of $B$, and hence homotopic to  $X_{\phi^\sharp}$.
This shows that the $\spinc$-grading of $\xt$ is $\s_{\phi^\sharp}$.

A similar argument, now rotating $\nabla f$ downwards away from neighborhoods of $\gamma_{i}$ shows that the $\spinc$-grading of $\xb$ is $\overline{\s_{\phi^\sharp}}$. 
\end{proof}

\begin{figure}
    \centering
    \selectfont\fontsize{8pt}{8pt}
    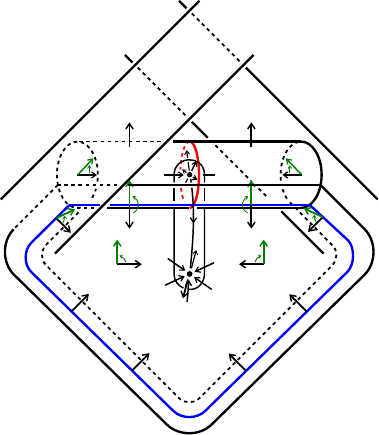
    \caption{Reasoning that the $\spinc$-grading of $\xt$ is $\s_\phi$.}
    \label{fig:topstatespinc}
\end{figure}

\Cref{thm:introspinc} can be deduced from combining these equations.

\begin{proof}[Proof of \Cref{thm:introspinc}]
We compute
\begin{align*}
\s(\x) - \overline{\s_{\phi}} &= \s(\x) - \s(\xb) & \text{by \Cref{prop:topbotspincgrading}} \\
&= PD[\epsilon(\x,\xb)] & \text{by \Cref{lemma:epsilonrelspinc}} \\
&= PD[\mu_\x] & \text{by \Cref{lemma:statesepsilon}.}
\end{align*}
\end{proof}

\Cref{thm:introtopbotgenerator} follows from combining the results of this section.

\begin{proof}[Proof of \Cref{thm:introtopbotgenerator}]
In \Cref{eg:topbottomloop} we showed that $\mu_{\xb} = \varnothing$. This implies that $\gamma_{\xb} = \varnothing$.
In \Cref{prop:topbotspincgrading} we showed that $\s(\xb) = \overline{\s_{\phi^\sharp}}$.
For any other state $\y$, combining \Cref{lemma:epsilonrelspinc} and \Cref{lemma:statesepsilon}, we have $\s(\y) - \s(\xb) = [\gamma_\y]$, where $\gamma_\y \neq \varnothing$ since $\y \neq \xb$.

In \Cref{prop:topbotspincgrading} we showed that $\s(\xt) = \s_{\phi^\sharp}$.
To show the statement about `top $\spinc$-grading', we define a multi-loop $\overset{\vee}{\mu}_\y$ and a multi-orbit $\overset{\vee}{\gamma}_\y$ by replacing $\xb$ by $\xt$ in the definitions of $\mu_\y$ and $\gamma_\y$.
More precisely, suppose we are given a state $\y$. 
For each sector $S_i$, let $\overset{\vee}{e}_{\y,i}$ be the path in the augmented dual graph $G_+$ that connects $\y(S_i)$ to the top corner of $S_i$. 
The same argument as in \Cref{lemma:stateloopisloop} shows that $\overset{\vee}{\mu}_\y := \bigcup_i \overset{\vee}{e}_{\y,i}$ is a multi-loop of $G_+$.
Note that $\overset{\vee}{\mu}_\y$ may not be embedded.
Nevertheless, we can define $\overset{\vee}{\gamma}_\y$ to be the closed multi-orbit corresponding to $\overset{\vee}{\mu}_\y$ under the bijection of \Cref{prop:looporbitcorr}. 

The same argument as in \Cref{lemma:statesepsilon} shows that $\epsilon(\xt,\y) = [\overset{\vee}{\mu}_\y]=[\overset{\vee}{\gamma}_\y]$ in $H_1(Y^\sharp)$. Thus by \Cref{lemma:epsilonrelspinc}, we have $\s(\xt) - \s(\y) = [\overset{\vee}{\gamma}_\y]$.
\end{proof}

\section{Combinatorics of the chain complex} \label{sec:chaincomplexcombin}

\subsection{Color of triple points} \label{sec:vbscolor}

Let $B$ be a veering branched surface on a compact oriented 3-manifold with torus boundary $M$.
As in the previous sections we shall denote by $G$ the dual graph, i.e. the $1$-skeleton of $B$, and $G$ the augmented dual graph (see \Cref{defn:augmenteddualgraph}).

\begin{defn}[Red/blue triple points] \label{defn:vbscolor}
A triple point $v$ of $B$ is called \textit{blue} or \textit{red} depending on whether in a neighborhood of $v$ the branched surface looks as in \Cref{fig:vbscolor} left or right. 
\end{defn}

\begin{figure}
    \centering
    \selectfont\fontsize{10pt}{10pt}
    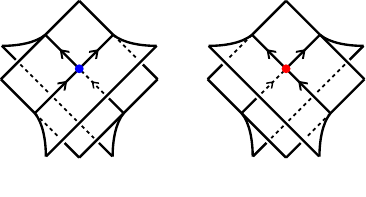
    \caption{The color of a triple point is determined by the local form of the branched surface. Here the orientation on $M$ is used to distinguish between the two local pictures.}
    \label{fig:vbscolor}
\end{figure}

Note that in \Cref{fig:vbscolor} the two local pictures are distinguished by the orientation of the 3-manifold $M$. Reversing the orientation of $M$ has the effect of switching the coloring of triple points. 

A handy mnemonic for \Cref{defn:vbscolor} is to point one's thumb downwards at $v$. If the sectors spiral out in the \textbf{L}eft-handed manner, $v$ is b\textbf{L}ue. If they spiral out in the \textbf{R}ight-handed manner, $v$ is \textbf{R}ed.

\begin{defn}[Red/blue sectors, toggle, fans]
A sector is \textit{blue/red} if its top corner is blue/red, respectively.
We say that a sector is \emph{toggle} if the colors of its top and bottom corners are different. We say that it is \emph{fan} if the colors of its top and bottom corners are the same.
\end{defn}

We shall make repeated use of the following lemma.

\begin{lemma} \label{prop:vbstogglefan}
Let $S$ be a blue/red sector of a veering branched surface $B$. The top and side corners of $S$ are blue/red, respectively. Each bottom side of $S$ is an edge of $G$. Each top side of $s$ is the union of $\delta \geq 1$ edges of $G$.

Suppose a top side of $S$ is divided into edges $e_1,\dots,e_\delta$, listed from bottom to top. Let $S_i$ be the sector that has $e_i$ as a bottom side. If $\delta=1$, then $S_1$ is blue/red fan, respectively. If $\delta \geq 2$, then $S_1$ and $S_\delta$ are toggle while $S_i$ for $i=2,\dots,\delta-1$ are red/blue fan, respectively.
\end{lemma}
\begin{proof}
This statement first appeared in \cite[Observation 2.6]{FG13} in the language of veering triangulations. In \cite[Propositions 2.8 and 2.11]{Tsa23b} the statement was translated in the language of veering branched surfaces we are using in our set up.
\end{proof}

See \Cref{fig:vbssectors} for an example illustrating \Cref{prop:vbstogglefan}.

\begin{figure}
    \centering
    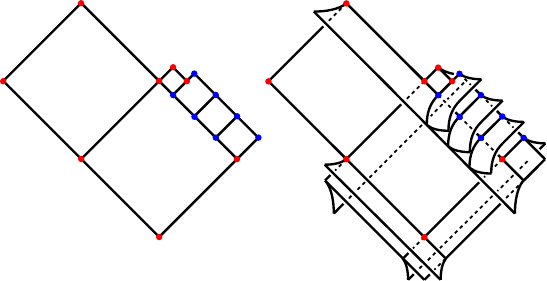
    \caption{An example of the local configuration of sectors in a veering branched surface.}
    \label{fig:vbssectors}
\end{figure}

\begin{prop} \label{prop:vbsbranchcurvecolor}
Every branch loop of a veering branched surface must meet triple points of both colors.
\end{prop}
\begin{proof}
This is the content of \cite[Proposition 2.12]{Tsa23b}.
\end{proof}

\subsection{Structure of the Heegaard diagram: local picture}

In this section we relate the structure of the Heegaard diagram associated to a veering branched surface $B$ to the combinatorics of the veering branched surface itself.

Let $(\Sigma, \alpha, \beta)$ be the Heegaard diagram associated to a veering branched surface $B$. Recall that the Heegaard surface $\Sigma$ is obtained from the boundary of a regular neighbourhood $\Sigma_0= \partial \mathcal{N}(G)$ of the 1-skeleton $G=B^{(1)}$ by cutting along a multicurve $\mathbf{c}=\{ c_1, \dots, c_k\}$ disjoint from the $\alpha$- and the $\beta$-curves. We shall call these the \emph{cutting curves}, and observe that being parallel to the branch loops of $B$, they have a natural orientation. 

Let $v$ be a triple point of $B$, i.e. a vertex of $G$. In a small neighborhood $N$ of $v$ the surface $\Sigma_0$ intersects $N$ in a 4-holed sphere as illustrated in \Cref{fig:vbstoheegaardtriplepoint}. Furthermore, among all the $\alpha$-curves, only the one associated to $v$ meets $N$. We denote this $\alpha$-curve by $\alpha_v$, and observe that it cuts $\Sigma \cap N$ into two pairs of pants $P_1$ and $P_2$.

We will think of $P_1$ and $P_2$ as one-holed cylinders, where the hole lies along the $\alpha$-curve $\alpha_v$ corresponding to the triple point. Locally there are exactly two cutting curves intersecting  $\Sigma \cap N$, one passing through $P_1$, and the other passing through $P_2$. Note that since the cutting curves are oriented we have a way to distinguish the two boundary components of the cylinder $P_1$, and $P_2$. In the pictures we shall use the convention that \emph{lower boundary} is where the cutting arc enters $P_i$, while the \emph{upper boundary} is where the cutting arc exits $P_i$. 

There are locally six $\beta$-curves that meet $N$, corresponding to the six local sectors that meet $v$. (Here the word `local' is crucial: these are only arcs of $\beta$-curves that may reconnect outside of the neighbourhood $N$.)
Among the six $\beta$-curves, two of them do not meet $\alpha_v$. These two $\beta$-curves correspond to the two sectors not having $v$ as a corner (in the local picture of the triple point these are the sectors represented by the two `flaps' coming in). Each cylinder $P_i$ contains exactly one of these two $\beta$-curves. This $\beta$-curve runs between the upper and lower boundary components of the cylinder. 

Among the remaining four $\beta$-curves, there is one that corresponds to the local sector that has $v$ lying on its bottom corner. This $\beta$-curve meets both cylinders $P_1$ and $P_2$ in an arc that runs between the hole and the upper boundary component.
Similarly, the $\beta$-curve that corresponds to the local sector that has $v$ lying on its top corner meets both cylinders $P_1$ and $P_2$ in an arc that runs between the hole and the lower boundary component.

Each of the remaining two $\beta$-curves corresponds to a local sector $S$ that has $v$ lying on a side corner.
Each one of these $\beta$-curves meet both cylinders $P_1$ and $P_2$ giving rise to two arcs of $\beta$-curves on each one of them.
The boundary of $S$ consists of two edges of $G$, one of them pointing towards $v$ and the other away from $v$.
Without loss of generality suppose $P_1$ contains the edge pointing towards $v$.
Then the $\beta$-arc determined by $S$ on $P_1$ runs from the lower boundary component to the hole, while the $\beta$-arc determined by $S$ on $P_2$ runs from the hole to the upper boundary component.

Summarizing,  each cylinder $P_i$ intersects the $\beta$-curves in five arcs. 
Four of these arcs meet the hole carved by $\alpha_v$.
Furthermore, since we distinguished the upper and lower boundary components of $P_i$, we can refer to these four arcs as the \emph{upper left}, \emph{upper right}, \emph{lower left}, \emph{lower right} arcs unambiguously. The remaining $\beta$-arc runs between the upper and lower boundary components of $P_i$. We refer to this as the \emph{lone arc}.

There is now a difference in combinatorics: On the cylinder $P_i$ the lone arc can lie on the left or on the right of the cutting arc. This is where the combinatorics of \Cref{sec:vbscolor} plays its role.

\begin{lemma}\label{lemma:localmodel}
If $v$ is blue triple point, then the cutting arc lies to the right of the lone $\beta$-arc.
If $v$ is red triple point, then the cutting arc lies to the left of the lone $\beta$-arc.
\end{lemma}
\begin{proof}
This is evident from the local picture. See \Cref{fig:heegaardcombinalphaglue} first two rows.
\end{proof}

\begin{figure}
    \centering
    \selectfont\fontsize{12pt}{12pt}
    \resizebox{!}{10.5cm}{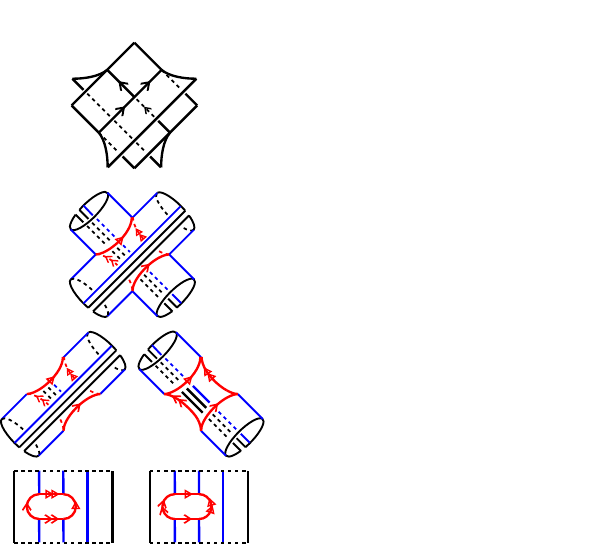}
    \caption{The local combinatorics of the Heegaard diagram near a triple point $v$ depends on the color of $v$.}
    \label{fig:heegaardcombinalphaglue}
\end{figure}

After cutting $\Sigma_0$ along $\mathbf{c}$ the cylinders $P_1$ and $P_2$ open up  and give rise to two punctured rectangles. These look as in \Cref{fig:heegaardcombinalphaglue} last row where the vertical sides represent portions of the boundary of the Heegaard surface $\Sigma$. 

\begin{cor}\label{cor:diagramcolors}
According to the color of the underlying triple point, there are two local models for the Heegaard diagram $(\Sigma, \boa, \bob)$ in proximity of the $\alpha$-curves. The two local models are depicted in \Cref{fig:heegaardcombinalphaglue} where the two arrows suggest how to glue the rectangles to get back a neighborhood of the $\alpha$-curve in the Heegaard diagram. 
\end{cor}
\begin{proof}
Apply \Cref{lemma:localmodel}, and cut the local model along the $\alpha$-curve, and the two arcs of cutting curves. 
\end{proof}

Note that in a neighborhood $N$ of a triple point $v$ there is  exactly one component $x_i^\top$ of the top Heegaard state $\xt=x_1^\top+ \dots + x_n^\top$, and exactly one component $x_j^\bot$ of the bottom  state $\xb=x_1^\bot+ \dots + x_n^\bot$.
Furthermore, if $v$ is a blue triple point then on each component  of $(\Sigma \cap N) \setminus  \alpha_v=P_1 \cup P_2$, the point  $x_i^\top$ lies on the lower left $\beta$-arc while the point of $x_j^\bot$ lies on the upper right $\beta$-arc.
Similarly, if $v$ is a red triple point then on each cylinder $P_1$ and $P_2$, the point  $x_i^\top$ lies on the lower right $\beta$-arc while the point  $x_j^\bot$ lies on the upper left $\beta$-arc.
See the purple and yellow dots on \Cref{fig:heegaardcombinalphaglue}.

\begin{rmk}
    When we identify the two rectangles along their holes to reconstruct a neighborhood of an $\alpha$-curve, the gluing must identify the components of the top and bottom Heegaard states. This, together with the fact that the identification  must respect the orientation of the Heegaard surface $\Sigma$, gives a characterization of the gluing instructions we illustrated in \Cref{fig:heegaardcombinalphaglue}.
\end{rmk}

\begin{figure}
    \centering
    \selectfont\fontsize{8pt}{8pt}
    \resizebox{!}{5cm}{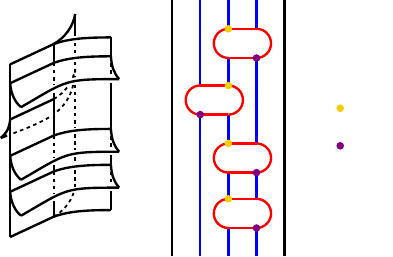}
    \caption{Reading the combinatorics of $\Sigma \backslash \alpha$ off from the colors of triple points lying on the corresponding branch loop.}
    \label{fig:heegaardcombinbranchcurve}
\end{figure}

\subsection{Structure of the Heegaard diagram: global picture}

With the local picture in mind, we can now understand the global combinatorics. This is governed by the dual graph $G$ of the branched surface.

First of all we observe that the branch loops give us a decomposition of $G$ into a union of cyclic graphs: $G=C_1\cup \dots \cup C_l$. After cutting the Heegaard surface $\Sigma$ along the $\alpha$-curves we get a collection of holed annuli $A_1, \dots , A_l$ in one-to-one correspondence with the cyclic components $C_1, \dots, C_l$ of this decomposition. 
This is because after compressing the $\alpha$-curves, a tubular neighborhood $U=\mathcal{N}(G)$ of the dual graph $G$ decomposes into a union of solid  tori $T_1, \dots , T_l$ each containing a unique branch curve $b_i$ in its interior, and a unique cutting curve $c_i$ on its boundary.

Based on the conventions we set in \Cref{defn:vbscolor} the vertices of the dual graph $G$ can be colored blue or red.
The portion of the Heegaard diagram $(\Sigma,\boa, \bob)$ lying inside the annulus $A_i$ can be reconstructed from the cyclic component $C_i$ using this coloring: 
Each vertex $v$ of $C_i$ corresponds to an $\alpha$-curve delimiting one of the holes in the interior of the annulus $A_i$. We can thus decompose each annulus $A_i$ into a union of rectangles each containing a different $\alpha$-curve. As established in \Cref{cor:diagramcolors}, inside each of these rectangles the Heegaard diagram looks as in \Cref{fig:heegaardcombinalphaglue} left or right according to the color of the corresponding vertex.

We summarize our findings in the following proposition.

\begin{prop}\label{prop:globalcombinatorics}
Suppose that $B$ is a veering branched surface, and that $(\Sigma, \boa, \bob)$ is its corresponding Heegaard diagram. Then $\Sigma \setminus \alpha_1 \cup \dots \cup \alpha_n$ decomposes as a union of punctured annuli $A_0, \dots , A_n$ in one-to-one correspondence with the branch loops of $B$. Furthermore, each annulus $A_i$ can be decomposed into a union of punctured rectangles, one for each $\alpha$-curve it contains in its closure.  
Each of these rectangles looks like one of the two local models depicted in \Cref{fig:heegaardcombinalphaglue}, and the specific local model is decided from the color of the triple point corresponding to the $\alpha$-curve. \qed
\end{prop}

\subsection{Left and right domains}
We now turn to analyze the combinatorics of domains of the diagram $(\Sigma, \boa, \bob)$ associated to a veering branched surface.

To perform our task we use \Cref{prop:globalcombinatorics} to decompose the Heegaard diagram $(\Sigma, \boa, \bob)$ into a union of annuli.  
It is useful to think about the combinatorics of the $\alpha$-curves, and the $\beta$-arcs on an annulus $A_i$ corresponding to a branch loop $b_i$ as follows:
\begin{itemize}
    \item Start with an annulus $S^1 \times [0,1]$ with three strands $S^1 \times \{\frac{1}{4}\}$, $S^1 \times \{\frac{1}{2}\}$, $S^1 \times \{\frac{3}{4}\}$ running along the $S^1$-direction. We refer to these as the \emph{left}, \emph{middle}, and \emph{right} strands respectively.
    \item Going around $b_i$, we encounter a sequence of triple points that are coloured either red or blue. For each triple point $v$, we introduce a corresponding puncture on the annulus by cutting out a hole that intersects the left and middle strand if $v$ is colored blue, and cutting out a hole that intersects the middle and right strand if $v$ is colored red. See Figure \Cref{fig:heegaardcombinbranchcurve} for an example.
    \item The elementary domains bounded between $S^1 \times \{0\}$ and $S^1 \times \{\frac{1}{4}\}$, and between $S^1 \times \{\frac{3}{4}\}$ and $S^1 \times \{1\}$ are the ones that do not contain a basepoint.
\end{itemize}

Under this perspective, each elementary domain $D$ that lies on $A_i$ and does not contain a basepoint either lies between the left and middle strands, or lies between the middle and right strands. We say that $D$ is a \textit{left domain} or a \textit{right domain} respectively.

Note that each empty elementary domain $D$ has four distinguished consecutive corners that we refer to as the \emph{upper-left}, \emph{upper-right}, \emph{lower-left}, and \emph{lower-right} corners, as suggested by \Cref{fig:heegaardcombinalphaglue}.
If $D$ is a left domain, then the upper-left corner of $D$ is a coordinate of the top state $\xt$ and the lower-right corner of $D$ is a coordinate of the bottom state $\xb$.
Similarly, if $D$ is a right domain, then the upper-right corner of $D$ is a coordinate of the top state $\xt$ and the lower-left corner of $D$ is a coordinate of the bottom state $\xb$.

Furthermore, we observe that each empty elementary domain has exactly one coordinate of $\xt$ but possibly more than one coordinate of $\xb$ on its boundary.

\begin{prop} \label{prop:effdomainemb}
In the Heegaard diagram of a veering branched surface, let $D = \sum n_i D_i$ be an  effective  domain  connecting a state $\x$ to a state $\y$. If $n_\z(D)=0$ then $D$ is an \emph{embedded} domain, that is, each  multiplicity $n_i$ is either $0$ or $1$.
\end{prop}
\begin{proof}
Each elementary domain $D_i$ has as a corner point a component $p$ of the top Heegaard state $\xt$.  Among the four quadrants at $p$ two opposite ones are the corners of two elementary regions containing a basepoint. Let $D_i$ and $D_k$ be the remaining two  elementary domains. Since $n_\z(D)=0$ when we write down \Cref{eq:initialfinallocal} at $p$ we get that $n_i+n_k\in \{1, 0, -1\}$. It follows from effectiveness that $n_i=1$ or $n_i=0$. 
\end{proof}

\begin{rmk} \label{rmk:embdomainssubsurfaces}
Let $D$ be an effective domain with $n_\z(D) = 0$. The fact that $D$ is embedded means that we can identify an effective domain $D$ with $n_\z(D)=0$ with the subsurface of $\Sigma$ that is the closure of the union of the elementary domains that have multiplicity $1$ in $D$.
\end{rmk}

\subsection{Admissibility} \label{subsec:admissibility}

As we pointed out in \Cref{thm:cfwelldefined}, to have a well defined Heegaard Floer chain complex one needs to stick to some special  Heegaard diagrams called admissible. 
In this subsection, we show that the diagram $(\Sigma, \boa, \bob)$ associated to a veering branched surface is admissible.
First, we recall the definition of admissibility.

\begin{defn}[Periodic domains]
Let $D_1, \dots , D_r$ denote the elementary domains that do not contain a basepoint in the associated multi-pointed diagram.
A \emph{periodic domain} is a domain $P=\sum_{i=1}^r n_i  D_i$ whose boundary $\partial P\in C_2(\overline{\Sigma})$ is a union of full $\alpha$- and $\beta$-curves. 
\end{defn}

In the notation of \Cref{defn:initialfinalstate}, a periodic domain is a domain such that $\partial_{\boa} P=0$, and $\partial_{\bob} P=0$. In the correspondence between domains and Whitney disks, periodic domains correspond to disks connecting an intersection point $\x$ to itself, that is, elements of $\pi_2(\x,\x)$ for some $\x \in \mathbb{T}_{\boa} \cap \mathbb{T}_{\bob}$. Existence of \emph{effective} periodic domains is associated with phenomena like bubbling and boundary degenerations that give fatal complications when running the Lagrangian Floer homology routine.

\begin{defn}[Admissibility] A multi-pointed Heegaard diagram is called \emph{admissible} if every periodic domain $P \neq 0$ has both positive an negative coefficients, that is, there is no effective periodic domain. 
\end{defn}

In \cite[Proposition 3.15]{Juh06} it is shown that if a sutured Heegaard diagram is not admissible it can be assumed to be so after complicating it with some Heegaard moves. In our set up this operation is not needed. Indeed, one of the key features of the veering condition is that it gives rise to Heegaard diagrams that are automatically admissible.

\begin{lemma} \label{lemma:nofullalphabetacurve}
Let $D$ be an effective domain with $n_\z(D) = 0$. Then the subsurface associated to $D$ as in \Cref{rmk:embdomainssubsurfaces} cannot contain a full $\alpha$- or $\beta$-curve. 
\end{lemma}
\begin{proof}
Each $\alpha$- and $\beta$-curve contains a coordinate of the top state $\xt$.
As pointed out in \Cref{prop:effdomainemb}, at each coordinate of $\xt$, there are two opposite elementary domains containing a basepoint.
If (the subsurface associated to) $D$ contained a full $\alpha$- or $\beta$-curve, then it would contain all elementary domains lying to one side of this curve, thus contain an elementary domain containing a basepoint, contradicting $n_z(D) = 0$.
\end{proof}

\begin{prop} \label{prop:admissibility} 
The diagram $(\Sigma, \boa, \bob)$ is admissible.
\end{prop} 
\begin{proof}
An effective periodic domain contains a full $\alpha$- or $\beta$-curve in its boundary, but \Cref{lemma:nofullalphabetacurve} states that this impossible.
\end{proof}

Thus we can consider the Heegaard Floer chain complex  $CF(\Sigma, \boa, \bob) $. As consequence of \Cref{thm:sfhindependence} we have that the homology of $CF(\Sigma, \boa, \bob) $ only depends on the topology of the sutured manifold  $(M,\Gamma)$.
This completes the proof of \Cref{thm:introinformal}.

\Cref{lemma:nofullalphabetacurve} is stronger than \Cref{prop:admissibility} since it implies that no boundary component of an effective domain $D$ with $n_\z(D)$ can have a boundary component lying along a full $\alpha$- or $\beta$-curve.
In other words, each boundary component of $D$ has corners.
Each corner $p$ can be either \emph{convex}, if $D$ contains one of the four quadrants at $p$, or \emph{concave}, if $D$ contains three of the four quadrants at $p$.

The following proposition will be used in the sequel paper \cite{AT25b}.

\begin{prop}
Let $D$ be an effective domain with $n_\z(D) = 0$. Then the subsurface associated to $D$ as in \Cref{rmk:embdomainssubsurfaces} cannot contain a coordinate of $\xt$ in the interiors of its sides.
\end{prop}
\begin{proof}
This again uses the fact that at each coordinate of $\xt$, there are two opposite elementary domains containing a basepoint.
If $D$ contains a coordinate of $\xt$ in the interior of one of its sides, then it must contain one of these elementary domains, contradicting $n_z(D) = 0$.
\end{proof}

\subsection{Non-vanishing results} 

We conclude this section studying the homology of the chain complex $CF(\Sigma, \boa, \bob) $ associated to a veering branched surface $B$ via its Heegaard diagram. We shall concentrate on the portion of the chain complex relative to the $\spinc$-grading of the flow $\s_\phi$ and its conjugate, where the top generator $\xt$, and the bottom generator $\xb$ are located.

\begin{prop} \label{prop:topstatenontrivial}
The top Heegaard state $\xt$ is not the initial or final states of any effective domain $D$ with $n_\z(D)=0$.
\end{prop}
\begin{proof}
Suppose $\xt$ is the initial state of some effective domain $D = \sum n_i D_i$. In \Cref{prop:effdomainemb} we showed that $D$ must be embedded, so each elementary domain has coefficient $1$ or $0$.

First we claim that no left elementary domain $D_i$ can be included in $D$. 
To prove this, we consider the top left corner of $D_i$, which is a coordinate $p \in \xt$ of the top generator.
As we observed in the proof of \Cref{prop:effdomainemb}, of the four regions showing up at $p$, two contain a basepoint, and the coefficients of the remaining two, $D_i$ and $D_k$ say, must satisfy \Cref{eq:initialfinallocal}. This reads as $-n_i-n_k=1$, or $-n_i-n_k=0$ depending if $p$ is a component of the the terminal state of $D$. Since $D\geq 0$ the first case is not possible, and the second case is only possible if $n_i=0$.

Secondly, we claim that no right elementary domain $D_i$ can be included in $D$. 
To prove this, we consider the bottom right corner of $D_i$, which is a coordinate $p \in \xb$ of the bottom generator.
Of the four regions showing up at $p$, two of them are left domains hence have coefficient $0$ by the paragraph above, and the coefficients of the remaining two, $D_i$ and $D_k$ say, must satisfy \Cref{eq:initialfinallocal}. 
Since $\xt$ and $\xb$ do not have any coordinates in common, $p$ does not lie in the initial state of $D$, and so this reads as $n_i+n_k=-1$, or $n_i+n_k=0$ depending if $p$ is part of the terminal state or not. This leads to a contradiction in both cases unless $n_i=0$.

The proof that $\xt$ is not the final state of any effective domain is similar.
\end{proof}

\begin{prop} \label{prop:botstatenontrivial}
The bottom Heegaard state $\xb$ is not the initial or final states of any effective domain $D$ with $n_\z(D)=0$.
\end{prop}
\begin{proof}
The proof follows the same strategy as \Cref{prop:topstatenontrivial}.
Suppose $\xb$ is the final state of some effective domain $D = \sum n_i D_i$. 

First we claim that no left elementary domain $D_i$ can be included in $D$. 
To achieve this we write down \Cref{eq:initialfinallocal} at the top left corner of $D_i$ where a coordinate $p \in \xt$ of the top state lies. Again, we observe that two of the four regions meeting at $p$  contain a basepoint. Since $\xt$ and $\xb$ do not have any coordinates in common, $p$ does not lie in the final state of $D$, and so the coefficients $n_i$ and $n_k$ of the two remaining regions must satisfy $-n_i-n_k=1$, or $-n_i-n_k=0$. Again since $D\geq 0$ the only possibility is that $n_i=n_j=0$.

Secondly, we claim that no right elementary domain $D_i$ can be included in $D$.
To achieve this we write down \Cref{eq:initialfinallocal} at the bottom right corner of $D_i$, which is occupied by a coordinate $p \in \xb$ of the bottom state. 
The equation at this point writes as $n_i+n_k=-1$, or $n_i+n_k=0$ since two opposite regions at $p$ are left domains. Again we reach a contradiction with the fact that $D$ is effective.

The proof that $\xb$ is not the initial state of any effective domain is similar.
\end{proof}

We are now ready to complete the proof of \Cref{thm:introtopbotnontrivial} and its corollaries.

\begin{proof}[Proof of \Cref{thm:introtopbotnontrivial}]
Since $\xt$ is not the initial state of any effective domain, there are no Whitney discs $\phi \in \pi_2(\x,\y)$ with holomorphic representatives by \Cref{lemma:holdisctodomain}. So $\partial \x = 0$.

Since $\xt$ is not the final state of any effective domain, there are no Whitney discs $\phi \in \pi_2(\y,\x)$ with holomorphic representatives by \Cref{lemma:holdisctodomain}. So $\x$ does not appear in the differential of any state $\y$. In particular $\x$ does not lie in the image of the differential.

Thus $\xt$ determines a non-trivial homology class in $SFH(M,\Gamma)$. The proof for $\xb$ is similar.
\end{proof}

\begin{proof}[Proof of \Cref{cor:introdimgeq2}]
Under the context of the corollary, $SFH(Y^\sharp,\s_{\phi^\sharp})$ contains $\xt$ hence has dimension $\geq 1$, and $SFH(Y^\sharp,\overline{\s_{\phi^\sharp}})$ contains $\xb$ hence has dimension $\geq 1$.
\end{proof}

\begin{proof}[Proof of \Cref{cor:introfibered}]
Our first task is to compute the image of $\s(\xt)$ under the chosen identification $\Spinc(Y^\sharp) \cong H_1(Q)$ that sends $\overline{\s_{\phi^\sharp}}$ to $0$.
Let $b_1,\dots,b_l$ be the branch loops of $B$. Observe that $2[\mu_{\xt}] = 3 \sum_i [b_i]$ in $H_1(Y^\sharp)$. This is because if we take two copies of the vertical edge of $G_+$ in each sector and strum them in the two possible ways, the resulting collection of loops traverses each edge $e$ of $G$ three times, since $e$ meets three sectors.
Thus $\s(\xt) = \overline{\s_{\phi^\sharp}} + [\mu_\x]$ is sent to $\frac{3}{2} \sum_i [b_i]$ under our identification $\Spinc(Y^\sharp) \cong H_1(Q)$.

Still using our identification, since $\s(\xb) = \overline{\s_{\phi^\sharp}}$ is sent to $0$, the non-trivial homology class determined by $\xb$ lives in $SFH(Y^\sharp,0)$.
Meanwhile, since $\s(\xt)$ is sent to $\frac{3}{2} \sum_i [b_i]$, and $\langle \frac{3}{2} \sum_i [b_i], [S^\sharp] \rangle = \frac{3}{2} \sum_{x \in S \cap \mathcal{C}} p_x = \frac{3e}{2}$, the non-trivial homology class determined by $\xb$ lives in $SFH(Y^\sharp,\frac{3e}{2})$.

Finally, for any state $\y$ that is not the top or bottom state, \Cref{thm:introtopbotgenerator} states that $\s(\y) - \s(\xb)$ is the homology class of a nonempty closed multi-orbit, thus $\langle \s(\y) - \s(\xb), [S^\sharp] \rangle > 0$. 
This implies that $\langle \s(\y), [S^\sharp] \rangle > \langle \s(\xb), [S^\sharp] \rangle = 0$.
Similarly, since $\s(\xt) - \s(\y)$ is the homology class of a nonempty closed multi-orbit, thus $\langle \s(\xt) - \s(\y), [S^\sharp] \rangle > 0$.
This implies that $\langle \s(\y), [S^\sharp] \rangle < \langle \s(\xt), [S^\sharp] \rangle = \frac{3e}{2}$.
Hence any other non-trivial homology classes must live in $SFH(Y^\sharp,n)$ for $0<n<\frac{3e}{2}$.
\end{proof}

\section{Polynomial invariants} \label{sec:polyinv}

There are two goal of this section:
In \Cref{subsec:polysdefn}, we define the taut, veering, and anti-veering polynomials associated to a veering branched surface $B$.
In \Cref{subsec:zetafunctions}, we explain the connection between these polynomial invariants and zeta functions of various flows derived from the pseudo-Anosov flow $\phi$ corresponding to $B$.

\subsection{Algebraic preliminaries}

We first recall some algebraic facts. Let $G$ be a finitely generated free abelian group. The \emph{group ring} of $G$ (with $\mathbb{Z}$-coefficients) is the ring 
$$\mathbb{Z}[G] = \left\{ \sum_{g \in G} a_g g \mid a_g \in \mathbb{Z}, \text{all but finitely many $a_g$ are zero} \right\}.$$
Given a basis $\{g_1,\dots,g_n\}$ of $G$, we can identify $\mathbb{Z}[G]$ with the multivariable polynomial ring $\mathbb{Z}[t_1,t_1^{-1},\dots,t_n,t_n^{-1}]$ by sending $\sum_{i_1,\dots,i_n} a_{g_1^{i_1} \dots g_n^{i_n}} g_1^{i_1} \dots g_n^{i_n}$ to $\sum_{i_1,\dots,i_n} a_{g_1^{i_1} \dots g_n^{i_n}} t_1^{i_1} \dots t_n^{i_n}$.
In particular, $\mathbb{Z}[G]$ is a unique factorization domain.

Let $[D]$ be a $s \times r$ matrix with entries in $\mathbb{Z}[G]$.
If $r \geq s$, we define the \emph{(zeroth) Fitting invariant} of $[D]$, which we denote as $\Fit([D])$, to be the greatest common divisor of the $s \times s$ submatrices of $[D]$.
If $r < s$, we define $\Fit([D])=0$.
Note that $\Fit([D])$ is only well-defined up to multiplication by a unit of $\mathbb{Z}[G]$.

Let $\mathcal{S}$ and $\mathcal{R}$ be finitely generated free $\mathbb{Z}[G]$-modules and let $D:\mathcal{R} \to \mathcal{S}$ be a $\mathbb{Z}[G]$-module homomorphism. 
Given a basis $\beta_\mathcal{R}$ for $\mathcal{R}$ and a basis $\beta_\mathcal{S}$ for $\mathcal{S}$, $D$ can be represented by a matrix $[D]_{\beta_\mathcal{S}}^{\beta_\mathcal{R}}$ with entries in $\mathbb{Z}[G]$.
We define the \emph{(zeroth) Fitting invariant} of $D$ to be $\Fit(D)=\Fit([D]_{\beta_\mathcal{S}}^{\beta_\mathcal{R}})$.

Let $\mathcal{M}$ be a $\mathbb{Z}[G]$-module. Suppose $\mathcal{M}$ is finitely presented, then there exists finitely generated free $\mathbb{Z}[G]$-modules $\mathcal{S}$ and $\mathcal{R}$, and a $\mathbb{Z}[G]$-module homomorphism $D:\mathcal{R} \to \mathcal{S}$ such that $\mathcal{M} \cong \coker D$.
We define the \emph{(zeroth) Fitting invariant} of $\mathcal{M}$ to be $\Fit(D)$.
It is a classical algebraic fact that this definition is independent of the choice of $D$, up to multiplication by a unit of $\mathbb{Z}[G]$.
See \cite[Chapter 3.1]{Nor76} for more details.

\subsection{Relations and polynomials} \label{subsec:polysdefn}

For the rest of this section, let $B$ be a veering branched surface on a 3-manifold $M$. Let $G=H_1(M)/\text{Torsion}$. Let $\widehat{M}$ be the universal abelian cover of $M$, i.e. $\widehat{M}$ is the covering space of $M$ corresponding to the kernel of the homomorphism $\pi_1(M) \to H_1(M) \to G$. 
Let $\widehat{B}$ be the lift of $B$ to $\widehat{M}$.

The deck transformation group $G$ acts freely on the set of sectors/edges/triple points of $\widehat{B}$, where orbits can be identified with the sectors/edges/triple points of $B$, respectively.
It follows that the free $\mathbb{Z}$-module generated by the sectors/edges/triple points of $\widehat{B}$ is a free $\mathbb{Z}[G]$-module with rank equals to the number of sectors/edges/triple points of $B$, respectively.
We denote these finitely generated free $\mathbb{Z}[G]$-modules as $\mathcal{S}$, $\mathcal{E}$, and $\mathcal{T}$, respectively.

Over the next few subsections, we will define some relations in $\mathcal{S}$. 
These relations will in turn define modules whose Fitting ideal gives the respective polynomial invariants.

\subsubsection{The face relation and the taut polynomial}

Let $\widehat{e}$ be an edge of $\widehat{B}$. Let $\widehat{b}_1$, $\widehat{b}_2$, and $\widehat{t}$ be the sectors of $\widehat{B}$ that contain $\widehat{e}$ in its boundary, where the maw coorientation on $\widehat{e}$ points out of $\widehat{b}_1$ and $\widehat{b}_2$, and into $\widehat{t}$. See \Cref{fig:polyrels} left.
The \emph{face relation} at $\widehat{e}$ is the relation $\widehat{b}_1+\widehat{b}_2-\widehat{t}$. 
A mnemonic for this relation is to think of the widths of $\widehat{b}_1$ and $\widehat{b}_2$ adding together to become the width of $\widehat{t}$ as one moves across $\widehat{e}$.

Let $D_\Theta:\mathcal{E} \to \mathcal{S}$ be the $\mathbb{Z}[G]$-module homomorphism defined by sending $\widehat{e}$ to the face relation at $\widehat{e}$. The \emph{taut module} $\mathcal{M}_\Theta$ is defined to be the cokernel of $D_\Theta$.
The \emph{taut polynomial} $\Theta$ is defined to be the Fitting invariant of $\mathcal{M}_\Theta$.

\subsubsection{The tetrahedron relation and the veering polynomial}

Let $\widehat{v}$ be a triple point of $\widehat{B}$.
We set up some notation for the edges and sectors that meet $\widehat{v}$. The reader should refer to \Cref{fig:polyrels} right for the next few paragraphs.

\begin{figure}
    \centering
    \selectfont\fontsize{8pt}{8pt}
    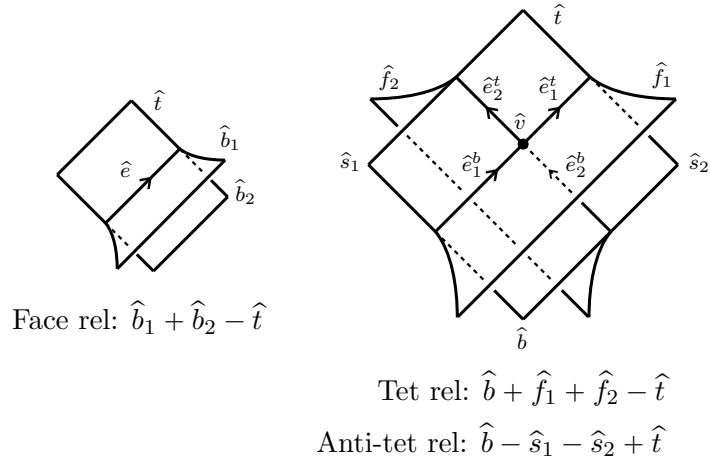
    \caption{The relations used to define the taut, veering, and anti-veering polynomials.}
    \label{fig:polyrels}
\end{figure}

Let $\widehat{e}^b_1$ and $\widehat{e}^b_2$ be the two edges of $\widehat{B}$ that have terminal vertex at $\widehat{v}$, and let $\widehat{e}^t_1$ and $\widehat{e}^t_2$ be the two edges of $\widehat{B}$ that have initial vertex at $\widehat{v}$, where $\widehat{e}^b_i$ and $\widehat{e}^t_i$ lies on the same branch loop locally at $\widehat{v}$ for each $i$.

Let $\widehat{b}$, $\widehat{f}_1$, $\widehat{f}_2$, $\widehat{s}_1$, $\widehat{s}_2$, and $\widehat{t}$ be the sectors of $\widehat{B}$ that contain $\widehat{v}$ in its boundary, where
\begin{itemize}
    \item $\widehat{b}$ is the sector lying below $\widehat{v}$,
    \item $\widehat{t}$ is the sector lying above $\widehat{v}$,
    \item $\widehat{f}_i$ contains $\widehat{e}^b_i * \widehat{e}^t_i$ in its boundary, and
    \item $\widehat{s}_i$ contains $\widehat{e}^b_i * \widehat{e}^t_{i+1}$ in its boundary (where the indices are taken mod $2$).
\end{itemize}

The \emph{tetrahedron relation} at $\widehat{v}$ is the relation $\widehat{b}+\widehat{f}_1+\widehat{f}_2-\widehat{t}$.
Similar to the face relation, one can interpret the tetrahedron relation as the widths of $\widehat{b}$, $\widehat{f}_1$, and $\widehat{f}_2$ adding up to give the width of $\widehat{t}$.

Let $D_V:\mathcal{T} \to \mathcal{S}$ be the $\mathbb{Z}[G]$-module homomorphism defined by sending $\widehat{v}$ to the tetrahedron relation at $\widehat{v}$. The \emph{veering module} $\mathcal{M}_V$ is defined to be the cokernel of $D_V$.
The \emph{veering polynomial} $V$ is defined to be the Fitting invariant of $\mathcal{M}_V$.

Observe that the tetrahedron relation $\widehat{b}+\widehat{f}_1+\widehat{f}_2-\widehat{t}$ at $\widehat{v}$ can be written as the sum $(\widehat{b}+\widehat{f}_1-\widehat{s}_1)+(\widehat{s}_1+\widehat{f}_2-\widehat{t})$, which is exactly the sum of the face relations at $\widehat{e}^b_1$ and $\widehat{e}^t_2$. 
Building on this, Landry-Minsky-Taylor showed that the taut and veering polynomials are related.
To state the relation precisely, we have to introduce some terminology. 

An \textit{anti-branch loop} of $B$ (called an AB-cycle in \cite{LMT24}) is an embedded loop of the dual graph $G$ that is not smoothly embedded at every vertex.
That is, an anti-branch loop `takes a turn' at each vertex.
Each edge of $G$ lies in exactly one anti-branch loop. In other words, $G$ admits a canonical decomposition as a union of anti-branch loops.

Note that each anti-branch loop lies on $B$. 
We say that an anti-branch loop $a$ is \emph{orientation-preserving} if the plane bundle $TB|_a$ is orientable, otherwise we say that $a$ is \emph{orientation-reversing}.
By considering the bases on $B$ determined by pairs of orientations on adjacent edges in $a$, one can see that $a$ is orientation-preserving if and only if it is of even length.

\begin{thm}[{\cite[Theorem 6.1 and Lemma 6.2]{LMT24}}] \label{thm:tautveering}
Let $B$ be a veering branched surface on a 3-manifold $M$. 
Let $a_1,\dots,a_m$ be the anti-branch loops of $B$.

Suppose $b_1(M) \geq 2$. Then 
\begin{equation} \label{eq:tautveering}
V = \Theta \cdot \prod_{\text{$a_i$ ori. pres.}} (1-[a_i]) \prod_{\text{$a_i$ ori. rev.}} (1+[a_i]).
\end{equation}

If instead $b_1(M)=1$, let $t$ be a generator of $G$. Then \Cref{eq:tautveering} holds up to possibly multiplying the left hand side by $1+t$ or $1-t$. \qed
\end{thm}

\subsubsection{The anti-tetrahedron relation and the anti-veering polynomial} \label{subsec:antiveeringpoly}

We continue using the notation in \Cref{fig:polyrels}. 
The \emph{anti-tetrahedron relation} at the triple point $\widehat{v}$ is the relation $\widehat{b}-\widehat{s}_1-\widehat{s}_2+\widehat{t}$. 

Let $D_A:\mathcal{T} \to \mathcal{S}$ be the $\mathbb{Z}[G]$-module homomorphism defined by sending $\widehat{v}$ to the anti-tetrahedron relation at $\widehat{v}$. The \emph{anti-veering module} $\mathcal{M}_A$ is defined to be the cokernel of $D_A$.
The \emph{anti-veering polynomial} $A$ is defined to be the Fitting invariant of $\mathcal{M}_A$.

Observe that the anti-tetrahedron relation $\widehat{b}-\widehat{s}_1-\widehat{s}_2+\widehat{t}$ at $\widehat{v}$ can be written as the difference $(\widehat{b}+\widehat{f}_1-\widehat{s}_1)-(\widehat{s}_2+\widehat{f}_1-\widehat{t})$, which is exactly the difference of the face relations at $\widehat{e}^b_1$ and $\widehat{e}^t_1$.
Since this observation will be used repeatedly in \Cref{sec:factorantiveerpoly}, we record it more formally as a lemma.

\begin{lemma} \label{lemma:facereldiff}
Let $\widehat{e}^b$ and $\widehat{e}^t$ be edges of $\widehat{B}$ such that $\widehat{e}^b * \widehat{e}^t$ takes a turn at a vertex $\widehat{v}$. 
Then
$$D_\Theta(\widehat{e}^b) - D_\Theta(\widehat{e}^t) = D_A(\widehat{v}).$$ \qed
\end{lemma}

We have the following analogue of \Cref{thm:tautveering} for the anti-veering polynomial.

\begin{thm} \label{thm:tautantiveering}
Let $B$ be a veering branched surface on a 3-manifold $M$. Let $b_1,\dots,b_n$ be the branch loops of $B$.

Suppose $b_1(M) \geq 2$, then 
\begin{equation} \label{eq:tautantiveering}
A = \Theta \cdot \prod_i (1-[b_i]).
\end{equation}

If instead $b_1(M)=1$, let $t$ be a generator of $G$. Then \Cref{eq:tautantiveering} holds up to possibly multiplying the left hand side by $1+t$ or $1-t$.
\end{thm}

The proof of \Cref{thm:tautantiveering} is deferred to \Cref{sec:factorantiveerpoly} since it involves ideas that do not play a role elsewhere in this paper.

\subsection{Zeta functions} \label{subsec:zetafunctions}

In this subsection, we explain the connection between the previously defined polynomials and zeta functions. This material was explained to us by Jonathan Zung.

Let $G$ be a finitely generated free abelian group.
The \emph{group ring of formal power series} in $G$ (with $\mathbb{Z}$-coefficients) is the ring
$\mathbb{Z}[[G]] = \left\{ \sum_{g \in G} a_g g \mid a_g \in \mathbb{Z} \right\}$.
For $g \in G$, we write $(1-g)^{-1}$ to mean $\sum_{k=0}^\infty g^k$ in $\mathbb{Z}[[G]]$.

Given a basis $\{g_1,\dots,g_n\}$ of $G$, we can identify $\mathbb{Z}[[G]]$ with the multivariable formal power series ring $\mathbb{Z}[[t_1,t_1^{-1},\dots,t_n,t_n^{-1}]]$ by sending $\sum_{i_1,\dots,i_n} a_{g_1^{i_1} \dots g_n^{i_n}} g_1^{i_1} \dots g_n^{i_n}$ to $\sum_{i_1,\dots,i_n} a_{g_1^{i_1} \dots g_n^{i_n}} t_1^{i_1} \dots t_n^{i_n}$.

A \textit{partial flow} on a space $X$ is a continuous map $\phi$ from a closed subset $K \subset \mathbb{R} \times X$ to $X$ such that
\begin{itemize}
    \item $K$ contains $\{0\} \times X$,
    \item $\phi(0,x)=x$ for all $x \in X$, and
    \item $\phi(s,\phi(t,x))=\phi(s+t,x)$ for all $s,t \in \mathbb{R}$, $x \in X$, whenever the two sides are well-defined.
\end{itemize}
More intuitively, a partial flow is a flow except some flow lines may only be well-defined up to some finite (forward or backward) time.
We say that $\phi$ is a \textit{forward semi-flow} if the domain $K$ contains $[0,\infty) \times X$.

Let $G = H_1(X)/\text{Torsion}$. We define the \textit{zeta function} of $\phi$ to be 
$$\zeta_\phi = \prod_{\text{primitive $\gamma$}} (1-[\gamma])^{-1} \in \mathbb{Z}[[G]]$$
where the product is taken over all primitive closed orbits $\gamma$, and $[\gamma] \in G$ denotes the homology class of $\gamma$, provided that the coefficient of each $g \in G$ converges in the (possibly) infinite product. 
See \cite{Fri86}, \cite{Pol20}, and \cite{JZ24} for more on dynamical zeta functions.

Now let $\phi$ be a pseudo-Anosov flow on a closed 3-manifold $Y$.
A \textit{flow box} of $\phi$ is a compact subset $F$ homeomorphic to $I_s \times I_u \times [0,1]$ where:
\begin{itemize}
    \item $I_s$ and $I_u$ are (possibly degenerate) closed intervals.
    \item For each $(s_0,u_0) \in I_s \times I_u$, the oriented segment $\{s_0\} \times \{u_0\} \times [0,1]$ is a segment of a flow line.
    \item For each $s_0 \in [0,1]$, the slice $\{s_0\} \times I_u \times [0,1]$ lies on a leaf of the stable foliation.
    \item For each $u_0 \in [0,1]$, the slice $I_s \times \{u_0\} \times [0,1]$ lies on a leaf of the unstable foliation.
\end{itemize}

We refer to $I_s \times I_u \times \{1\}$ as the \emph{top rectangle}, $I_s \times I_u \times \{0\}$ as the \emph{bottom rectangle}, $\partial I_s \times I_u \times [0,1]$ as the \emph{stable rectangles}, and $I_s \times \partial I_u \times [0,1]$ as the \emph{unstable rectangles} of $F$.
A \emph{$s$-subrectangle} of the top rectangle of $F$ is a subset of the form $J_s \times I_u \times \{1\}$ where $J_s$ is a closed subinterval of $I_s$. A \emph{$u$-subrectangle} of the bottom face of a $F$ is a subset of the form $I_s \times J_u \times \{0\}$ where $J_u$ is a closed subinterval of $I_u$.

A \textit{flow box decomposition} for $\phi$ is a decomposition of $Y$ into finitely many flow boxes $F_1,\dots,F_n$ with disjoint interiors such that the intersection of the top rectangle of $F_i$ and the bottom rectangle of $F_j$ is a (possibly empty) union of $s$-subrectangles of the former and a (possibly empty) union of $u$-subrectangle of the latter, for every $i,j$.
See \Cref{fig:flowboxdecomp} middle.

\begin{figure}
    \centering
    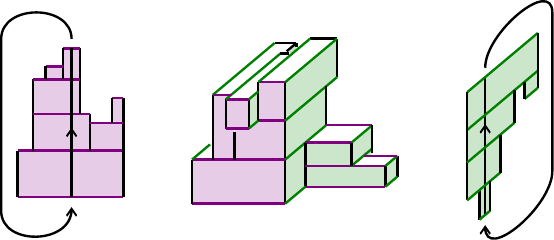
    \caption{A flow box decomposition. Left: An unstable spine with an unstable spine orbit. Middle: The flow boxes. Right: A stable spine with a stable spine orbit.}
    \label{fig:flowboxdecomp}
\end{figure}

Fix a flow box decomposition $\mathcal{F}$.
A \textit{stable spine} is a component of the union of stable rectangles in $Y$. 
See \Cref{fig:flowboxdecomp} right.
Each stable spine lies on stable leaf.
In fact, since there are only finitely many stable rectangles, each stable spine must lie along a periodic stable leaf, i.e. an stable leaf that contains a primitive closed orbit. 
We refer to such a closed orbit as a \textit{stable spine orbit}, and we denote the collection of stable spine orbits as $\mathcal{C}^s$. 
We define the \textit{unstable spines} and \textit{unstable spine orbits} in a symmetric way.
See \Cref{fig:flowboxdecomp} left.

The flow boxes in $\mathcal{F}$ are glued along their stable rectangles, along their unstable rectangles, and along their top and bottom rectangles. 
We define the \textit{unstable ungluing} of $\mathcal{F}$ to be the subset $D^u \mathcal{F} \subset Y$ obtained by ungluing the unstable rectangles.
See \Cref{fig:unstableunglue}.
Note that $D^u \mathcal{F}$ is homotopy equivalent to the complement of the unstable spine orbits $M = Y \backslash \nu(\mathcal{C}^u)$.
Going forward, we will always implicitly identify $H_1(D^u \mathcal{F})/\text{Torsion}$ with $H_1(M)/\text{Torsion}$.

\begin{figure}
    \centering
    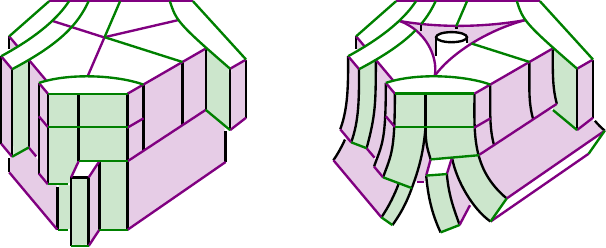
    \caption{Left: A flow box decomposition $\mathcal{F}$ near an unstable spine orbit. Right: The unstable ungluing of $\mathcal{F}$.}
    \label{fig:unstableunglue}
\end{figure}

There is a natural map $D^u \mathcal{F} \to Y$ obtained by regluing the unstable rectangles.
Pulling back the flow lines of $\phi$, we have a forward semi-flow $D^u \phi$ on $D^u \mathcal{F}$.
Here, some backward flow lines are not well-defined because they run into the edges where two unstable rectangles meet along their top edges, where they `split apart' into two flow lines lying on the two unstable rectangles.
We denote the collection of primitive closed orbits of $D^u \phi$ that are the preimages of the stable/unstable spine orbits as $D^u \mathcal{C}^{s/u}$.

Let us compare the (semi-)flows $D^u \phi$ and the blow-up $\phi^\sharp$ of $\phi$ along $\mathcal{C}$.
As pointed out above, these flows are defined on homotopy equivalent spaces.
They share the same primitive closed orbits, with the same corresponding homology class, except: 
For each $\gamma \in \mathcal{C}^u$, let $\gamma_1,\dots,\gamma_k$ be the subset of elements in $D^u \mathcal{C}^u$ that map to $\gamma$. Then all the $\gamma_i$ have the same homotopy class, while $\phi^\sharp$ has $2k$ many primitive closed orbits of this homotopy class on $\partial \nu(\gamma)$. Thus 
\begin{equation} \label{eq:zetablowuptounstable}
\zeta_{\phi^\sharp} = \zeta_{D^u \phi} \prod_{\gamma \in D^u \mathcal{C}^u} (1-[\gamma])^{-1}.
\end{equation}

Meanwhile, we define the \emph{double ungluing} of $\mathcal{F}$ to be the subset $D^2 \mathcal{F} \subset Y$ obtained by ungluing both the stable and unstable rectangles.
See \Cref{fig:doubleunglue}.
There is a natural map $D^2 \mathcal{F} \to D^u \mathcal{F}$ obtained by regluing the stable rectangles. 
Pulling back the flow lines of $D^u \phi$, we have a partial flow $D^2 \phi$.
We denote the collection of primitive closed orbits of $D^2 \phi$ that are the preimages of the orbits in $D^u \mathcal{C}^{s/u}$ as $D^2 \mathcal{C}^{s/u}$.
(Equivalently, one can also define $D^2 \phi$ by pulling back $\phi$ along the map $D^2 \mathcal{F} \to Y$ that reglues the stable and unstable rectangles, and define $D^2 \mathcal{C}^{s/u}$ to be the collection of primitive closed orbits that are the preimages of the stable/unstable spine orbits.)

\begin{figure}
    \centering
    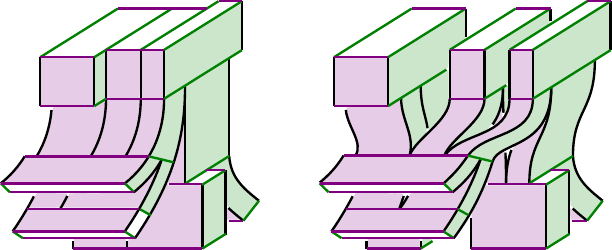
    \caption{Left: The unstable ungluing of a flow box decomposition $\mathcal{F}$. Right: The double ungluing of $\mathcal{F}$.}
    \label{fig:doubleunglue}
\end{figure}

Let us compare the partial flows $D^2 \phi$ and $D^u \phi$.
These flows are not defined on homotopy equivalent spaces (in general). However, $D^2 \mathcal{F}$ naturally embeds in $M = Y \backslash \nu(\mathcal{C}^u)$, thus we have a homomorpism $\iota: H_1(D^2 \mathcal{F})/\text{Torsion} \to H_1(M)/\text{Torsion}$.
The flows share the same primitive closed orbits, with the same corresponding homology class (once we apply $\iota$ to the classes of $D^2 \phi$), except:
Each $\gamma \in D^u \mathcal{C}^s$ is either (i) mapped homeomorphically onto by two primitive closed orbits $\widehat{\gamma}_1, \widehat{\gamma_2}$ of $D^2 \phi$ in the same homotopy class in $M$, or (ii) double covered by a primitive closed orbit $\widehat{\gamma}$ of $D^2 \phi$ in the homotopy class $[\gamma]^2$ in $M$.
We say that $\gamma$ is \emph{orientation-preserving} in case (i) and \emph{orientation-reversing} in case (ii).
Thus
\begin{equation} \label{eq:zetaunstabletodouble}
\zeta_{D^2 \phi} = \zeta_{D^u \phi} \prod_{\text{$\gamma \in D^u \mathcal{C}^s$ ori. pres.}} (1-[\gamma])^{-1} \prod_{\text{$\gamma \in D^u \mathcal{C}^s$ ori. rev.}} (1+[\gamma])^{-1}.
\end{equation}

To explain the connection between zeta functions and the polynomial invariants in the previous subsection, we have to apply these formulas to a particular flow box decomposition.

Let $\mathcal{C}$ be a finite collection of closed orbits such that $\phi$ has no perfect fits relative to $\mathcal{C}$. 
Let $B$ be the corresponding veering branched surface on the 3-manifold $M=Y \backslash \nu(\mathcal{C})$.
As before, we write $G=H_1(M)/\text{Torsion}$.

\begin{prop}[Zung] \label{prop:zetadoubleveering}
There exists a flow box decomposition $\mathcal{F}$ for $\phi$ such that:
\begin{enumerate}
    \item The unstable spine orbits are the orbits in $\mathcal{C}$.
    \item The orbits in $D^u \mathcal{C}^s$ are the ones homotopic to the anti-branch loops of $B$.
    \item The zeta function $\zeta_{D^2 \phi}$ of the partial flow on the double ungluing is well-defined. The image of $\zeta_{D^2 \phi}$ in $\mathbb{Z}[[G]]$ is the reciprocal of the veering polynomial of $B$.
\end{enumerate}
\end{prop}
\begin{proof}
The \textit{flow graph} of a veering branched surface $B$ is a directed graph introduced by Landry-Minsky-Taylor \cite{LMT24} that is embedded in $B$ and that can be defined from the combinatorics of $B$. This is different from the dual graph, and is commonly denoted as $\Phi$. The exact definition of it is not important for the purposes of this paper. The important fact is that $\Phi$ encodes a flow box decomposition, in the sense that one can obtain a flow box decomposition $\mathcal{F}$ by thickening each edge of $\Phi$ into a flow box and connecting their top and bottom rectangles appropriately at the vertices. This is shown in \cite[Section 5]{AT24}, but see \cite[Chapter 2]{Tsathesis} for a more updated exposition, where items (1) and (2) are also shown.

It remains for us to explain item (3). Note that $D^2 \mathcal{F}$ is homotopy equivalent to $\Phi$. Going forward, we will always implicitly identify $H_1(D^2 \mathcal{F})/\text{Torsion}$ with $H_1(\Phi)/\text{Torsion} = H_1(\Phi)$. Since $\Phi$ encodes $\mathcal{F}$, the primitive closed orbits of $D^2 \phi$ correspond exactly to the primitive cycles of $\Phi$. Thus $\zeta_{D^2 \phi} = \prod_{\text{primitive cycles $c$}} (1-[c])^{-1}$.

The key fact now is that this last expression can be expressed as the Perron polynomial of $\Phi$. This is essentially a classical result of Bowen-Lanford \cite{BL70}, but we will include an explanation.

Take the obvious simplicial structure on $\Phi$ as a graph. Consider the group $C_1(\Phi)$ of 1-chains on $\Phi$. Note that $H_1(\Phi)$ is a subspace of $C_1(\Phi)$. 
The adjacency matrix of $\Phi$ is the matrix $A \in M_{V(\Phi) \times V(\Phi)}(\mathbb{Z}[C_1(\Phi)])$ defined by $A_{vu} = \sum_{\text{edges $e:u \to v$}} e$. 
The \emph{Perron polynomial} of $\Phi$ is the determinant $P_\Phi = \det(I-A)$. A priori $P_\Phi$ is an element of $\mathbb{Z}[C_1(\Phi)]$, but since every product involves cycles of edges, $P_\Phi$ is in fact an element of $\mathbb{Z}[H_1(\Phi)]$.

We compute
\begingroup
\allowdisplaybreaks
\begin{align*}
-\log \det(I-A) &= \sum_{n=1}^\infty \frac{1}{n} \tr A^n \\
&= \sum_{n=1}^\infty \frac{1}{n} \sum_{\text{length $n$ based cycles $c$}} c \\
\multicolumn{2}{l}{\text{where a \textit{based cycle} is a cycle along with a vertex of the underlying primitive cycle}} \\
&= \sum_{n=1}^\infty \frac{1}{n} \sum_{d|n} \sum_{\text{length $d$ based primitive cycles $c$}} c^{\frac{n}{d}} \\
\multicolumn{2}{l}{\text{by writing each based cycle as a multiple of its underlying based primitive cycle}} \\
&= \sum_{n=1}^\infty \frac{1}{n} \sum_{d|n} d \sum_{\text{length $d$ primitive cycles $c$}} c^{\frac{n}{d}} \\
\multicolumn{2}{l}{\text{since there are $d$ ways to choose a vertex on a length $d$ primitive cycle}} \\
&= \sum_{d=1}^\infty \sum_{m=1}^\infty \frac{1}{m} \sum_{\text{length $d$ primitive cycles $c$}} c^m \\
\multicolumn{2}{l}{\text{where we write $n=md$}} \\
&= \sum_{\text{primitive cycles $c$}} \sum_{m=1}^\infty \frac{1}{m} c^m \\
&= - \sum_{\text{primitive cycles $c$}} \log (1-c).
\end{align*}
\endgroup
Taking the exponential of both sides gives $P_\gamma^{-1} = \zeta_\Phi$.

Finally, we can conclude by \cite[Theorem 4.8]{LMT24}, which states that the veering polynomial of $B$ is the image of the Perron polynomial in $\mathbb{Z}[[G]]$.
\end{proof}

\begin{prop}[Zung] \label{prop:zetavbstaut}
Suppose $b_1(M) \geq 2$. Then the zeta function $\zeta_{D^u \phi}$ of the partial flow on the unstable ungluing is the reciprocal of the taut polynomial of $B$.

If instead $b_1(M)=1$, let $t$ be a generator of $G$. Then the same result holds up to possibly multiplying $\zeta_{D^u \phi}$ by $(1 \pm t)^{-1}$.
\end{prop}
\begin{proof}
By \Cref{prop:zetadoubleveering}(2), the orbits in $D^u \mathcal{C}^s$ are the ones homotopic to the anti-branch loops.
Moreover, since $B$ carries the blown-up unstable foliation (\Cref{thm:vbspaflowcorr}(i)), such a orbit is orientation-preserving if and only if the corresponding anti-branch loop is orientation-preserving.
Thus the proposition follows from \Cref{prop:zetadoubleveering}(3), \Cref{thm:tautveering}, and \Cref{eq:zetaunstabletodouble}.
\end{proof}

\begin{prop} \label{prop:zetablowupantiveering}
Suppose $b_1(M) \geq 2$. Then the zeta function $\zeta_{\phi^\sharp}$ of the blown-up flow is the reciprocal of the anti-veering polynomial of $B$.

If instead $b_1(M)=1$, let $t$ be a generator of $G$. Then the same result holds up to possibly multiplying $\zeta_{\phi^\sharp}$ by $(1 \pm t)^{-1}$ and/or $1 \pm t$.
\end{prop}
\begin{proof}
The orbits in $D^u \mathcal{C}^u$ are the ones homotopic to the repelling closed orbits of $\phi^\sharp$ on $\partial M$, which are in turn homotopic to the branch loops of $B$ (\Cref{thm:vbspaflowcorr}(i)).
Thus the proposition follows from \Cref{prop:zetavbstaut}, \Cref{thm:tautantiveering}, and \Cref{eq:zetablowuptounstable}.
\end{proof}

\section{Factorization of the anti-veering polynomial} \label{sec:factorantiveerpoly}

The goal of this section is to show \Cref{thm:tautantiveering}.
Our proof will follow the proof of \Cref{thm:tautveering} in \cite{LMT24} closely.
Throughout this section, we fix a veering branched surface $B$.
We let $b_1,\dots,b_n$ be the branch loops of $B$, and write $h_i =[b_i] \in H_1(M)/\text{Torsion} =: G$ for the homology class of $b_i$.

\subsection{An alternate presentation of $\mathcal{M}_\Theta$}

The starting point of relating the taut and anti-veering polynomials is the following lemma.

\begin{lemma} \label{lemma:alttautmodule}
The taut module $\mathcal{M}_\Theta$ equals the cokernel of $D_\Theta \oplus D_A:\mathcal{E} \oplus \mathcal{T} \to \mathcal{S}$.
\end{lemma}
\begin{proof}
\Cref{lemma:facereldiff} implies that $D_A(\mathcal{T}) \subset D_\Theta(\mathcal{E})$ from which the lemma follows.
\end{proof}

One direction of \Cref{thm:tautantiveering} follows quickly from this presentation.

\begin{prop} \label{prop:factorantiveeringeasydir}
We have $A \mid \Theta \cdot \prod_i (1-h_i)$.
\end{prop}
\begin{proof}
For each $i$, let $e_{i,1}, \dots, e_{i,l_i}$ be the edges of $b_i$, in that order, and pick edges $\widehat{e}_{i,j}$ that map to $e_{i,j}$, in a way such that the terminal vertex of $\widehat{e}_{i,j}$ equals the initial vertex of $\widehat{e}_{i,j+1}$, which we denote as $\widehat{v}_{i,j}$, for $j=1,\dots,l_i$.

Then given any basis $\beta_{\mathcal{T}}$ of $\mathcal{T}$, since $\widehat{v}_{i,j} \in \mathcal{T}$ for all $i,j$,
$$\beta_{\mathcal{E} \oplus \mathcal{T}} = \{\widehat{e}_{i,1} \mid i=1,\dots,n\} \cup \{ \widehat{e}_{i,1}-\widehat{e}_{i,j}-\sum_{k=1}^{j-1} \widehat{v}_{i,k} \mid i=1,\dots,n, j=2,\dots,l_i\} \cup \beta_{\mathcal{T}}$$
is a basis for $\mathcal{E} \oplus \mathcal{T}$.

But by applying \Cref{lemma:facereldiff}, we have
$$D_\Theta(\widehat{e}_{i,1}-\widehat{e}_{i,j}) = D_\Theta(\widehat{e}_{i,1}) - D_\Theta(\widehat{e}_{i,j}) = \sum_{k=1}^{j-1} (D_\Theta(\widehat{e}_{i,k}) - D_\Theta(\widehat{e}_{i,k+1})) = \sum_{k=1}^{j-1} D_A(\widehat{v}_{i,k}).$$
Hence the map $D_\Theta \oplus D_A:\mathcal{E} \oplus \mathcal{T}$ takes each $\widehat{e}_{i,1}-\widehat{e}_{i,j}-\sum_{k=1}^{j-1} \widehat{v}_{i,k}$ to $0$.

In other words, the nonzero columns of $[D_\Theta \oplus D_A]^{\beta_{\mathcal{E} \oplus \mathcal{T}}}_{\beta_\mathcal{S}}$ forms the submatrix
$$[D'] = \begin{bmatrix} [D_\Theta(\widehat{e}_{1,1})]_{\beta_\mathcal{S}} & \cdots & [D_\Theta(\widehat{e}_{n,1})]_{\beta_\mathcal{S}} & [D_A]^{\beta_\mathcal{T}}_{\beta_\mathcal{S}} \end{bmatrix}$$
and 
\begin{equation} \label{eq:factorantiveeringeasydir1}
\Theta = \Fit([D_\Theta \oplus D_A]^{\beta_{\mathcal{E} \oplus \mathcal{T}}}_{\beta_\mathcal{S}}) = \Fit([D']).
\end{equation}

Meanwhile, consider the matrix 
$$[D''] = \begin{bmatrix} (1-h_1) \cdot [D_\Theta(\widehat{e}_{1,1})]_{\beta_\mathcal{S}} & \cdots & (1-h_n) \cdot [D_\Theta(\widehat{e}_{n,1})]_{\beta_\mathcal{S}} & [D_A]^{\beta_\mathcal{T}}_{\beta_\mathcal{S}} \end{bmatrix}.$$
Applying \Cref{lemma:facereldiff} as above, we have 
$$(1-h_i) \cdot D_\Theta(\widehat{e}_{i,1}) = D_\Theta(\widehat{e}_{i,1}) - D_\Theta(h_i \widehat{e}_{i,1}) = \sum_{k=1}^{l_i} D_A(\widehat{v}_{i,k}) \in D_A(\mathcal{T})$$
for each $i$.
Thus the determinant of any $\rank(\mathcal{T}) \times \rank(\mathcal{T}) =: t \times t$ submatrix of $[D'']$ is a multiple of $\det [D_A]^{\beta_\mathcal{T}}_{\beta_\mathcal{S}}$.
In fact, since $[D_A]^{\beta_\mathcal{T}}_{\beta_\mathcal{S}}$ is a $t \times t$ submatrix of $[D^\sharp]$, we have 
\begin{equation} \label{eq:factorantiveeringeasydir2}
\Fit([D'']) = \det [D_A]^{\beta_\mathcal{T}}_{\beta_\mathcal{S}} = A.
\end{equation}

Meanwhile, every $t \times t$ submatrix of $[D'']$ is a $t \times t$ submatrix of $[D']$ with some columns multiplied by $1-h_i$, thus 
\begin{equation} \label{eq:factorantiveeringeasydir3}
\Fit([D'']) \mid \Fit([D']) \cdot \prod (1-h_i).
\end{equation}

The lemma now follows from combining Equations \eqref{eq:factorantiveeringeasydir1}, \eqref{eq:factorantiveeringeasydir2}, and \eqref{eq:factorantiveeringeasydir3}.
\end{proof}

\subsection{Factorization with branch-simple loops}

To show the other direction of \Cref{thm:tautantiveering}, we have to introduce new ideas.

Observe that each loop $c$ of $G$ is either a multiple of a branch loop, or it can be uniquely written as a concatenation $c_1*\dots*c_k$ where each $c_i$ is a maximal subsegment of a distinct branch loop, i.e. such that $c_i*c_{i+1}$ takes a turn at the terminal vertex of $c_i$. We refer to the value of $k$ as the \emph{complexity} of $c$.
For notational consistency, if $c$ is a multiple of a branch loop, we define its complexity to be $0$.

We say that a loop $c$ of $G$ is \emph{branch-simple} if it is embedded and in the decomposition $c=c_1*\dots*c_k$ above, each $c_i$ lies along a distinct branch loop.
We also define branch loops to be branch-simple.

\begin{lemma} \label{lemma:branchsimplefactor}
Let $c$ be a branch-simple loop of complexity $k$. Then $\Theta \cdot \prod (1-h_i) \mid A \cdot (1-(-1)^k[c])$. 
\end{lemma}
\begin{proof}
Recall that $b_1,...,b_n$ are the branch cycles of $B$. 
For each $i=1,\dots,n$, let $e_{i,1}, \dots, e_{i,l_i}$ be the edges of $b_i$, in that order, and let $v_{i,j}$ be the terminal vertex of $e_{i,j}$. 
Up to relabeling, we can assume $c = c_1 * \dots * c_k$ where $c_i = e_{i,1} * \dots * e_{i,m_i}$ for some $m_i < l_i$.

For $i=k+1,\dots,n$, we claim that, up to relabeling the branch loops, there is a vertex $w_i$ that lies on $b_i$ and another branch loop $b_j$ for $j < i$. 

If $n=k$, then the claim is vacuous, so we assume $n \geq k+1$. Since $G$ is connected, $\bigcup_{j=1}^k b_j$ must meet $\bigcup_{j=k+1}^n b_j$ at some vertex $w_{k+1}$. Up to relabeling $b_{k+1},\dots,b_n$, we can assume $w_{k+1}$ lies in $b_{k+1}$. If $n=k+1$ then we are done, otherwise $\bigcup_{j=1}^{k+1} b_j$ must meet $\bigcup_{j=k+2}^n b_j$ at some vertex $w_{k+2}$, and we repeat the argument inductively.

Up to relabeling, we can assume that $v_{i,1}=w_i$.

For each vertex $v$ of $B$, we choose a lift $\widehat{v}$ of $v$, and for each edge $e$ of $B$, we choose a lift $\widehat{e}$ of $e$.
We can choose these lifts so that 
\begin{itemize}
    \item for $i=1,\dots,n$, $j=1,\dots,l_i$, $\widehat{v_{i,j}}$ is the terminal vertex of $\widehat{e_{i,j}}$, and
    \item for $i=1,\dots,k-1$, $\widehat{v_{i,m_i}}$ is the initial vertex of $\widehat{e_{i+1,1}}$.
\end{itemize}
See \Cref{fig:branchsimplefactorlift}.

\begin{figure}
    \centering
    \selectfont\fontsize{6pt}{6pt}
    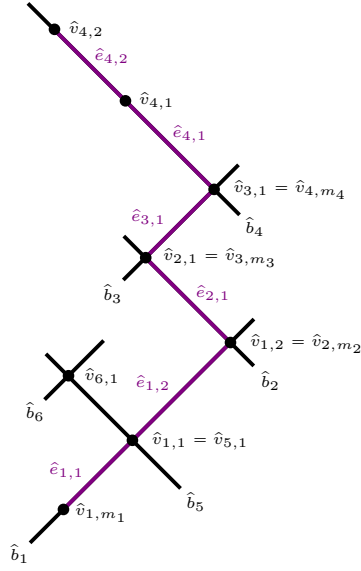
    \caption{Choosing lifts for vertices and edges for the proof of \Cref{lemma:branchsimplefactor}.}
    \label{fig:branchsimplefactorlift}
\end{figure}

Then $\beta_\mathcal{T} = \{\widehat{v} \mid \text{vertex $v$ of $B$}\}$ is a basis of $\mathcal{T}$. 
We partition $\beta_\mathcal{T}$ into $\beta'_\mathcal{T} \sqcup \beta''_\mathcal{T} := \{\widehat{v_{1,1}},\dots,\widehat{v_{n,1}}\} \sqcup \{\widehat{v} \mid v \neq v_{i,1}\}$, which induces a splitting $\mathcal{T} = \mathcal{T}' \oplus \mathcal{T}''$.
We write $D'_A = D_A|_{\mathcal{T}'}$ and $D''_A = D_A|_{\mathcal{T}''}$.
Similarly, $\beta_\mathcal{E} = \{\widehat{e} \mid \text{edge $e$ of $B$}\}$ is a basis of $\mathcal{E}$. 
We partition $\beta_\mathcal{E}$ into $\beta'_\mathcal{E} \sqcup \beta''_\mathcal{E} := \{\widehat{e_{1,1}},\dots,\widehat{e_{n,1}}\} \sqcup \{\widehat{e} \mid e \neq e_{i,1}\}$.

Now consider the matrix 
$$[D'_\Theta \oplus D''_A]_{\beta_\mathcal{S}}^{\beta'_\mathcal{E} \sqcup \beta''_\mathcal{T}} = \begin{bmatrix} D_\Theta(\widehat{e_{1,1}}) & \cdots & D_\Theta(\widehat{e_{n,1}}) & [D''_A]_{\beta_\mathcal{S}}^{\beta''_\mathcal{T}} \end{bmatrix}.$$
This is a $\rank(\mathcal{T}) \times \rank(\mathcal{T}) =: t \times t$ submatrix of $[D_\Theta \oplus D_A]_{\beta_\mathcal{S}}^{\beta_\mathcal{E} \sqcup \beta_\mathcal{T}}$.
Thus 
\begin{equation} \label{eq:branchsimplefactor1}
\Theta = \Fit(D_\Theta \oplus D_A) \mid \det([D_\Theta \oplus D_A]_{\beta_\mathcal{S}}^{\beta'_\mathcal{E} \sqcup \beta''_\mathcal{T}})
\end{equation}

Meanwhile, consider the matrix 
$$[D'''] = \begin{bmatrix} (1-h_1) \cdot D_\Theta(\widehat{e_1}) & \cdots & (1-h_n) \cdot D_\Theta(\widehat{e_n}) & [D''_A]_{\beta_\mathcal{S}}^{\beta''_\mathcal{T}} \end{bmatrix}$$
obtained by multiplying the $i^{\text{th}}$ column of the previous matrix by $1-h_i$, for $i=1,\dots,n$.
We have 
\begin{equation} \label{eq:branchsimplefactor2}
\det [D'''] = \det ([D'_\Theta \oplus D''_A]_{\beta_\mathcal{S}}^{\beta'_\mathcal{E} \sqcup \beta''_\mathcal{T}}) \cdot \prod_{i=1}^n (1-h_i).
\end{equation}

For the first column, we can apply \Cref{lemma:facereldiff} as in \Cref{prop:factorantiveeringeasydir} to get
\begin{footnotesize}
\begin{align*}
(1-h_1) \cdot D_\Theta(\widehat{e_1}) &= \sum_{j=1}^{l_1} D_A(\widehat{v_{1,j}}) \\
&\in D_A(\widehat{v_{1,1}}) + D_A(\widehat{v_{1,m_1}}) + \spn(\{D_A(\widehat{v_{k+1,1}}),\dots,D_A(\widehat{v_{n,1}})\} \cup D_A(\beta''_\mathcal{T})) \\
&= D_A([c]^{-1} \widehat{v_{k,m_k}}) + D_A(\widehat{v_{1,m_1}}) + \spn(\{D_A(\widehat{v_{k+1,1}}),\dots,D_A(\widehat{v_{n,1}})\} \cup D_A(\beta''_\mathcal{T}))
\end{align*}
\end{footnotesize}

For each of the next $k-1$ columns, we get
\begin{footnotesize}
\begin{align*}
(1-h_i) \cdot D_\Theta(\widehat{e_i}) &= \sum_{j=1}^{l_i} D_A(\widehat{v_{i,j}}) \\
&\in D_A(\widehat{v_{i,1}}) + D_A(\widehat{v_{i,m_i}}) + \spn(\{D_A(\widehat{v_{k+1}}),\dots,D_A(\widehat{v_n})\} \cup D_A(\beta''_\mathcal{T})) \\
&= D_A(\widehat{v_{i-1,m_{i-1}}}) + D_A(\widehat{v_{i,m_i}}) + \spn(\{D_A(\widehat{v_{k+1,1}}),\dots,D_A(\widehat{v_{n,1}})\} \cup D_A(\beta''_\mathcal{T}))
\end{align*}
\end{footnotesize}

For the next $n-k$ columns, we get
\begin{align*}
(1-h_i) \cdot D_\Theta(\widehat{e_i}) &= \sum_{j=1}^{l_i} D_A(\widehat{v}_{i,j}) \\
&\in D_A(\widehat{v_{i,1}}) + \spn(\{D_A(\widehat{v_{i+1,1}}),\dots,D_A(\widehat{v_{n,1}})\} \cup D_A(\beta''_\mathcal{T}))
\end{align*}
where we have used the defining property of $v_{j,1} = w_j$ to know that $\widehat{v_{i,2}},\dots,\widehat{v_{i,l_i}}$ do not map to $\widehat{v_{j,1}}$ for $j < i$.

Using elementary column operations, we can thus reduce $[D''']$ to 
\begin{small}
\begin{multline*}
\left[ \begin{matrix}
D_A([c]^{-1} \widehat{v_{k,m_k}}) + D_A(\widehat{v_{1,m_1}}) & D_A(\widehat{v_{1,m_1}}) + D_A(\widehat{v_{2,m_2}}) & \cdots & D_A(\widehat{v_{k-1,m_{k-1}}}) + D_A(\widehat{v_{k,m_k}})
\end{matrix} \right.\\
\left. \begin{matrix}
D_A(\widehat{v_{k+1,1}}) & \cdots & D_A(\widehat{v_{n,1}}) & [D''_A]_{\beta_\mathcal{S}}^{\beta''_\mathcal{T}}
\end{matrix} \right]
\end{multline*}
\end{small}

Thus 
\begin{small}
\begin{equation} \label{eq:branchsimplefactor3}
\begin{aligned}
&\det [D'''] \\
&= (1-(-1)^k [c]^{-1}) \det \begin{bmatrix} D_A(\widehat{v_{1,1}}) & \cdots & D_A(\widehat{v_{k,1}}) & D_A(\widehat{v_{k+1,1}}) & \cdots & D_A(\widehat{v_{n,1}}) & [D''_A]_{\beta_\mathcal{S}}^{\beta''_\mathcal{T}} \end{bmatrix} \\
&= (1-(-1)^k [c]^{-1}) \det ([D_A]_{\beta_\mathcal{S}}^{\beta_\mathcal{T}}) \\
&= (1-(-1)^k [c]^{-1}) \cdot A
\end{aligned}
\end{equation}
\end{small}

The lemma follows from Equations \eqref{eq:branchsimplefactor1}, \eqref{eq:branchsimplefactor2}, and \eqref{eq:branchsimplefactor3}.
\end{proof}

\subsection{Abundance of branch-simple loops}

The last step is to show that there are many branch-simple loops.

\begin{lemma} \label{lemma:manybranchsimple}
The group $H_1(\Gamma)$ is generated by branch-simple loops.
\end{lemma}
\begin{proof}
It is a general fact that for any strongly connected directed graph $G$, $H_1(G)$ is generated by the loops of $G$.
Hence it suffices to show that every loop of $G$ can be written as a linear combination of branch-simple loops. We do this by induction on the complexity.

If a loop $c$ has complexity $0$, then it is a multiple of a branch loop, and the statement is clear.
If a loop $c$ has complexity $1$, then it consists of one path $c_1$ along a branch loop $b_1$. Up to subtracting multiples of $b_1$, we can assume that $c_1$ is embedded, at which point $c$ is branch-simple.

Inductively, suppose a loop $c$ has complexity $k$. Write $c$ as a concatenation $c_1*\dots*c_k$ where each $c_i$ lies along a branch loop $b_i$. Up to subtracting multiples of $b_i$, we can assume that each $c_i$ is embedded.
If the $b_i$ are distinct and if $c$ is embedded, then $c$ is branch-simple itself.

Suppose otherwise that some of the $b_i$ coincide, say $b_1=b_j$, then we perform the following operation: 
Let $c'$ be the loop obtained by starting at the initial vertex of $c_1$, traveling along $b_1$ to the terminal vertex of $c_j$, then travel along $b$ until we return to the initial vertex of $c_1$. 
Let $c''$ be the loop obtained by starting at the initial vertex of $c_j$, traveling along $b_1$ to the terminal vertex of $c_1$, then travel along $b$ until we return to the initial vertex of $c_j$.

Depending on the configurations of $c_1$ and $c_j$ along $b_1$, we have $[c']+[c'']=[c]+[b]$, $[c']+[c'']=[c]-[b]$, or $[c']+[c'']=[c]$.
One can also check that the complexity of $c'$ equals $j-1$, while that of $c'' $ equals $k-j+1$. 
Hence we are done by the induction hypothesis.

The remaining case is if the $b_i$ are all distinct, but $c$ is not embedded. In this case, let $v$ be a self-intersection vertex of $c$. 
Since all $c_i$ are embedded, and all $b_i$ are distinct, $c$ passes through $v$ exactly twice, and at both occurrences $v$ lies in the interior of some $c_i$, say $v$ lies in the interior of $c_1$ and $c_j$.

We perform a similar operation as above:
Let $c'$ be the loop obtained by starting at $v$, traveling along $b$ in the direction of $c_1$ until we return to $v$ (but from the direction of $c_j$).
Let $c''$ be the loop obtained by starting at $v$, traveling along $b$ in the direction of $c_j$ until we return to $v$ (but from the direction of $c_1$).
In other words, we cut-and-paste $c$ at $v$.

Then $[c']+[c'']=[c]$, the complexity of $c'$ equals $j$, and the complexity of $c''$ equals $k-j+2$.
If $j \neq 2,k$, then we are done by the induction hypothesis.

Otherwise, we note that the lengths of $c'$ and $c''$ are strictly less than that of $c$. Thus by repeating this operation, we either eventually reduce to curves of smaller complexity, or embedded curves of the same complexity $k$. In the latter case, the embedded curves are branch-simple.
\end{proof}

\begin{lemma} \label{lemma:branchsimplerelprime}
Let $G$ be a finitely generated free abelian group. Suppose $Z \subset G$ is a subset that generates $G$, and suppose for every $z \in Z$, there is an associated sign $\epsilon_z = 1$ or $-1$. 

Suppose $\rank(G) \geq 2$. Then $\gcd\{1-\epsilon_z z \mid z \in Z\} = 1$.

If instead $\rank(G) \geq 1$, let $t$ be a generator of $G$. Then $\gcd\{1-\epsilon_z z \mid z \in Z\}$ equals $1$, $1+t$, or $1-t$.
\end{lemma}
\begin{proof}
This is a straightforward exercise in algebra. See \cite[Lemma 6.17 and Remark 6.18]{LMT24} for a proof.
\end{proof}

\begin{prop} \label{prop:factorantiveeringharddir}
Suppose $\rank(G) \geq 2$. Then $\Theta \cdot \prod_i (1-h_i) \mid A$.

If instead $\rank(G) \geq 1$, let $t$ be a generator of $G$. Then $\Theta \cdot \prod_i (1-h_i) \mid A$ or $A \cdot (1+t)$ or $A \cdot (1-t)$.
\end{prop}
\begin{proof}
This follows from \Cref{lemma:branchsimplefactor}, \Cref{lemma:manybranchsimple}, and \Cref{lemma:branchsimplerelprime}.
\end{proof}

\Cref{thm:tautantiveering} follows from \Cref{prop:factorantiveeringeasydir} and \Cref{prop:factorantiveeringharddir}.

\section{Categorification of the anti-veering polynomial} \label{sec:sfhcategorifies}

Let $B$ be a veering branched surface on a 3-manifold $M$ with torus boundary. For the duration of this section, we fix a numbering on the set of sectors $\{S_1,\dots,S_n\}$ and fix a numbering on the set of triple points $\{v_1,\dots,v_n\}$.

Recall that each state $\x$ is a bijective mapping $\{S_1,\dots,S_n\} \to \{v_1,\dots,v_n\}$ such that $\x(S_i)$ is a corner of the sector $S_i$. 
(There are four possible choices for each $\x(S_i)$ but not all choices give rise to a bijection $\{S_1,\dots,S_n\} \to \{v_1,\dots,v_n\}$ and hence to a generator.)   
Using our numbering, we can associate to $\x$ an element of the symmetric group $\mathfrak{S}_n$.
In particular, we can make sense of the \emph{sign} of $x$, $\sgn x \in \mathbb{Z}/2$.

\begin{defn}[The $\nu$-grading]
For a generator $\x$ we define 
$$\nu(\x) = \# \{S_i \mid \text{$\x(S_i)$ is a side corner of $S_i$}\} + \sgn \x \in \mathbb{Z}/2$$
and extend $\nu$ into a $\mathbb{Z}/2$-grading on $SFC$.
\end{defn}

In general, if $\nu$ is a $\mathbb{Z}/2$-grading on a vector space $V$, then we denote by $V^\even$ the subspace of $V$ consisting of the elements that have grading $0$ and $V^\odd$ the subspace of $V$ consisting of the elements that have grading $1$. The \emph{Euler characteristic} $\chi_\nu(V)$ of $V$ (with respect to $\nu$) is defined to be $\dim V^\even - \dim V^\odd \in \mathbb{Z}$.

The goal of this section is to prove the following theorem.

\begin{thm} \label{thm:SFHcategorifiesantiveering}
Let $B$ be a veering branched surface for a 3-manifold with torus boundary $M$.  Denote by  $\Gamma$ the sutured structure induced by $B$ on $\partial M$, then $SFH(M,\Gamma)$ equipped with the $\nu$-grading induced by $B$ categorifies the veering polynomial $A_B$. That is  
$$A_B = \sum_{\s \in \text{Spin}^c(M,\Gamma)} \chi_\nu(SFH(M,\Gamma, \s ))  \cdot \s$$
up to multiplication by a unit in the group ring $\mathbb{Z}[G]$.
\end{thm}

Once we prove \Cref{thm:SFHcategorifiesantiveering}, \Cref{thm:introzeta} will follow from \Cref{prop:zetablowupantiveering}.

\subsection{Relating $\mu$ and $\nu$}

The key to proving \Cref{thm:SFHcategorifiesantiveering} is the following result.

\begin{thm} \label{prop:munumod2}
In the Heegaard diagaram of a veering branched surface, let $D$ be an effective domain connecting a state $\x$ to a state $\y$. If $n_\z(D) = 0$ then \[\nu(\x)-\nu(\y) \equiv \mu(D) \pmod{2} \ .\]
\end{thm}

We start with some terminology. Let $(\Sigma, \boa, \bob)$  be the sutured Heegaard diagram associated to the veering branched surface $B$, and $D$ an  effective domain with  $n_\z(D) = 0$.
Then, by \Cref{prop:effdomainemb}, $D$ is embedded (i.e. its coefficients are either $0$ or $1$) and as in \Cref{rmk:embdomainssubsurfaces} we can imagine it as a subsurface of $\Sigma$ with \emph{convex} and \emph{concave corners}. 

More precisely, a corner of $D$ is the set of points where the boundary switches between lying on a $\alpha$-curve and lying on $\beta$-curve. A corner is concave if $D$ occupies three quadrants at the point, and is convex otherwise. A \emph{side} of $D$ is a component of the complement of the corners on $\partial D$. We define the \emph{length} of a side to be the number of points in $\alpha \cap \beta$ that is contained in its interior plus one.

\begin{prop} \label{prop:corners}
If $D$ is an effective domain connecting $\x$ to $\y$. Then the coordinates of $\x$ and $\y$ occupy alternatingly the corners of  $D$, and none of the shows up in the interior of the sides.
\end{prop}
\begin{proof}
Let $c$  be a boundary component of the subsurface specified by $D$.
First of all note that by \Cref{lemma:nofullalphabetacurve}, $c$ cannot be a full $\alpha$- or $\beta$-curve. Thus $c$ has corners and has arcs of $\alpha$- and $\beta$-curves as sides. 
From the equations $\partial \partial_{\boa} D= \y-\x$ and $\partial \partial_{\bob} D= \x-\y$ from \Cref{defn:initialfinalstate}, it descends that the corners of $c$ are occupied alternatingly by the coordinates of $\y$ and $\x$.
Finally we note that no side of $c$ can contain a component of $\x$ (or $\y$) in its interior otherwise there would be an $\alpha$- or a $\beta$-curve containing more than one coordinate of $\x$.  
\end{proof}

We then introduce some quantities we shall use for the proof of \Cref{prop:munumod2}

\begin{defn} \label{defn:munudomain} 
Let $D$ be an embedded domain with $n_\z(D) = 0$.
We define 
\begin{equation} \label{eq:mudomain}
\overline{\mu}(D) = e(D) + \frac{1}{4} \text{\#(convex corners)} + \frac{3}{4} \text{\#(concave corners)},
\end{equation}
and 
\begin{equation} \label{eq:nudomain}
\overline{\nu}(D) = \sum_{\text{$\beta$-side $b$}} (\len(b)+1) + |\partial D|.
\end{equation}
\end{defn}

Note that these expressions depend only on $D$ and not on any initial/final states.
We first show that the right hand side of \Cref{prop:munumod2} can be replaced with $\overline{\mu}(D)$.

\begin{lemma} \label{lemma:mudomain}
Let $D$ be an effective domain connecting a state $\x$ to a state $\y$, and for which $n_\z(D) = 0$. Then $\mu(D) \equiv \overline{\mu}(D) \pmod{2}$.
\end{lemma}
\begin{proof}
By \Cref{prop:effdomainemb}, the coefficients of $D$ are either $0$ or $1$, so $D$ is an embedded domain and it makes sense to talk about $\overline{\mu}(D)$. Writing down the explicit formulae for the Maslov index $\mu$ and the index $\overline{\mu}$ from Equations \eqref{eq:lipshitzindex} and \eqref{eq:mudomain}, and using the fact that there are no coordinates  of $\x$ and $\y$  lying on the sides of $D$ except at the corners (that we proved in \Cref{prop:corners}) one computes:
\[\mu(D) - \overline{\mu}(D)= \sum_{x_i \in \text{int}(D) } n_{x_i}(D) + \sum_{y_i \in \text{int}(D) } n_{y_i}(D) \ .  \]
Writing down \Cref{eq:initialfinallocal} at the coordinates of $\x$ and $\y$ lying in the interior of $D$ one can conclude that each coordinate of $\x$ lying in the interior of $D$ is also a coordinate of $\y$, and vice versa.  Thus $\sum_{x_i \in \text{int}(D) } n_{x_i}(D) = \sum_{y_i \in \text{int}(D) } n_{y_i}(D) \in \mathbb{Z}$  showing that $\mu(D) - \overline{\mu}(D)$ is an even integer.
\end{proof}

Next we show that the left hand side of \Cref{prop:munumod2} can be replaced with $\overline{\nu}(D)$.

\begin{lemma} \label{lemma:nudomain}
Let $D$ be an effective domain connecting a state $\x$ to a state $\y$, and for which $n_\z(D) = 0$. Then $\nu(\x)-\nu(\y) \equiv \overline{\nu}(D) \pmod{2}$.
\end{lemma}
\begin{proof}
Suppose that the $\beta$-arcs on $\partial D$ are $b_1,\dots,b_n$, where $b_i \subset \beta_i$. 
(Here if $\partial D$ does not contain a $\beta$-arc lying along $\beta_i$, we take $b_i = \varnothing$.)
Each $b_i$ connects a point $y_i \in \y$ to a point $x_i \in \x$.

Recall that each $\beta$-curve contains 4 points of $\alpha \cap \beta$, corresponding to the 4 corners of the corresponding sector.
Observe that the 2 points of $\alpha \cap \beta$ that correspond to the side corners are non-consecutive on the $\beta$-curve.
This implies that if $b_i$ is of odd length, then exactly one of $x_i$ and $y_i$ is a side corner of $S_i$, and if $b_i$ is of even length, then either both or none of $x_i$ and $y_i$ is a side corner of $S_i$.
In other words, 
\begin{small}
$$\sum_{\text{$\beta$-side $b$}} \len(b) \equiv \{S \mid \text{$\x(S)$ is side corner of $S$}\} - \{S \mid \text{$\y(S)$ is side corner of $S$}\}  \pmod{2}.$$
\end{small}

On the other hand, let $\sigma$ and $\tau$ be the permutations associated to $\x$ and $\y$. Then $x_i \in \alpha_{\sigma(i)}$ and $y_i \in \alpha_{\tau(i)}$.
Then $\sigma^{-1} \tau$ is the permutation that sends each $i$ to the $j$ for which the point $x_j$ lies in the same $\alpha$-curve as $y_i$.
In other words, the configuration of the points $x_i$ and $y_i$ on $\partial D$ determines a disjoint cycle factorization of $\sigma^{-1} \tau$, with one cycle per boundary component, and the length of each cycle being the number of $\alpha$-sides, which equals the number of $\beta$-sides, on the corresponding boundary component.
Thus 
$$\sgn \x - \sgn \y = \sgn(\sigma^{-1} \tau) \equiv \text{\# $\beta$-sides} + |\partial D| \pmod{2}$$

Combining the two equations gives the lemma.
\end{proof}

With these two lemmas, we are ready to prove the proposition.

\begin{proof}[Proof of \Cref{prop:munumod2}]
By \Cref{lemma:mudomain} and \Cref{lemma:nudomain}, it suffices to show that $\overline{\mu}(D) \equiv \overline{\nu}(D) \pmod{2}$ for every embedded domain $D$ with $n_\z(D) = 0$.
We do this by inducting on (\# elementary domains in $D$, $-\chi(D)$) with the lexicographic order.

If $D$ just consists of one elementary domain, then one can verify via Equations \eqref{eq:mudomain} and \eqref{eq:nudomain} that $\overline{\mu}(D) \equiv \overline{\nu}(D) \equiv 1 \pmod{2}$.

For the induction step, first suppose that $D$ is the disjoint union of two embedded domains $D_1$ and $D_2$. Then the number of elementary domains in $D_1$ and $D_2$ are less than that of $D$.
So we can apply the induction hypothesis to compute
$$\overline{\mu}(D) \equiv \overline{\mu}(D_1) + \overline{\mu}(D_2) \equiv \overline{\nu}(D_1)+ \overline{\nu}(D_2) \equiv \overline{\nu}(D) \pmod{2}.$$

Otherwise there is a $\alpha$-curve or a $\beta$-curve intersecting the interior of $D$. This curve must intersect $D$ in an arc $c$ (as opposed to intersecting $D$ in the whole curve) by \Cref{lemma:nofullalphabetacurve}.
Either $D \backslash c$ contains two components $D_1$ and $D_2$, in which case the number of elementary domains in $D_1$ and $D_2$ are less than that of $D$, or $D \backslash c$ has just one component, in which case $-\chi(D \backslash c) < -\chi(D)$.
This allows us to apply the induction hypothesis on $D_1, D_2$ or on $D \backslash c$. For notational simplicity we suppose we are in the former case; the latter case can be proven similarly.

\begin{figure}
    \centering
    \selectfont\fontsize{10pt}{10pt}
    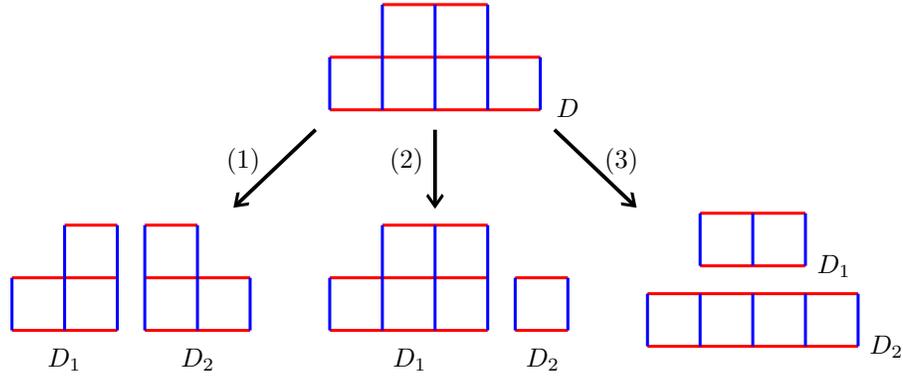
    \caption{Dividing the domain $D$ into $D_1$ and $D_2$ in order to apply the induction hypothesis. There are three cases depending on how many endpoints of $c$ lie on concave corners of $D$. Furthermore, in each case $c$ can either be an $\alpha$- or $\beta$-arc. For simplicity in this picture we only draw three of the possible six cases. }
    \label{fig:domainsplit}
\end{figure}

There are now three subcases:
\begin{enumerate}
    \item $c$ connects two points of $\partial D$ that lie in interiors of sides.
    \item $c$ connects a point of $\partial D$ that lie in interior of a side to a point that is a concave corner of $\partial D$.
    \item $c$ connects two concave corners of $\partial D$.
\end{enumerate}
See \Cref{fig:domainsplit}.
We will verify in each case that the induction hypothesis implies the proposition for $D$. In case (1) one has that 
$$\overline{\mu}(D) \equiv \overline{\mu}(D_1) + \overline{\mu}(D_2) +1 \equiv \overline{\nu}(D_1)+ \overline{\nu}(D_2) +1 \equiv \overline{\nu}(D) \pmod{2}.$$
In this chain of identities:
\begin{itemize}
    \item The first identity is because the Euler measure is additive: $e(D)=e(D_1)+e(D_2)$. Furthermore, the number of  concave corners of $D_1$  and $D_2$ equals the number of concave corners of $D$, while the number of convex corners of $D_1$  and $D_2$ equals the number of convex corners of $D$ plus $4$.
    Thus $\overline{\mu}(D) = \overline{\mu}(D_1) + \overline{\mu}(D_2) - 1$.
    \item The second identity is because of the induction hypothesis.
    \item The third identity can be verified noting that $\sum_{\text{$\beta$-side of $D$}} (\len(b)+1)$ equals
    \[\sum_{\text{$\beta$-side of $D_1$}} (\len(b)+1)+ \sum_{\text{$\beta$-side of $D_2$}} (\len(b)+1) -2 \]
    if $c \subset \boa$, and 
    \[\sum_{\text{$\beta$-side of $D_1$}} (\len(b)+1)+ \sum_{\text{$\beta$-side of $D_2$}} (\len(b)+1) - 2 \len(c) -2 \]
    if $c \subset \bob$. Furthermore, $|\partial D|+1= |\partial D_1|+ |\partial D_2|$, showing that  $\overline{\nu}(D)$ and $ \overline{\nu}(D_1) + \overline{\nu}(D_2)$ differ by an odd number.
\end{itemize}

Case (3) is analogous to case (1). In case (2) one shows similarly that
$$\overline{\mu}(D) \equiv \overline{\mu}(D_1) + \overline{\mu}(D_2) \equiv \overline{\nu}(D_1)+ \overline{\nu}(D_2) \equiv \overline{\nu}(D) \pmod{2}\ .$$
This completes the induction and concludes the proof of the proposition.
\end{proof}

\subsection{The categorification argument}

We are now ready to prove \Cref{thm:SFHcategorifiesantiveering}.

\begin{proof}[Proof of \Cref{thm:SFHcategorifiesantiveering}]
We use the notation in \Cref{subsec:antiveeringpoly}.
For each $i=1,\dots,n$, let $\widehat{S_i}$ be a lift of $S_i$ and let $\widehat{v_i}$ be a lift of $v_i$.
Then $\beta_\mathcal{S} = \{\widehat{S_1},\dots,\widehat{S_n}\}$ is a basis of $\mathcal{S}$ and $\beta_\mathcal{T} = \{\widehat{v_1},\dots,\widehat{v_n}\}$ is a basis of $\mathcal{T}$.

The $(i,j)$ entry of the matrix $[D_A]_{\beta_\mathcal{S}}^{\beta_\mathcal{T}}$ equals
\begin{small}
$$d_{i,j} = \sum \left\{g \mid \text{$\widehat{v_i}$ is the top or bottom corner of $g \cdot \widehat{S_i}$}\right\} - \sum \left\{g \mid \text{$\widehat{v_i}$ is a side corner of $g \cdot \widehat{S_i}$}\right\}.$$
\end{small}
Hence the nonzero terms in $\det [D_A]_{\beta_\mathcal{S}}^{\beta_\mathcal{T}}$ correspond to bijections $\x:\{S_1,\dots,S_N\} \to \{v_1,\dots,v_N\}$, with the sign of the term given by 
$$\sgn(\x) + \# \{S \mid \text{$\x(S)$ is a side corner of $S$}\} = \nu(\x).$$
Finally, the term itself, as an element of $G$, equals the product of the $g$, hence equals the homology class of the curve determined by $\x$.

It follows that 
\begin{align*}A_B &=\sum_{\x \in \mathfrak{S}(\Sigma, \boa, \bob)} \nu(\x ) \cdot \s(\x)  \\
&=\sum_{\s \in \text{Spin}^c(M, \Gamma)} \chi_\nu(CF(\Sigma, \boa, \bob, \s ))  \cdot \s \ ,
\end{align*}
where $CF(\Sigma, \boa, \bob, \s )$ denotes the portion of the Heegaard Floer chain complex associated to the veering branched surface $B$ generated by the Heegaard states in $\spinc$-grading $\s$. On the other hand using \Cref{prop:munumod2} we can conclude that $\nu$ is a homological grading for the chain complex $CF(\Sigma, \boa, \bob, \s )$. Thus 
$$A_B = \sum_{\s \in \text{Spin}^c(M, \Gamma)} \chi_\nu(SFH(M, \Gamma, \s ))  \cdot \s , $$
and we are done.
\end{proof}

\bibliographystyle{alphaurl}

\bibliography{bib.bib}

\end{document}

%% file: paflow.pdf_tex
\begingroup%
  \makeatletter%
  \providecommand\color[2][]{%
    \errmessage{(Inkscape) Color is used for the text in Inkscape, but the package 'color.sty' is not loaded}%
    \renewcommand\color[2][]{}%
  }%
  \providecommand\transparent[1]{%
    \errmessage{(Inkscape) Transparency is used (non-zero) for the text in Inkscape, but the package 'transparent.sty' is not loaded}%
    \renewcommand\transparent[1]{}%
  }%
  \providecommand\rotatebox[2]{#2}%
  \newcommand*\fsize{\dimexpr\f@size pt\relax}%
  \newcommand*\lineheight[1]{\fontsize{\fsize}{#1\fsize}\selectfont}%
  \ifx\svgwidth\undefined%
    \setlength{\unitlength}{188.25116495bp}%
    \ifx\svgscale\undefined%
      \relax%
    \else%
      \setlength{\unitlength}{\unitlength * \real{\svgscale}}%
    \fi%
  \else%
    \setlength{\unitlength}{\svgwidth}%
  \fi%
  \global\let\svgwidth\undefined%
  \global\let\svgscale\undefined%
  \makeatother%
  \begin{picture}(1,0.63258053)%
    \lineheight{1}%
    \setlength\tabcolsep{0pt}%
    \put(0,0){\includegraphics[width=\unitlength,page=1]{paflow.pdf}}%
  \end{picture}%
\endgroup%

%% file: paflowblowup.pdf_tex
\begingroup%
  \makeatletter%
  \providecommand\color[2][]{%
    \errmessage{(Inkscape) Color is used for the text in Inkscape, but the package 'color.sty' is not loaded}%
    \renewcommand\color[2][]{}%
  }%
  \providecommand\transparent[1]{%
    \errmessage{(Inkscape) Transparency is used (non-zero) for the text in Inkscape, but the package 'transparent.sty' is not loaded}%
    \renewcommand\transparent[1]{}%
  }%
  \providecommand\rotatebox[2]{#2}%
  \newcommand*\fsize{\dimexpr\f@size pt\relax}%
  \newcommand*\lineheight[1]{\fontsize{\fsize}{#1\fsize}\selectfont}%
  \ifx\svgwidth\undefined%
    \setlength{\unitlength}{135.09553564bp}%
    \ifx\svgscale\undefined%
      \relax%
    \else%
      \setlength{\unitlength}{\unitlength * \real{\svgscale}}%
    \fi%
  \else%
    \setlength{\unitlength}{\svgwidth}%
  \fi%
  \global\let\svgwidth\undefined%
  \global\let\svgscale\undefined%
  \makeatother%
  \begin{picture}(1,0.88148023)%
    \lineheight{1}%
    \setlength\tabcolsep{0pt}%
    \put(0,0){\includegraphics[width=\unitlength,page=1]{paflowblowup.pdf}}%
  \end{picture}%
\endgroup%

%% file: perfectfitrect.pdf_tex
\begingroup%
  \makeatletter%
  \providecommand\color[2][]{%
    \errmessage{(Inkscape) Color is used for the text in Inkscape, but the package 'color.sty' is not loaded}%
    \renewcommand\color[2][]{}%
  }%
  \providecommand\transparent[1]{%
    \errmessage{(Inkscape) Transparency is used (non-zero) for the text in Inkscape, but the package 'transparent.sty' is not loaded}%
    \renewcommand\transparent[1]{}%
  }%
  \providecommand\rotatebox[2]{#2}%
  \newcommand*\fsize{\dimexpr\f@size pt\relax}%
  \newcommand*\lineheight[1]{\fontsize{\fsize}{#1\fsize}\selectfont}%
  \ifx\svgwidth\undefined%
    \setlength{\unitlength}{59.52754284bp}%
    \ifx\svgscale\undefined%
      \relax%
    \else%
      \setlength{\unitlength}{\unitlength * \real{\svgscale}}%
    \fi%
  \else%
    \setlength{\unitlength}{\svgwidth}%
  \fi%
  \global\let\svgwidth\undefined%
  \global\let\svgscale\undefined%
  \makeatother%
  \begin{picture}(1,0.85714309)%
    \lineheight{1}%
    \setlength\tabcolsep{0pt}%
    \put(0,0){\includegraphics[width=\unitlength,page=1]{perfectfitrect.pdf}}%
  \end{picture}%
\endgroup%

%% file: bsdefn.pdf_tex
\begingroup%
  \makeatletter%
  \providecommand\color[2][]{%
    \errmessage{(Inkscape) Color is used for the text in Inkscape, but the package 'color.sty' is not loaded}%
    \renewcommand\color[2][]{}%
  }%
  \providecommand\transparent[1]{%
    \errmessage{(Inkscape) Transparency is used (non-zero) for the text in Inkscape, but the package 'transparent.sty' is not loaded}%
    \renewcommand\transparent[1]{}%
  }%
  \providecommand\rotatebox[2]{#2}%
  \newcommand*\fsize{\dimexpr\f@size pt\relax}%
  \newcommand*\lineheight[1]{\fontsize{\fsize}{#1\fsize}\selectfont}%
  \ifx\svgwidth\undefined%
    \setlength{\unitlength}{106.13795363bp}%
    \ifx\svgscale\undefined%
      \relax%
    \else%
      \setlength{\unitlength}{\unitlength * \real{\svgscale}}%
    \fi%
  \else%
    \setlength{\unitlength}{\svgwidth}%
  \fi%
  \global\let\svgwidth\undefined%
  \global\let\svgscale\undefined%
  \makeatother%
  \begin{picture}(1,0.45597834)%
    \lineheight{1}%
    \setlength\tabcolsep{0pt}%
    \put(0,0){\includegraphics[width=\unitlength,page=1]{bsdefn.pdf}}%
  \end{picture}%
\endgroup%

%% file: veeringcondition.pdf_tex
\begingroup%
  \makeatletter%
  \providecommand\color[2][]{%
    \errmessage{(Inkscape) Color is used for the text in Inkscape, but the package 'color.sty' is not loaded}%
    \renewcommand\color[2][]{}%
  }%
  \providecommand\transparent[1]{%
    \errmessage{(Inkscape) Transparency is used (non-zero) for the text in Inkscape, but the package 'transparent.sty' is not loaded}%
    \renewcommand\transparent[1]{}%
  }%
  \providecommand\rotatebox[2]{#2}%
  \newcommand*\fsize{\dimexpr\f@size pt\relax}%
  \newcommand*\lineheight[1]{\fontsize{\fsize}{#1\fsize}\selectfont}%
  \ifx\svgwidth\undefined%
    \setlength{\unitlength}{76.07695392bp}%
    \ifx\svgscale\undefined%
      \relax%
    \else%
      \setlength{\unitlength}{\unitlength * \real{\svgscale}}%
    \fi%
  \else%
    \setlength{\unitlength}{\svgwidth}%
  \fi%
  \global\let\svgwidth\undefined%
  \global\let\svgscale\undefined%
  \makeatother%
  \begin{picture}(1,0.99999993)%
    \lineheight{1}%
    \setlength\tabcolsep{0pt}%
    \put(0,0){\includegraphics[width=\unitlength,page=1]{veeringcondition.pdf}}%
  \end{picture}%
\endgroup%

%% file: maxrecttoveertet.pdf_tex
\begingroup%
  \makeatletter%
  \providecommand\color[2][]{%
    \errmessage{(Inkscape) Color is used for the text in Inkscape, but the package 'color.sty' is not loaded}%
    \renewcommand\color[2][]{}%
  }%
  \providecommand\transparent[1]{%
    \errmessage{(Inkscape) Transparency is used (non-zero) for the text in Inkscape, but the package 'transparent.sty' is not loaded}%
    \renewcommand\transparent[1]{}%
  }%
  \providecommand\rotatebox[2]{#2}%
  \newcommand*\fsize{\dimexpr\f@size pt\relax}%
  \newcommand*\lineheight[1]{\fontsize{\fsize}{#1\fsize}\selectfont}%
  \ifx\svgwidth\undefined%
    \setlength{\unitlength}{164.05515055bp}%
    \ifx\svgscale\undefined%
      \relax%
    \else%
      \setlength{\unitlength}{\unitlength * \real{\svgscale}}%
    \fi%
  \else%
    \setlength{\unitlength}{\svgwidth}%
  \fi%
  \global\let\svgwidth\undefined%
  \global\let\svgscale\undefined%
  \makeatother%
  \begin{picture}(1,1.4542853)%
    \lineheight{1}%
    \setlength\tabcolsep{0pt}%
    \put(0,0){\includegraphics[width=\unitlength,page=1]{maxrecttoveertet.pdf}}%
  \end{picture}%
\endgroup%

%% file: veertettovbs.pdf_tex
\begingroup%
  \makeatletter%
  \providecommand\color[2][]{%
    \errmessage{(Inkscape) Color is used for the text in Inkscape, but the package 'color.sty' is not loaded}%
    \renewcommand\color[2][]{}%
  }%
  \providecommand\transparent[1]{%
    \errmessage{(Inkscape) Transparency is used (non-zero) for the text in Inkscape, but the package 'transparent.sty' is not loaded}%
    \renewcommand\transparent[1]{}%
  }%
  \providecommand\rotatebox[2]{#2}%
  \newcommand*\fsize{\dimexpr\f@size pt\relax}%
  \newcommand*\lineheight[1]{\fontsize{\fsize}{#1\fsize}\selectfont}%
  \ifx\svgwidth\undefined%
    \setlength{\unitlength}{179.85060324bp}%
    \ifx\svgscale\undefined%
      \relax%
    \else%
      \setlength{\unitlength}{\unitlength * \real{\svgscale}}%
    \fi%
  \else%
    \setlength{\unitlength}{\svgwidth}%
  \fi%
  \global\let\svgwidth\undefined%
  \global\let\svgscale\undefined%
  \makeatother%
  \begin{picture}(1,0.90464468)%
    \lineheight{1}%
    \setlength\tabcolsep{0pt}%
    \put(0,0){\includegraphics[width=\unitlength,page=1]{veertettovbs.pdf}}%
  \end{picture}%
\endgroup%

%% file: bstosut.pdf_tex
\begingroup%
  \makeatletter%
  \providecommand\color[2][]{%
    \errmessage{(Inkscape) Color is used for the text in Inkscape, but the package 'color.sty' is not loaded}%
    \renewcommand\color[2][]{}%
  }%
  \providecommand\transparent[1]{%
    \errmessage{(Inkscape) Transparency is used (non-zero) for the text in Inkscape, but the package 'transparent.sty' is not loaded}%
    \renewcommand\transparent[1]{}%
  }%
  \providecommand\rotatebox[2]{#2}%
  \newcommand*\fsize{\dimexpr\f@size pt\relax}%
  \newcommand*\lineheight[1]{\fontsize{\fsize}{#1\fsize}\selectfont}%
  \ifx\svgwidth\undefined%
    \setlength{\unitlength}{326.63936585bp}%
    \ifx\svgscale\undefined%
      \relax%
    \else%
      \setlength{\unitlength}{\unitlength * \real{\svgscale}}%
    \fi%
  \else%
    \setlength{\unitlength}{\svgwidth}%
  \fi%
  \global\let\svgwidth\undefined%
  \global\let\svgscale\undefined%
  \makeatother%
  \begin{picture}(1,0.15097606)%
    \lineheight{1}%
    \setlength\tabcolsep{0pt}%
    \put(0,0){\includegraphics[width=\unitlength,page=1]{bstosut.pdf}}%
    \put(0.44193885,0.06604117){\color[rgb]{0,0.50196078,0}\makebox(0,0)[lt]{\lineheight{1.25}\smash{\begin{tabular}[t]{l}$R_+ = \partial_v M$\end{tabular}}}}%
    \put(0.44193885,0.02930337){\color[rgb]{0.50196078,0,0.50196078}\makebox(0,0)[lt]{\lineheight{1.25}\smash{\begin{tabular}[t]{l}$R_- = \partial_h M$\end{tabular}}}}%
  \end{picture}%
\endgroup%

%% file: indexeg.pdf_tex
\begingroup%
  \makeatletter%
  \providecommand\color[2][]{%
    \errmessage{(Inkscape) Color is used for the text in Inkscape, but the package 'color.sty' is not loaded}%
    \renewcommand\color[2][]{}%
  }%
  \providecommand\transparent[1]{%
    \errmessage{(Inkscape) Transparency is used (non-zero) for the text in Inkscape, but the package 'transparent.sty' is not loaded}%
    \renewcommand\transparent[1]{}%
  }%
  \providecommand\rotatebox[2]{#2}%
  \newcommand*\fsize{\dimexpr\f@size pt\relax}%
  \newcommand*\lineheight[1]{\fontsize{\fsize}{#1\fsize}\selectfont}%
  \ifx\svgwidth\undefined%
    \setlength{\unitlength}{215.66973108bp}%
    \ifx\svgscale\undefined%
      \relax%
    \else%
      \setlength{\unitlength}{\unitlength * \real{\svgscale}}%
    \fi%
  \else%
    \setlength{\unitlength}{\svgwidth}%
  \fi%
  \global\let\svgwidth\undefined%
  \global\let\svgscale\undefined%
  \makeatother%
  \begin{picture}(1,0.99201197)%
    \lineheight{1}%
    \setlength\tabcolsep{0pt}%
    \put(0,0){\includegraphics[width=\unitlength,page=1]{indexeg.pdf}}%
    \put(0.03853523,0.90102175){\color[rgb]{0,0,0}\makebox(0,0)[lt]{\lineheight{1.25}\smash{\begin{tabular}[t]{l}$\y$\end{tabular}}}}%
    \put(0.31712332,0.90102175){\color[rgb]{0,0,0}\makebox(0,0)[lt]{\lineheight{1.25}\smash{\begin{tabular}[t]{l}$\x$\end{tabular}}}}%
    \put(0.80153818,0.9634671){\color[rgb]{0,0,0}\makebox(0,0)[lt]{\lineheight{1.25}\smash{\begin{tabular}[t]{l}$\x$\end{tabular}}}}%
    \put(0.54663443,0.74284321){\color[rgb]{0,0,0}\makebox(0,0)[lt]{\lineheight{1.25}\smash{\begin{tabular}[t]{l}$\x$\end{tabular}}}}%
    \put(0.82020547,0.61774613){\color[rgb]{0,0,0}\makebox(0,0)[lt]{\lineheight{1.25}\smash{\begin{tabular}[t]{l}$\x$\end{tabular}}}}%
    \put(0.3341198,0.39377862){\color[rgb]{0,0,0}\makebox(0,0)[lt]{\lineheight{1.25}\smash{\begin{tabular}[t]{l}$\x$\end{tabular}}}}%
    \put(0.19501809,0.22685666){\color[rgb]{0,0,0}\makebox(0,0)[lt]{\lineheight{1.25}\smash{\begin{tabular}[t]{l}$\x$\end{tabular}}}}%
    \put(0.00723078,0.11557549){\color[rgb]{0,0,0}\makebox(0,0)[lt]{\lineheight{1.25}\smash{\begin{tabular}[t]{l}$\x$\end{tabular}}}}%
    \put(0.54103719,0.39889683){\color[rgb]{0,0,0}\makebox(0,0)[lt]{\lineheight{1.25}\smash{\begin{tabular}[t]{l}$\y$\end{tabular}}}}%
    \put(0.03853523,0.68541409){\color[rgb]{0,0,0}\makebox(0,0)[lt]{\lineheight{1.25}\smash{\begin{tabular}[t]{l}$\x$\end{tabular}}}}%
    \put(0.31712332,0.68541409){\color[rgb]{0,0,0}\makebox(0,0)[lt]{\lineheight{1.25}\smash{\begin{tabular}[t]{l}$\y$\end{tabular}}}}%
    \put(0.91131109,0.74284321){\color[rgb]{0,0,0}\makebox(0,0)[lt]{\lineheight{1.25}\smash{\begin{tabular}[t]{l}$\y$\end{tabular}}}}%
    \put(0.63937395,0.61774613){\color[rgb]{0,0,0}\makebox(0,0)[lt]{\lineheight{1.25}\smash{\begin{tabular}[t]{l}$\y$\end{tabular}}}}%
    \put(0.00723078,0.39377862){\color[rgb]{0,0,0}\makebox(0,0)[lt]{\lineheight{1.25}\smash{\begin{tabular}[t]{l}$\y$\end{tabular}}}}%
    \put(0.3341198,0.22685666){\color[rgb]{0,0,0}\makebox(0,0)[lt]{\lineheight{1.25}\smash{\begin{tabular}[t]{l}$\y$\end{tabular}}}}%
    \put(0.91661131,0.39889683){\color[rgb]{0,0,0}\makebox(0,0)[lt]{\lineheight{1.25}\smash{\begin{tabular}[t]{l}$\x$\end{tabular}}}}%
    \put(0.54103719,0.11373842){\color[rgb]{0,0,0}\makebox(0,0)[lt]{\lineheight{1.25}\smash{\begin{tabular}[t]{l}$\x$\end{tabular}}}}%
    \put(0.91661131,0.11373842){\color[rgb]{0,0,0}\makebox(0,0)[lt]{\lineheight{1.25}\smash{\begin{tabular}[t]{l}$\y$\end{tabular}}}}%
    \put(0.19501809,0.11557549){\color[rgb]{0,0,0}\makebox(0,0)[lt]{\lineheight{1.25}\smash{\begin{tabular}[t]{l}$\y$\end{tabular}}}}%
    \put(0.65548208,0.9634671){\color[rgb]{0,0,0}\makebox(0,0)[lt]{\lineheight{1.25}\smash{\begin{tabular}[t]{l}$\y$\end{tabular}}}}%
    \put(0.69422917,0.25668254){\color[rgb]{0,0,0}\makebox(0,0)[lt]{\lineheight{1.25}\smash{\begin{tabular}[t]{l}$\x \cap \y$\end{tabular}}}}%
    \put(0.11958872,0.5420782){\color[rgb]{0,0,0}\makebox(0,0)[lt]{\lineheight{1.25}\smash{\begin{tabular}[t]{l}\small $\mu=1$\end{tabular}}}}%
    \put(0.68295075,0.5420782){\color[rgb]{0,0,0}\makebox(0,0)[lt]{\lineheight{1.25}\smash{\begin{tabular}[t]{l}\small $\mu=1$\end{tabular}}}}%
    \put(0.11958872,0.00653851){\color[rgb]{0,0,0}\makebox(0,0)[lt]{\lineheight{1.25}\smash{\begin{tabular}[t]{l}\small $\mu=2$\end{tabular}}}}%
    \put(0.68295075,0.00653851){\color[rgb]{0,0,0}\makebox(0,0)[lt]{\lineheight{1.25}\smash{\begin{tabular}[t]{l}\small $\mu=1$\end{tabular}}}}%
  \end{picture}%
\endgroup%

%% file: initialfinallocal.pdf_tex
\begingroup%
  \makeatletter%
  \providecommand\color[2][]{%
    \errmessage{(Inkscape) Color is used for the text in Inkscape, but the package 'color.sty' is not loaded}%
    \renewcommand\color[2][]{}%
  }%
  \providecommand\transparent[1]{%
    \errmessage{(Inkscape) Transparency is used (non-zero) for the text in Inkscape, but the package 'transparent.sty' is not loaded}%
    \renewcommand\transparent[1]{}%
  }%
  \providecommand\rotatebox[2]{#2}%
  \newcommand*\fsize{\dimexpr\f@size pt\relax}%
  \newcommand*\lineheight[1]{\fontsize{\fsize}{#1\fsize}\selectfont}%
  \ifx\svgwidth\undefined%
    \setlength{\unitlength}{79.51210623bp}%
    \ifx\svgscale\undefined%
      \relax%
    \else%
      \setlength{\unitlength}{\unitlength * \real{\svgscale}}%
    \fi%
  \else%
    \setlength{\unitlength}{\svgwidth}%
  \fi%
  \global\let\svgwidth\undefined%
  \global\let\svgscale\undefined%
  \makeatother%
  \begin{picture}(1,0.83884052)%
    \lineheight{1}%
    \setlength\tabcolsep{0pt}%
    \put(0,0){\includegraphics[width=\unitlength,page=1]{initialfinallocal.pdf}}%
    \put(0.16234278,0.58344597){\color[rgb]{0,0,0}\makebox(0,0)[lt]{\lineheight{1.25}\smash{\begin{tabular}[t]{l}$D_j$\end{tabular}}}}%
    \put(0.67169804,0.58344597){\color[rgb]{0,0,0}\makebox(0,0)[lt]{\lineheight{1.25}\smash{\begin{tabular}[t]{l}$D_i$\end{tabular}}}}%
    \put(0.16234278,0.14954963){\color[rgb]{0,0,0}\makebox(0,0)[lt]{\lineheight{1.25}\smash{\begin{tabular}[t]{l}$D_k$\end{tabular}}}}%
    \put(0.67169804,0.14954963){\color[rgb]{0,0,0}\makebox(0,0)[lt]{\lineheight{1.25}\smash{\begin{tabular}[t]{l}$D_l$\end{tabular}}}}%
    \put(0.54213481,0.31118639){\color[rgb]{0,0,0}\makebox(0,0)[lt]{\lineheight{1.25}\smash{\begin{tabular}[t]{l}$p$\end{tabular}}}}%
  \end{picture}%
\endgroup%

%% file: morsespincdefn.pdf_tex
\begingroup%
  \makeatletter%
  \providecommand\color[2][]{%
    \errmessage{(Inkscape) Color is used for the text in Inkscape, but the package 'color.sty' is not loaded}%
    \renewcommand\color[2][]{}%
  }%
  \providecommand\transparent[1]{%
    \errmessage{(Inkscape) Transparency is used (non-zero) for the text in Inkscape, but the package 'transparent.sty' is not loaded}%
    \renewcommand\transparent[1]{}%
  }%
  \providecommand\rotatebox[2]{#2}%
  \newcommand*\fsize{\dimexpr\f@size pt\relax}%
  \newcommand*\lineheight[1]{\fontsize{\fsize}{#1\fsize}\selectfont}%
  \ifx\svgwidth\undefined%
    \setlength{\unitlength}{164.94661106bp}%
    \ifx\svgscale\undefined%
      \relax%
    \else%
      \setlength{\unitlength}{\unitlength * \real{\svgscale}}%
    \fi%
  \else%
    \setlength{\unitlength}{\svgwidth}%
  \fi%
  \global\let\svgwidth\undefined%
  \global\let\svgscale\undefined%
  \makeatother%
  \begin{picture}(1,0.96378884)%
    \lineheight{1}%
    \setlength\tabcolsep{0pt}%
    \put(0,0){\includegraphics[width=\unitlength,page=1]{morsespincdefn.pdf}}%
    \put(0.68464189,0.74471298){\color[rgb]{0,0,0}\makebox(0,0)[lt]{\lineheight{1.25}\smash{\begin{tabular}[t]{l}$Q_1$\end{tabular}}}}%
    \put(0.45718682,0.62164595){\color[rgb]{0,0,0}\makebox(0,0)[lt]{\lineheight{1.25}\smash{\begin{tabular}[t]{l}$B_1$\end{tabular}}}}%
    \put(0.18754185,0.14481097){\color[rgb]{0,0,0}\makebox(0,0)[lt]{\lineheight{1.25}\smash{\begin{tabular}[t]{l}$P_1$\end{tabular}}}}%
    \put(0.14851425,0.32903654){\color[rgb]{1,0,0}\makebox(0,0)[lt]{\lineheight{1.25}\smash{\begin{tabular}[t]{l}$\alpha_1$\end{tabular}}}}%
    \put(0.73007815,0.16612993){\color[rgb]{0,0,0}\makebox(0,0)[lt]{\lineheight{1.25}\smash{\begin{tabular}[t]{l}\large $\Sigma$\end{tabular}}}}%
    \put(0.75065344,0.59348158){\color[rgb]{0,0,1}\makebox(0,0)[lt]{\lineheight{1.25}\smash{\begin{tabular}[t]{l}$\beta_1$\end{tabular}}}}%
  \end{picture}%
\endgroup%

%% file: vbstoheegaardtriplepoint.pdf_tex
\begingroup%
  \makeatletter%
  \providecommand\color[2][]{%
    \errmessage{(Inkscape) Color is used for the text in Inkscape, but the package 'color.sty' is not loaded}%
    \renewcommand\color[2][]{}%
  }%
  \providecommand\transparent[1]{%
    \errmessage{(Inkscape) Transparency is used (non-zero) for the text in Inkscape, but the package 'transparent.sty' is not loaded}%
    \renewcommand\transparent[1]{}%
  }%
  \providecommand\rotatebox[2]{#2}%
  \newcommand*\fsize{\dimexpr\f@size pt\relax}%
  \newcommand*\lineheight[1]{\fontsize{\fsize}{#1\fsize}\selectfont}%
  \ifx\svgwidth\undefined%
    \setlength{\unitlength}{188.25538215bp}%
    \ifx\svgscale\undefined%
      \relax%
    \else%
      \setlength{\unitlength}{\unitlength * \real{\svgscale}}%
    \fi%
  \else%
    \setlength{\unitlength}{\svgwidth}%
  \fi%
  \global\let\svgwidth\undefined%
  \global\let\svgscale\undefined%
  \makeatother%
  \begin{picture}(1,0.40447027)%
    \lineheight{1}%
    \setlength\tabcolsep{0pt}%
    \put(0,0){\includegraphics[width=\unitlength,page=1]{vbstoheegaardtriplepoint.pdf}}%
  \end{picture}%
\endgroup%

%% file: vbstoheegaardcutannulus.pdf_tex
\begingroup%
  \makeatletter%
  \providecommand\color[2][]{%
    \errmessage{(Inkscape) Color is used for the text in Inkscape, but the package 'color.sty' is not loaded}%
    \renewcommand\color[2][]{}%
  }%
  \providecommand\transparent[1]{%
    \errmessage{(Inkscape) Transparency is used (non-zero) for the text in Inkscape, but the package 'transparent.sty' is not loaded}%
    \renewcommand\transparent[1]{}%
  }%
  \providecommand\rotatebox[2]{#2}%
  \newcommand*\fsize{\dimexpr\f@size pt\relax}%
  \newcommand*\lineheight[1]{\fontsize{\fsize}{#1\fsize}\selectfont}%
  \ifx\svgwidth\undefined%
    \setlength{\unitlength}{366.05601874bp}%
    \ifx\svgscale\undefined%
      \relax%
    \else%
      \setlength{\unitlength}{\unitlength * \real{\svgscale}}%
    \fi%
  \else%
    \setlength{\unitlength}{\svgwidth}%
  \fi%
  \global\let\svgwidth\undefined%
  \global\let\svgscale\undefined%
  \makeatother%
  \begin{picture}(1,0.37523267)%
    \lineheight{1}%
    \setlength\tabcolsep{0pt}%
    \put(0,0){\includegraphics[width=\unitlength,page=1]{vbstoheegaardcutannulus.pdf}}%
    \put(0.13854616,0.00577843){\color[rgb]{0,0,0}\makebox(0,0)[lt]{\lineheight{1.25}\smash{\begin{tabular}[t]{l}$M$\end{tabular}}}}%
    \put(0.4950502,0.00577849){\color[rgb]{0,0,0}\makebox(0,0)[lt]{\lineheight{1.25}\smash{\begin{tabular}[t]{l}$M_0$\end{tabular}}}}%
    \put(0.85155383,0.00577849){\color[rgb]{0,0,0}\makebox(0,0)[lt]{\lineheight{1.25}\smash{\begin{tabular}[t]{l}$M'_0$\end{tabular}}}}%
  \end{picture}%
\endgroup%

%% file: vbstoheegaardbranchloop.pdf_tex
\begingroup%
  \makeatletter%
  \providecommand\color[2][]{%
    \errmessage{(Inkscape) Color is used for the text in Inkscape, but the package 'color.sty' is not loaded}%
    \renewcommand\color[2][]{}%
  }%
  \providecommand\transparent[1]{%
    \errmessage{(Inkscape) Transparency is used (non-zero) for the text in Inkscape, but the package 'transparent.sty' is not loaded}%
    \renewcommand\transparent[1]{}%
  }%
  \providecommand\rotatebox[2]{#2}%
  \newcommand*\fsize{\dimexpr\f@size pt\relax}%
  \newcommand*\lineheight[1]{\fontsize{\fsize}{#1\fsize}\selectfont}%
  \ifx\svgwidth\undefined%
    \setlength{\unitlength}{199.89368829bp}%
    \ifx\svgscale\undefined%
      \relax%
    \else%
      \setlength{\unitlength}{\unitlength * \real{\svgscale}}%
    \fi%
  \else%
    \setlength{\unitlength}{\svgwidth}%
  \fi%
  \global\let\svgwidth\undefined%
  \global\let\svgscale\undefined%
  \makeatother%
  \begin{picture}(1,0.78250145)%
    \lineheight{1}%
    \setlength\tabcolsep{0pt}%
    \put(0,0){\includegraphics[width=\unitlength,page=1]{vbstoheegaardbranchloop.pdf}}%
  \end{picture}%
\endgroup%

%% file: vbssector.pdf_tex
\begingroup%
  \makeatletter%
  \providecommand\color[2][]{%
    \errmessage{(Inkscape) Color is used for the text in Inkscape, but the package 'color.sty' is not loaded}%
    \renewcommand\color[2][]{}%
  }%
  \providecommand\transparent[1]{%
    \errmessage{(Inkscape) Transparency is used (non-zero) for the text in Inkscape, but the package 'transparent.sty' is not loaded}%
    \renewcommand\transparent[1]{}%
  }%
  \providecommand\rotatebox[2]{#2}%
  \newcommand*\fsize{\dimexpr\f@size pt\relax}%
  \newcommand*\lineheight[1]{\fontsize{\fsize}{#1\fsize}\selectfont}%
  \ifx\svgwidth\undefined%
    \setlength{\unitlength}{221.18335393bp}%
    \ifx\svgscale\undefined%
      \relax%
    \else%
      \setlength{\unitlength}{\unitlength * \real{\svgscale}}%
    \fi%
  \else%
    \setlength{\unitlength}{\svgwidth}%
  \fi%
  \global\let\svgwidth\undefined%
  \global\let\svgscale\undefined%
  \makeatother%
  \begin{picture}(1,0.49480325)%
    \lineheight{1}%
    \setlength\tabcolsep{0pt}%
    \put(0,0){\includegraphics[width=\unitlength,page=1]{vbssector.pdf}}%
  \end{picture}%
\endgroup%

%% file: vbstoheegaardsector.pdf_tex
\begingroup%
  \makeatletter%
  \providecommand\color[2][]{%
    \errmessage{(Inkscape) Color is used for the text in Inkscape, but the package 'color.sty' is not loaded}%
    \renewcommand\color[2][]{}%
  }%
  \providecommand\transparent[1]{%
    \errmessage{(Inkscape) Transparency is used (non-zero) for the text in Inkscape, but the package 'transparent.sty' is not loaded}%
    \renewcommand\transparent[1]{}%
  }%
  \providecommand\rotatebox[2]{#2}%
  \newcommand*\fsize{\dimexpr\f@size pt\relax}%
  \newcommand*\lineheight[1]{\fontsize{\fsize}{#1\fsize}\selectfont}%
  \ifx\svgwidth\undefined%
    \setlength{\unitlength}{277.62314924bp}%
    \ifx\svgscale\undefined%
      \relax%
    \else%
      \setlength{\unitlength}{\unitlength * \real{\svgscale}}%
    \fi%
  \else%
    \setlength{\unitlength}{\svgwidth}%
  \fi%
  \global\let\svgwidth\undefined%
  \global\let\svgscale\undefined%
  \makeatother%
  \begin{picture}(1,0.44608524)%
    \lineheight{1}%
    \setlength\tabcolsep{0pt}%
    \put(0,0){\includegraphics[width=\unitlength,page=1]{vbstoheegaardsector.pdf}}%
  \end{picture}%
\endgroup%

%% file: statemultiloopdefn.pdf_tex
\begingroup%
  \makeatletter%
  \providecommand\color[2][]{%
    \errmessage{(Inkscape) Color is used for the text in Inkscape, but the package 'color.sty' is not loaded}%
    \renewcommand\color[2][]{}%
  }%
  \providecommand\transparent[1]{%
    \errmessage{(Inkscape) Transparency is used (non-zero) for the text in Inkscape, but the package 'transparent.sty' is not loaded}%
    \renewcommand\transparent[1]{}%
  }%
  \providecommand\rotatebox[2]{#2}%
  \newcommand*\fsize{\dimexpr\f@size pt\relax}%
  \newcommand*\lineheight[1]{\fontsize{\fsize}{#1\fsize}\selectfont}%
  \ifx\svgwidth\undefined%
    \setlength{\unitlength}{258.7668904bp}%
    \ifx\svgscale\undefined%
      \relax%
    \else%
      \setlength{\unitlength}{\unitlength * \real{\svgscale}}%
    \fi%
  \else%
    \setlength{\unitlength}{\svgwidth}%
  \fi%
  \global\let\svgwidth\undefined%
  \global\let\svgscale\undefined%
  \makeatother%
  \begin{picture}(1,0.41337832)%
    \lineheight{1}%
    \setlength\tabcolsep{0pt}%
    \put(0,0){\includegraphics[width=\unitlength,page=1]{statemultiloopdefn.pdf}}%
    \put(-0.0023398,0.31115073){\color[rgb]{0,0,0}\makebox(0,0)[lt]{\lineheight{1.25}\smash{\begin{tabular}[t]{l}$\x(S_i)$\end{tabular}}}}%
    \put(0.00925366,0.07348486){\color[rgb]{0,0,0}\makebox(0,0)[lt]{\lineheight{1.25}\smash{\begin{tabular}[t]{l}$e_{\x,i}$\end{tabular}}}}%
  \end{picture}%
\endgroup%

%% file: strum.pdf_tex
\begingroup%
  \makeatletter%
  \providecommand\color[2][]{%
    \errmessage{(Inkscape) Color is used for the text in Inkscape, but the package 'color.sty' is not loaded}%
    \renewcommand\color[2][]{}%
  }%
  \providecommand\transparent[1]{%
    \errmessage{(Inkscape) Transparency is used (non-zero) for the text in Inkscape, but the package 'transparent.sty' is not loaded}%
    \renewcommand\transparent[1]{}%
  }%
  \providecommand\rotatebox[2]{#2}%
  \newcommand*\fsize{\dimexpr\f@size pt\relax}%
  \newcommand*\lineheight[1]{\fontsize{\fsize}{#1\fsize}\selectfont}%
  \ifx\svgwidth\undefined%
    \setlength{\unitlength}{147.67640602bp}%
    \ifx\svgscale\undefined%
      \relax%
    \else%
      \setlength{\unitlength}{\unitlength * \real{\svgscale}}%
    \fi%
  \else%
    \setlength{\unitlength}{\svgwidth}%
  \fi%
  \global\let\svgwidth\undefined%
  \global\let\svgscale\undefined%
  \makeatother%
  \begin{picture}(1,0.65114128)%
    \lineheight{1}%
    \setlength\tabcolsep{0pt}%
    \put(0,0){\includegraphics[width=\unitlength,page=1]{strum.pdf}}%
    \put(0.17200929,0.34456134){\color[rgb]{0,0,0}\makebox(0,0)[lt]{\lineheight{1.25}\smash{\begin{tabular}[t]{l}strum\end{tabular}}}}%
    \put(0.67549017,0.34456134){\color[rgb]{0,0,0}\makebox(0,0)[lt]{\lineheight{1.25}\smash{\begin{tabular}[t]{l}strum\end{tabular}}}}%
  \end{picture}%
\endgroup%

%% file: topstatespinc.pdf_tex
\begingroup%
  \makeatletter%
  \providecommand\color[2][]{%
    \errmessage{(Inkscape) Color is used for the text in Inkscape, but the package 'color.sty' is not loaded}%
    \renewcommand\color[2][]{}%
  }%
  \providecommand\transparent[1]{%
    \errmessage{(Inkscape) Transparency is used (non-zero) for the text in Inkscape, but the package 'transparent.sty' is not loaded}%
    \renewcommand\transparent[1]{}%
  }%
  \providecommand\rotatebox[2]{#2}%
  \newcommand*\fsize{\dimexpr\f@size pt\relax}%
  \newcommand*\lineheight[1]{\fontsize{\fsize}{#1\fsize}\selectfont}%
  \ifx\svgwidth\undefined%
    \setlength{\unitlength}{181.67941597bp}%
    \ifx\svgscale\undefined%
      \relax%
    \else%
      \setlength{\unitlength}{\unitlength * \real{\svgscale}}%
    \fi%
  \else%
    \setlength{\unitlength}{\svgwidth}%
  \fi%
  \global\let\svgwidth\undefined%
  \global\let\svgscale\undefined%
  \makeatother%
  \begin{picture}(1,1.1478003)%
    \lineheight{1}%
    \setlength\tabcolsep{0pt}%
    \put(0,0){\includegraphics[width=\unitlength,page=1]{topstatespinc.pdf}}%
    \put(0.3941924,0.41704351){\color[rgb]{0,0,0}\makebox(0,0)[lt]{\lineheight{1.25}\smash{\begin{tabular}[t]{l}$Q_i$\end{tabular}}}}%
    \put(0.41082775,0.71288232){\color[rgb]{0,0,0}\makebox(0,0)[lt]{\lineheight{1.25}\smash{\begin{tabular}[t]{l}$P_i$\end{tabular}}}}%
  \end{picture}%
\endgroup%

%% file: vbscolor.pdf_tex
\begingroup%
  \makeatletter%
  \providecommand\color[2][]{%
    \errmessage{(Inkscape) Color is used for the text in Inkscape, but the package 'color.sty' is not loaded}%
    \renewcommand\color[2][]{}%
  }%
  \providecommand\transparent[1]{%
    \errmessage{(Inkscape) Transparency is used (non-zero) for the text in Inkscape, but the package 'transparent.sty' is not loaded}%
    \renewcommand\transparent[1]{}%
  }%
  \providecommand\rotatebox[2]{#2}%
  \newcommand*\fsize{\dimexpr\f@size pt\relax}%
  \newcommand*\lineheight[1]{\fontsize{\fsize}{#1\fsize}\selectfont}%
  \ifx\svgwidth\undefined%
    \setlength{\unitlength}{175.35233986bp}%
    \ifx\svgscale\undefined%
      \relax%
    \else%
      \setlength{\unitlength}{\unitlength * \real{\svgscale}}%
    \fi%
  \else%
    \setlength{\unitlength}{\svgwidth}%
  \fi%
  \global\let\svgwidth\undefined%
  \global\let\svgscale\undefined%
  \makeatother%
  \begin{picture}(1,0.5367177)%
    \lineheight{1}%
    \setlength\tabcolsep{0pt}%
    \put(0,0){\includegraphics[width=\unitlength,page=1]{vbscolor.pdf}}%
    \put(0.16083516,0.00086324){\color[rgb]{0,0,1}\makebox(0,0)[lt]{\lineheight{1.25}\smash{\begin{tabular}[t]{l}bLue\end{tabular}}}}%
    \put(0.76408306,0.00086324){\color[rgb]{1,0,0}\makebox(0,0)[lt]{\lineheight{1.25}\smash{\begin{tabular}[t]{l}Red\end{tabular}}}}%
  \end{picture}%
\endgroup%

%% file: vbssectors.pdf_tex
\begingroup%
  \makeatletter%
  \providecommand\color[2][]{%
    \errmessage{(Inkscape) Color is used for the text in Inkscape, but the package 'color.sty' is not loaded}%
    \renewcommand\color[2][]{}%
  }%
  \providecommand\transparent[1]{%
    \errmessage{(Inkscape) Transparency is used (non-zero) for the text in Inkscape, but the package 'transparent.sty' is not loaded}%
    \renewcommand\transparent[1]{}%
  }%
  \providecommand\rotatebox[2]{#2}%
  \newcommand*\fsize{\dimexpr\f@size pt\relax}%
  \newcommand*\lineheight[1]{\fontsize{\fsize}{#1\fsize}\selectfont}%
  \ifx\svgwidth\undefined%
    \setlength{\unitlength}{262.18939113bp}%
    \ifx\svgscale\undefined%
      \relax%
    \else%
      \setlength{\unitlength}{\unitlength * \real{\svgscale}}%
    \fi%
  \else%
    \setlength{\unitlength}{\svgwidth}%
  \fi%
  \global\let\svgwidth\undefined%
  \global\let\svgscale\undefined%
  \makeatother%
  \begin{picture}(1,0.51382325)%
    \lineheight{1}%
    \setlength\tabcolsep{0pt}%
    \put(0,0){\includegraphics[width=\unitlength,page=1]{vbssectors.pdf}}%
  \end{picture}%
\endgroup%

%% file: heegaardcombinalphaglue.pdf_tex
\begingroup%
  \makeatletter%
  \providecommand\color[2][]{%
    \errmessage{(Inkscape) Color is used for the text in Inkscape, but the package 'color.sty' is not loaded}%
    \renewcommand\color[2][]{}%
  }%
  \providecommand\transparent[1]{%
    \errmessage{(Inkscape) Transparency is used (non-zero) for the text in Inkscape, but the package 'transparent.sty' is not loaded}%
    \renewcommand\transparent[1]{}%
  }%
  \providecommand\rotatebox[2]{#2}%
  \newcommand*\fsize{\dimexpr\f@size pt\relax}%
  \newcommand*\lineheight[1]{\fontsize{\fsize}{#1\fsize}\selectfont}%
  \ifx\svgwidth\undefined%
    \setlength{\unitlength}{291.24206543bp}%
    \ifx\svgscale\undefined%
      \relax%
    \else%
      \setlength{\unitlength}{\unitlength * \real{\svgscale}}%
    \fi%
  \else%
    \setlength{\unitlength}{\svgwidth}%
  \fi%
  \global\let\svgwidth\undefined%
  \global\let\svgscale\undefined%
  \makeatother%
  \begin{picture}(1,0.89589967)%
    \lineheight{1}%
    \setlength\tabcolsep{0pt}%
    \put(0,0){\includegraphics[width=\unitlength,page=1]{heegaardcombinalphaglue.pdf}}%
    \put(0.18276867,0.8645952){\color[rgb]{0,0,0}\makebox(0,0)[lt]{\lineheight{1.25}\smash{\begin{tabular}[t]{l}bLue\end{tabular}}}}%
    \put(0,0){\includegraphics[width=\unitlength,page=2]{heegaardcombinalphaglue.pdf}}%
    \put(0.76225466,0.8645946){\color[rgb]{0,0,0}\makebox(0,0)[lt]{\lineheight{1.25}\smash{\begin{tabular}[t]{l}Red\end{tabular}}}}%
    \put(0,0){\includegraphics[width=\unitlength,page=3]{heegaardcombinalphaglue.pdf}}%
    \put(0.45711594,0.52857955){\color[rgb]{1,0.8,0}\makebox(0,0)[lt]{\lineheight{1.25}\smash{\begin{tabular}[t]{l}$\x^\bot$\end{tabular}}}}%
    \put(0.45769232,0.40556463){\color[rgb]{0.50196078,0,0.50196078}\makebox(0,0)[lt]{\lineheight{1.25}\smash{\begin{tabular}[t]{l}$\x^\top$\end{tabular}}}}%
  \end{picture}%
\endgroup%

%% file: heegaardcombinbranchcurve.pdf_tex
\begingroup%
  \makeatletter%
  \providecommand\color[2][]{%
    \errmessage{(Inkscape) Color is used for the text in Inkscape, but the package 'color.sty' is not loaded}%
    \renewcommand\color[2][]{}%
  }%
  \providecommand\transparent[1]{%
    \errmessage{(Inkscape) Transparency is used (non-zero) for the text in Inkscape, but the package 'transparent.sty' is not loaded}%
    \renewcommand\transparent[1]{}%
  }%
  \providecommand\rotatebox[2]{#2}%
  \newcommand*\fsize{\dimexpr\f@size pt\relax}%
  \newcommand*\lineheight[1]{\fontsize{\fsize}{#1\fsize}\selectfont}%
  \ifx\svgwidth\undefined%
    \setlength{\unitlength}{188.57798659bp}%
    \ifx\svgscale\undefined%
      \relax%
    \else%
      \setlength{\unitlength}{\unitlength * \real{\svgscale}}%
    \fi%
  \else%
    \setlength{\unitlength}{\svgwidth}%
  \fi%
  \global\let\svgwidth\undefined%
  \global\let\svgscale\undefined%
  \makeatother%
  \begin{picture}(1,0.65208364)%
    \lineheight{1}%
    \setlength\tabcolsep{0pt}%
    \put(0,0){\includegraphics[width=\unitlength,page=1]{heegaardcombinbranchcurve.pdf}}%
    \put(0.89087745,0.36361639){\color[rgb]{1,0.8,0}\makebox(0,0)[lt]{\lineheight{1.25}\smash{\begin{tabular}[t]{l}$\x^\bot$\end{tabular}}}}%
    \put(0.89087745,0.26816508){\color[rgb]{0.50196078,0,0.50196078}\makebox(0,0)[lt]{\lineheight{1.25}\smash{\begin{tabular}[t]{l}$\x^\top$\end{tabular}}}}%
  \end{picture}%
\endgroup%

%% file: polyrels.pdf_tex
\begingroup%
  \makeatletter%
  \providecommand\color[2][]{%
    \errmessage{(Inkscape) Color is used for the text in Inkscape, but the package 'color.sty' is not loaded}%
    \renewcommand\color[2][]{}%
  }%
  \providecommand\transparent[1]{%
    \errmessage{(Inkscape) Transparency is used (non-zero) for the text in Inkscape, but the package 'transparent.sty' is not loaded}%
    \renewcommand\transparent[1]{}%
  }%
  \providecommand\rotatebox[2]{#2}%
  \newcommand*\fsize{\dimexpr\f@size pt\relax}%
  \newcommand*\lineheight[1]{\fontsize{\fsize}{#1\fsize}\selectfont}%
  \ifx\svgwidth\undefined%
    \setlength{\unitlength}{303.99559201bp}%
    \ifx\svgscale\undefined%
      \relax%
    \else%
      \setlength{\unitlength}{\unitlength * \real{\svgscale}}%
    \fi%
  \else%
    \setlength{\unitlength}{\svgwidth}%
  \fi%
  \global\let\svgwidth\undefined%
  \global\let\svgscale\undefined%
  \makeatother%
  \begin{picture}(1,0.55338647)%
    \lineheight{1}%
    \setlength\tabcolsep{0pt}%
    \put(0,0){\includegraphics[width=\unitlength,page=1]{polyrels.pdf}}%
    \put(0.17491843,0.42779278){\color[rgb]{0,0,0}\makebox(0,0)[lt]{\lineheight{1.25}\smash{\begin{tabular}[t]{l}$\wh{t}$\end{tabular}}}}%
    \put(0.62149081,0.40472561){\color[rgb]{0,0,0}\makebox(0,0)[lt]{\lineheight{1.25}\smash{\begin{tabular}[t]{l}$\wh{v}$\end{tabular}}}}%
    \put(0.67092815,0.53548411){\color[rgb]{0,0,0}\makebox(0,0)[lt]{\lineheight{1.25}\smash{\begin{tabular}[t]{l}$\wh{t}$\end{tabular}}}}%
    \put(0.62558961,0.13204226){\color[rgb]{0,0,0}\makebox(0,0)[lt]{\lineheight{1.25}\smash{\begin{tabular}[t]{l}$\wh{b}$\end{tabular}}}}%
    \put(0.13343639,0.34467332){\color[rgb]{0,0,0}\makebox(0,0)[lt]{\lineheight{1.25}\smash{\begin{tabular}[t]{l}$\wh{e}$\end{tabular}}}}%
    \put(0.58253706,0.44876312){\color[rgb]{0,0,0}\makebox(0,0)[lt]{\lineheight{1.25}\smash{\begin{tabular}[t]{l}$\wh{e}^t_2$\end{tabular}}}}%
    \put(0.65214465,0.44876312){\color[rgb]{0,0,0}\makebox(0,0)[lt]{\lineheight{1.25}\smash{\begin{tabular}[t]{l}$\wh{e}^t_1$\end{tabular}}}}%
    \put(0.55662499,0.35291475){\color[rgb]{0,0,0}\makebox(0,0)[lt]{\lineheight{1.25}\smash{\begin{tabular}[t]{l}$\wh{e}^b_1$\end{tabular}}}}%
    \put(0.68544396,0.35291475){\color[rgb]{0,0,0}\makebox(0,0)[lt]{\lineheight{1.25}\smash{\begin{tabular}[t]{l}$\wh{e}^b_2$\end{tabular}}}}%
    \put(0.25750657,0.38282773){\color[rgb]{0,0,0}\makebox(0,0)[lt]{\lineheight{1.25}\smash{\begin{tabular}[t]{l}$\wh{b}_1$\end{tabular}}}}%
    \put(0.40553503,0.35653008){\color[rgb]{0,0,0}\makebox(0,0)[lt]{\lineheight{1.25}\smash{\begin{tabular}[t]{l}$\wh{s}_1$\end{tabular}}}}%
    \put(0.45367668,0.45761235){\color[rgb]{0,0,0}\makebox(0,0)[lt]{\lineheight{1.25}\smash{\begin{tabular}[t]{l}$\wh{f}_2$\end{tabular}}}}%
    \put(0.83828477,0.35506293){\color[rgb]{0,0,0}\makebox(0,0)[lt]{\lineheight{1.25}\smash{\begin{tabular}[t]{l}$\wh{s}_2$\end{tabular}}}}%
    \put(0.79191824,0.46054003){\color[rgb]{0,0,0}\makebox(0,0)[lt]{\lineheight{1.25}\smash{\begin{tabular}[t]{l}$\wh{f}_1$\end{tabular}}}}%
    \put(0.27635267,0.31584161){\color[rgb]{0,0,0}\makebox(0,0)[lt]{\lineheight{1.25}\smash{\begin{tabular}[t]{l}$\wh{b}_2$\end{tabular}}}}%
    \put(-0.0006746,0.15613944){\color[rgb]{0,0,0}\makebox(0,0)[lt]{\lineheight{1.25}\smash{\begin{tabular}[t]{l}\small Face rel: $\wh{b}_1+\wh{b}_2-\wh{t}$\end{tabular}}}}%
    \put(0.45220719,0.06860062){\color[rgb]{0,0,0}\makebox(0,0)[lt]{\lineheight{1.25}\smash{\begin{tabular}[t]{l}\small Tet rel: $\wh{b}+\wh{f}_1+\wh{f}_2-\wh{t}$\end{tabular}}}}%
    \put(0.37676131,0.00377784){\color[rgb]{0,0,0}\makebox(0,0)[lt]{\lineheight{1.25}\smash{\begin{tabular}[t]{l}\small Anti-tet rel: $\wh{b}-\wh{s}_1-\wh{s}_2+\wh{t}$\end{tabular}}}}%
  \end{picture}%
\endgroup%

%% file: flowboxdecomp.pdf_tex
\begingroup%
  \makeatletter%
  \providecommand\color[2][]{%
    \errmessage{(Inkscape) Color is used for the text in Inkscape, but the package 'color.sty' is not loaded}%
    \renewcommand\color[2][]{}%
  }%
  \providecommand\transparent[1]{%
    \errmessage{(Inkscape) Transparency is used (non-zero) for the text in Inkscape, but the package 'transparent.sty' is not loaded}%
    \renewcommand\transparent[1]{}%
  }%
  \providecommand\rotatebox[2]{#2}%
  \newcommand*\fsize{\dimexpr\f@size pt\relax}%
  \newcommand*\lineheight[1]{\fontsize{\fsize}{#1\fsize}\selectfont}%
  \ifx\svgwidth\undefined%
    \setlength{\unitlength}{265.69115346bp}%
    \ifx\svgscale\undefined%
      \relax%
    \else%
      \setlength{\unitlength}{\unitlength * \real{\svgscale}}%
    \fi%
  \else%
    \setlength{\unitlength}{\svgwidth}%
  \fi%
  \global\let\svgwidth\undefined%
  \global\let\svgscale\undefined%
  \makeatother%
  \begin{picture}(1,0.4332384)%
    \lineheight{1}%
    \setlength\tabcolsep{0pt}%
    \put(0,0){\includegraphics[width=\unitlength,page=1]{flowboxdecomp.pdf}}%
  \end{picture}%
\endgroup%

%% file: unstableunglue.pdf_tex
\begingroup%
  \makeatletter%
  \providecommand\color[2][]{%
    \errmessage{(Inkscape) Color is used for the text in Inkscape, but the package 'color.sty' is not loaded}%
    \renewcommand\color[2][]{}%
  }%
  \providecommand\transparent[1]{%
    \errmessage{(Inkscape) Transparency is used (non-zero) for the text in Inkscape, but the package 'transparent.sty' is not loaded}%
    \renewcommand\transparent[1]{}%
  }%
  \providecommand\rotatebox[2]{#2}%
  \newcommand*\fsize{\dimexpr\f@size pt\relax}%
  \newcommand*\lineheight[1]{\fontsize{\fsize}{#1\fsize}\selectfont}%
  \ifx\svgwidth\undefined%
    \setlength{\unitlength}{290.76576053bp}%
    \ifx\svgscale\undefined%
      \relax%
    \else%
      \setlength{\unitlength}{\unitlength * \real{\svgscale}}%
    \fi%
  \else%
    \setlength{\unitlength}{\svgwidth}%
  \fi%
  \global\let\svgwidth\undefined%
  \global\let\svgscale\undefined%
  \makeatother%
  \begin{picture}(1,0.4079535)%
    \lineheight{1}%
    \setlength\tabcolsep{0pt}%
    \put(0,0){\includegraphics[width=\unitlength,page=1]{unstableunglue.pdf}}%
  \end{picture}%
\endgroup%

%% file: doubleunglue.pdf_tex
\begingroup%
  \makeatletter%
  \providecommand\color[2][]{%
    \errmessage{(Inkscape) Color is used for the text in Inkscape, but the package 'color.sty' is not loaded}%
    \renewcommand\color[2][]{}%
  }%
  \providecommand\transparent[1]{%
    \errmessage{(Inkscape) Transparency is used (non-zero) for the text in Inkscape, but the package 'transparent.sty' is not loaded}%
    \renewcommand\transparent[1]{}%
  }%
  \providecommand\rotatebox[2]{#2}%
  \newcommand*\fsize{\dimexpr\f@size pt\relax}%
  \newcommand*\lineheight[1]{\fontsize{\fsize}{#1\fsize}\selectfont}%
  \ifx\svgwidth\undefined%
    \setlength{\unitlength}{293.67731049bp}%
    \ifx\svgscale\undefined%
      \relax%
    \else%
      \setlength{\unitlength}{\unitlength * \real{\svgscale}}%
    \fi%
  \else%
    \setlength{\unitlength}{\svgwidth}%
  \fi%
  \global\let\svgwidth\undefined%
  \global\let\svgscale\undefined%
  \makeatother%
  \begin{picture}(1,0.40811205)%
    \lineheight{1}%
    \setlength\tabcolsep{0pt}%
    \put(0,0){\includegraphics[width=\unitlength,page=1]{doubleunglue.pdf}}%
  \end{picture}%
\endgroup%

%% file: branchsimplefactorlift.pdf_tex
\begingroup%
  \makeatletter%
  \providecommand\color[2][]{%
    \errmessage{(Inkscape) Color is used for the text in Inkscape, but the package 'color.sty' is not loaded}%
    \renewcommand\color[2][]{}%
  }%
  \providecommand\transparent[1]{%
    \errmessage{(Inkscape) Transparency is used (non-zero) for the text in Inkscape, but the package 'transparent.sty' is not loaded}%
    \renewcommand\transparent[1]{}%
  }%
  \providecommand\rotatebox[2]{#2}%
  \newcommand*\fsize{\dimexpr\f@size pt\relax}%
  \newcommand*\lineheight[1]{\fontsize{\fsize}{#1\fsize}\selectfont}%
  \ifx\svgwidth\undefined%
    \setlength{\unitlength}{108.10951293bp}%
    \ifx\svgscale\undefined%
      \relax%
    \else%
      \setlength{\unitlength}{\unitlength * \real{\svgscale}}%
    \fi%
  \else%
    \setlength{\unitlength}{\svgwidth}%
  \fi%
  \global\let\svgwidth\undefined%
  \global\let\svgscale\undefined%
  \makeatother%
  \begin{picture}(1,1.95590256)%
    \lineheight{1}%
    \setlength\tabcolsep{0pt}%
    \put(0,0){\includegraphics[width=\unitlength,page=1]{branchsimplefactorlift.pdf}}%
    \put(-0.00420037,0.01062311){\color[rgb]{0,0,0}\makebox(0,0)[lt]{\lineheight{1.25}\smash{\begin{tabular}[t]{l}$\wh{b}_1$\end{tabular}}}}%
    \put(0.22187145,0.16814894){\color[rgb]{0,0,0}\makebox(0,0)[lt]{\lineheight{1.25}\smash{\begin{tabular}[t]{l}$\wh{v}_{1,m_1}$\end{tabular}}}}%
    \put(0.13165993,0.30999855){\color[rgb]{0.50196078,0,0.50196078}\makebox(0,0)[lt]{\lineheight{1.25}\smash{\begin{tabular}[t]{l}$\wh{e}_{1,1}$\end{tabular}}}}%
    \put(0.64224074,0.92576706){\color[rgb]{0.50196078,0,0.50196078}\makebox(0,0)[lt]{\lineheight{1.25}\smash{\begin{tabular}[t]{l}$\wh{e}_{2,1}$\end{tabular}}}}%
    \put(0.41681757,1.19173033){\color[rgb]{0.50196078,0,0.50196078}\makebox(0,0)[lt]{\lineheight{1.25}\smash{\begin{tabular}[t]{l}$\wh{e}_{3,1}$\end{tabular}}}}%
    \put(0.56149859,1.48691903){\color[rgb]{0.50196078,0,0.50196078}\makebox(0,0)[lt]{\lineheight{1.25}\smash{\begin{tabular}[t]{l}$\wh{e}_{4,1}$\end{tabular}}}}%
    \put(0.28917675,1.75309184){\color[rgb]{0.50196078,0,0.50196078}\makebox(0,0)[lt]{\lineheight{1.25}\smash{\begin{tabular}[t]{l}$\wh{e}_{4,2}$\end{tabular}}}}%
    \put(0.43266233,0.61178352){\color[rgb]{0.50196078,0,0.50196078}\makebox(0,0)[lt]{\lineheight{1.25}\smash{\begin{tabular}[t]{l}$\wh{e}_{1,2}$\end{tabular}}}}%
    \put(0.25400607,0.63676458){\color[rgb]{0,0,0}\makebox(0,0)[lt]{\lineheight{1.25}\smash{\begin{tabular}[t]{l}$\wh{v}_{6,1}$\end{tabular}}}}%
    \put(0.44492181,1.60107618){\color[rgb]{0,0,0}\makebox(0,0)[lt]{\lineheight{1.25}\smash{\begin{tabular}[t]{l}$\wh{v}_{4,1}$\end{tabular}}}}%
    \put(0.19908287,1.84955645){\color[rgb]{0,0,0}\makebox(0,0)[lt]{\lineheight{1.25}\smash{\begin{tabular}[t]{l}$\wh{v}_{4,2}$\end{tabular}}}}%
    \put(0.48901465,0.41476685){\color[rgb]{0,0,0}\makebox(0,0)[lt]{\lineheight{1.25}\smash{\begin{tabular}[t]{l}$\wh{v}_{1,1}=\wh{v}_{5,1}$\end{tabular}}}}%
    \put(0.52895418,1.05553103){\color[rgb]{0,0,0}\makebox(0,0)[lt]{\lineheight{1.25}\smash{\begin{tabular}[t]{l}$\wh{v}_{2,1}=\wh{v}_{3,m_3}$\end{tabular}}}}%
    \put(0.77308394,1.29290041){\color[rgb]{0,0,0}\makebox(0,0)[lt]{\lineheight{1.25}\smash{\begin{tabular}[t]{l}$\wh{v}_{3,1}=\wh{v}_{4,m_4}$\end{tabular}}}}%
    \put(0.60997809,0.19184232){\color[rgb]{0,0,0}\makebox(0,0)[lt]{\lineheight{1.25}\smash{\begin{tabular}[t]{l}$\wh{b}_5$\end{tabular}}}}%
    \put(0.05014913,0.5054197){\color[rgb]{0,0,0}\makebox(0,0)[lt]{\lineheight{1.25}\smash{\begin{tabular}[t]{l}$\wh{b}_6$\end{tabular}}}}%
    \put(0.87032234,0.61178352){\color[rgb]{0,0,0}\makebox(0,0)[lt]{\lineheight{1.25}\smash{\begin{tabular}[t]{l}$\wh{b}_2$\end{tabular}}}}%
    \put(0.32094259,0.91526281){\color[rgb]{0,0,0}\makebox(0,0)[lt]{\lineheight{1.25}\smash{\begin{tabular}[t]{l}$\wh{b}_3$\end{tabular}}}}%
    \put(0.81579993,1.14984057){\color[rgb]{0,0,0}\makebox(0,0)[lt]{\lineheight{1.25}\smash{\begin{tabular}[t]{l}$\wh{b}_4$\end{tabular}}}}%
    \put(0.83125286,0.75724652){\color[rgb]{0,0,0}\makebox(0,0)[lt]{\lineheight{1.25}\smash{\begin{tabular}[t]{l}$\wh{v}_{1,2}=\wh{v}_{2,m_2}$\end{tabular}}}}%
  \end{picture}%
\endgroup%

%% file: domainsplit.pdf_tex
\begingroup%
  \makeatletter%
  \providecommand\color[2][]{%
    \errmessage{(Inkscape) Color is used for the text in Inkscape, but the package 'color.sty' is not loaded}%
    \renewcommand\color[2][]{}%
  }%
  \providecommand\transparent[1]{%
    \errmessage{(Inkscape) Transparency is used (non-zero) for the text in Inkscape, but the package 'transparent.sty' is not loaded}%
    \renewcommand\transparent[1]{}%
  }%
  \providecommand\rotatebox[2]{#2}%
  \newcommand*\fsize{\dimexpr\f@size pt\relax}%
  \newcommand*\lineheight[1]{\fontsize{\fsize}{#1\fsize}\selectfont}%
  \ifx\svgwidth\undefined%
    \setlength{\unitlength}{356.14293376bp}%
    \ifx\svgscale\undefined%
      \relax%
    \else%
      \setlength{\unitlength}{\unitlength * \real{\svgscale}}%
    \fi%
  \else%
    \setlength{\unitlength}{\svgwidth}%
  \fi%
  \global\let\svgwidth\undefined%
  \global\let\svgscale\undefined%
  \makeatother%
  \begin{picture}(1,0.38976828)%
    \lineheight{1}%
    \setlength\tabcolsep{0pt}%
    \put(0,0){\includegraphics[width=\unitlength,page=1]{domainsplit.pdf}}%
    \put(0.40125951,0.21569299){\color[rgb]{0,0,0}\makebox(0,0)[lt]{\lineheight{1.25}\smash{\begin{tabular}[t]{l}(2)\end{tabular}}}}%
    \put(0.22857587,0.21569299){\color[rgb]{0,0,0}\makebox(0,0)[lt]{\lineheight{1.25}\smash{\begin{tabular}[t]{l}(1)\end{tabular}}}}%
    \put(0.62869639,0.21569299){\color[rgb]{0,0,0}\makebox(0,0)[lt]{\lineheight{1.25}\smash{\begin{tabular}[t]{l}(3)\end{tabular}}}}%
    \put(0.57763066,0.26868709){\color[rgb]{0,0,0}\makebox(0,0)[lt]{\lineheight{1.25}\smash{\begin{tabular}[t]{l}$D$\end{tabular}}}}%
    \put(0.03934569,0.00494943){\color[rgb]{0,0,0}\makebox(0,0)[lt]{\lineheight{1.25}\smash{\begin{tabular}[t]{l}$D_1$\end{tabular}}}}%
    \put(0.18035005,0.00494943){\color[rgb]{0,0,0}\makebox(0,0)[lt]{\lineheight{1.25}\smash{\begin{tabular}[t]{l}$D_2$\end{tabular}}}}%
    \put(0.40456008,0.00494943){\color[rgb]{0,0,0}\makebox(0,0)[lt]{\lineheight{1.25}\smash{\begin{tabular}[t]{l}$D_1$\end{tabular}}}}%
    \put(0.85349611,0.10427614){\color[rgb]{0,0,0}\makebox(0,0)[lt]{\lineheight{1.25}\smash{\begin{tabular}[t]{l}$D_1$\end{tabular}}}}%
    \put(0.54556425,0.00494943){\color[rgb]{0,0,0}\makebox(0,0)[lt]{\lineheight{1.25}\smash{\begin{tabular}[t]{l}$D_2$\end{tabular}}}}%
    \put(0.90890923,0.01816613){\color[rgb]{0,0,0}\makebox(0,0)[lt]{\lineheight{1.25}\smash{\begin{tabular}[t]{l}$D_2$\end{tabular}}}}%
  \end{picture}%
\endgroup%

%% file: Part1_v1_arxiv.bbl
\begin{thebibliography}{CLMM22}

\bibitem[AT]{AT25b}
Antonio Alfieri and Chi~Cheuk Tsang.
\newblock Heegaard {F}loer theory and pseudo-{A}nosov flows {II}: Differential and orbit counting.
\newblock In preparation.

\bibitem[AT24]{AT24}
Ian Agol and Chi~Cheuk Tsang.
\newblock Dynamics of veering triangulations: infinitesimal components of their flow graphs and applications.
\newblock {\em Algebr. Geom. Topol.}, 24(6):3401--3453, 2024.
\newblock \href {https://doi.org/10.2140/agt.2024.24.3401} {\path{doi:10.2140/agt.2024.24.3401}}.

\bibitem[BGH24]{BGH24}
Steven Boyer, Cameron~McA. Gordon, and Ying Hu.
\newblock Recalibrating $\mathbb{R}$-order trees and {$\mbox{Homeo}_+(S^1)$}-representations of link groups, 2024.
\newblock URL: \url{https://arxiv.org/abs/2306.10357}, \href {https://arxiv.org/abs/2306.10357} {\path{arXiv:2306.10357}}.

\bibitem[BL70]{BL70}
R.~Bowen and O.~E. Lanford, III.
\newblock Zeta functions of restrictions of the shift transformation.
\newblock In {\em Global {A}nalysis ({P}roc. {S}ympos. {P}ure {M}ath., {V}ols. {XIV}, {XV}, {XVI}, {B}erkeley, {C}alif., 1968)}, Proc. Sympos. Pure Math., XIV-XVI, pages 43--49. Amer. Math. Soc., Providence, RI, 1970.

\bibitem[CD03]{CD03}
Danny Calegari and Nathan~M. Dunfield.
\newblock Laminations and groups of homeomorphisms of the circle.
\newblock {\em Invent. Math.}, 152(1):149--204, 2003.
\newblock \href {https://doi.org/10.1007/s00222-002-0271-6} {\path{doi:10.1007/s00222-002-0271-6}}.

\bibitem[CLMM22]{CLMM22}
Kai Cieliebak, Oleg Lazarev, Thomas Massoni, and Agustin Moreno.
\newblock Floer theory of {A}nosov flows in dimension three, 2022.
\newblock URL: \url{https://arxiv.org/abs/2211.07453}, \href {https://arxiv.org/abs/2211.07453} {\path{arXiv:2211.07453}}.

\bibitem[EGH00]{EGH00}
Y.~Eliashberg, A.~Givental, and H.~Hofer.
\newblock Introduction to symplectic field theory.
\newblock Number Special Volume, Part II, pages 560--673. 2000.
\newblock GAFA 2000 (Tel Aviv, 1999).
\newblock URL: \url{https://doi.org/10.1007/978-3-0346-0425-3_4}, \href {https://doi.org/10.1007/978-3-0346-0425-3\_4} {\path{doi:10.1007/978-3-0346-0425-3\_4}}.

\bibitem[Fen99]{Fen99}
S\'{e}rgio~R. Fenley.
\newblock Foliations with good geometry.
\newblock {\em J. Amer. Math. Soc.}, 12(3):619--676, 1999.
\newblock \href {https://doi.org/10.1090/S0894-0347-99-00304-5} {\path{doi:10.1090/S0894-0347-99-00304-5}}.

\bibitem[Fen16]{Fen16}
S\'{e}rgio~R. Fenley.
\newblock Quasigeodesic pseudo-{A}nosov flows in hyperbolic 3-manifolds and connections with large scale geometry.
\newblock {\em Adv. Math.}, 303:192--278, 2016.
\newblock \href {https://doi.org/10.1016/j.aim.2016.05.015} {\path{doi:10.1016/j.aim.2016.05.015}}.

\bibitem[FG13]{FG13}
David Futer and Fran\c{c}ois Gu\'{e}ritaud.
\newblock Explicit angle structures for veering triangulations.
\newblock {\em Algebr. Geom. Topol.}, 13(1):205--235, 2013.
\newblock \href {https://doi.org/10.2140/agt.2013.13.205} {\path{doi:10.2140/agt.2013.13.205}}.

\bibitem[FJR11]{FJR11}
Stefan Friedl, Andr\'{a}s Juh\'{a}sz, and Jacob Rasmussen.
\newblock The decategorification of sutured {F}loer homology.
\newblock {\em J. Topol.}, 4(2):431--478, 2011.
\newblock \href {https://doi.org/10.1112/jtopol/jtr007} {\path{doi:10.1112/jtopol/jtr007}}.

\bibitem[Flo88]{Flo88}
Andreas Floer.
\newblock Morse theory for {L}agrangian intersections.
\newblock {\em J. Differential Geom.}, 28(3):513--547, 1988.
\newblock URL: \url{http://projecteuclid.org/euclid.jdg/1214442477}.

\bibitem[FM01]{FM01}
S\'{e}rgio Fenley and Lee Mosher.
\newblock Quasigeodesic flows in hyperbolic 3-manifolds.
\newblock {\em Topology}, 40(3):503--537, 2001.
\newblock \href {https://doi.org/10.1016/S0040-9383(99)00072-5} {\path{doi:10.1016/S0040-9383(99)00072-5}}.

\bibitem[Fri83]{Fri83}
David Fried.
\newblock Transitive {A}nosov flows and pseudo-{A}nosov maps.
\newblock {\em Topology}, 22(3):299--303, 1983.
\newblock \href {https://doi.org/10.1016/0040-9383(83)90015-0} {\path{doi:10.1016/0040-9383(83)90015-0}}.

\bibitem[Fri86]{Fri86}
David Fried.
\newblock The zeta functions of {R}uelle and {S}elberg. {I}.
\newblock {\em Ann. Sci. \'{E}cole Norm. Sup. (4)}, 19(4):491--517, 1986.
\newblock URL: \url{http://www.numdam.org/item?id=ASENS_1986_4_19_4_491_0}.

\bibitem[Goo83]{Goo83}
Sue Goodman.
\newblock Dehn surgery on {A}nosov flows.
\newblock In {\em Geometric dynamics ({R}io de {J}aneiro, 1981)}, volume 1007 of {\em Lecture Notes in Math.}, pages 300--307. Springer, Berlin, 1983.
\newblock \href {https://doi.org/10.1007/BFb0061421} {\path{doi:10.1007/BFb0061421}}.

\bibitem[Hoz24]{Hoz24}
Surena Hozoori.
\newblock Symplectic geometry of {A}nosov flows in dimension 3 and bi-contact topology.
\newblock {\em Adv. Math.}, 450:Paper No. 109764, 41, 2024.
\newblock \href {https://doi.org/10.1016/j.aim.2024.109764} {\path{doi:10.1016/j.aim.2024.109764}}.

\bibitem[HT80]{HT80}
Michael Handel and William~P. Thurston.
\newblock Anosov flows on new three manifolds.
\newblock {\em Invent. Math.}, 59(2):95--103, 1980.
\newblock \href {https://doi.org/10.1007/BF01390039} {\path{doi:10.1007/BF01390039}}.

\bibitem[Hut14]{Hut14}
Michael Hutchings.
\newblock Lecture notes on embedded contact homology.
\newblock In {\em Contact and symplectic topology}, volume~26 of {\em Bolyai Soc. Math. Stud.}, pages 389--484. J\'{a}nos Bolyai Math. Soc., Budapest, 2014.
\newblock URL: \url{https://doi.org/10.1007/978-3-319-02036-5_9}, \href {https://doi.org/10.1007/978-3-319-02036-5\_9} {\path{doi:10.1007/978-3-319-02036-5\_9}}.

\bibitem[Juh06]{Juh06}
Andr\'{a}s Juh\'{a}sz.
\newblock Holomorphic discs and sutured manifolds.
\newblock {\em Algebr. Geom. Topol.}, 6:1429--1457, 2006.
\newblock \href {https://doi.org/10.2140/agt.2006.6.1429} {\path{doi:10.2140/agt.2006.6.1429}}.

\bibitem[Juh08]{Juh08}
Andr\'{a}s Juh\'{a}sz.
\newblock Floer homology and surface decompositions.
\newblock {\em Geom. Topol.}, 12(1):299--350, 2008.
\newblock \href {https://doi.org/10.2140/gt.2008.12.299} {\path{doi:10.2140/gt.2008.12.299}}.

\bibitem[JZ24]{JZ24}
Malo Jézéquel and Jonathan Zung.
\newblock Zeta functions and the {F}ried conjecture for smooth pseudo-{A}nosov flows, 2024.
\newblock URL: \url{https://arxiv.org/abs/2409.17014}, \href {https://arxiv.org/abs/2409.17014} {\path{arXiv:2409.17014}}.

\bibitem[Lip06]{Lip06}
Robert Lipshitz.
\newblock A cylindrical reformulation of {H}eegaard {F}loer homology.
\newblock {\em Geom. Topol.}, 10:955--1096, 2006.
\newblock [Paging previously given as 955--1097].
\newblock \href {https://doi.org/10.2140/gt.2006.10.955} {\path{doi:10.2140/gt.2006.10.955}}.

\bibitem[LMT23]{LMT23}
Michael~P. Landry, Yair~N. Minsky, and Samuel~J. Taylor.
\newblock Flows, growth rates, and the veering polynomial.
\newblock {\em Ergodic Theory Dynam. Systems}, 43(9):3026--3107, 2023.
\newblock \href {https://doi.org/10.1017/etds.2022.63} {\path{doi:10.1017/etds.2022.63}}.

\bibitem[LMT24a]{LMT24}
Michael~P. Landry, Yair~N. Minsky, and Samuel~J. Taylor.
\newblock A polynomial invariant for veering triangulations.
\newblock {\em J. Eur. Math. Soc. (JEMS)}, 26(2):731--788, 2024.
\newblock \href {https://doi.org/10.4171/jems/1368} {\path{doi:10.4171/jems/1368}}.

\bibitem[LMT24b]{LMT24b}
Michael~P. Landry, Yair~N. Minsky, and Samuel~J. Taylor.
\newblock Transverse surfaces and pseudo-{A}nosov flows, 2024.
\newblock URL: \url{https://arxiv.org/abs/2406.17717}, \href {https://arxiv.org/abs/2406.17717} {\path{arXiv:2406.17717}}.

\bibitem[Mos92]{Mos92a}
Lee Mosher.
\newblock Dynamical systems and the homology norm of a {$3$}-manifold. {I}. {E}fficient intersection of surfaces and flows.
\newblock {\em Duke Math. J.}, 65(3):449--500, 1992.
\newblock \href {https://doi.org/10.1215/S0012-7094-92-06518-5} {\path{doi:10.1215/S0012-7094-92-06518-5}}.

\bibitem[Mos96]{Mos96}
Lee Mosher.
\newblock Laminations and flows transverse to finite depth foliations.
\newblock {\em Preprint}, 1996.

\bibitem[Nor76]{Nor76}
D.~G. Northcott.
\newblock {\em Finite free resolutions}.
\newblock Cambridge Tracts in Mathematics, No. 71. Cambridge University Press, Cambridge-New York-Melbourne, 1976.

\bibitem[OS04]{OS04c}
Peter Ozsv\'{a}th and Zolt\'{a}n Szab\'{o}.
\newblock Holomorphic disks and topological invariants for closed three-manifolds.
\newblock {\em Ann. of Math. (2)}, 159(3):1027--1158, 2004.
\newblock \href {https://doi.org/10.4007/annals.2004.159.1027} {\path{doi:10.4007/annals.2004.159.1027}}.

\bibitem[OS08]{OS08}
Peter Ozsv\'{a}th and Zolt\'{a}n Szab\'{o}.
\newblock Holomorphic disks, link invariants and the multi-variable {A}lexander polynomial.
\newblock {\em Algebr. Geom. Topol.}, 8(2):615--692, 2008.
\newblock \href {https://doi.org/10.2140/agt.2008.8.615} {\path{doi:10.2140/agt.2008.8.615}}.

\bibitem[OSS]{OSS}
Peter~S. Ozsváth, András~I. Stipsicz, and Zoltán Szabó.
\newblock Heegaard {F}loer homology.
\newblock In preparation.
\newblock URL: \url{https://web.math.princeton.edu/~petero/konyv2.pdf}.

\bibitem[Per08]{Per08}
Timothy Perutz.
\newblock Hamiltonian handleslides for {H}eegaard {F}loer homology.
\newblock In {\em Proceedings of {G}\"{o}kova {G}eometry-{T}opology {C}onference 2007}, pages 15--35. G\"{o}kova Geometry/Topology Conference (GGT), G\"{o}kova, 2008.

\bibitem[Pol20]{Pol20}
Mark Pollicott.
\newblock Dynamical zeta functions and the distribution of orbits.
\newblock In {\em Handbook of group actions. {V}}, volume~48 of {\em Adv. Lect. Math. (ALM)}, pages 399--440. Int. Press, Somerville, MA, [2020] \copyright 2020.

\bibitem[Rob01]{Rob01}
Rachel Roberts.
\newblock Taut foliations in punctured surface bundles. {I}.
\newblock {\em Proc. London Math. Soc. (3)}, 82(3):747--768, 2001.
\newblock \href {https://doi.org/10.1112/plms/82.3.747} {\path{doi:10.1112/plms/82.3.747}}.

\bibitem[Sal23]{Sal21}
Federico Salmoiraghi.
\newblock Surgery on {A}nosov flows using bi-contact geometry, 2023.
\newblock \href {https://arxiv.org/abs/2104.07109} {\path{arXiv:2104.07109}}.

\bibitem[SS]{SSpart5}
Saul Schleimer and Henry Segerman.
\newblock From veering triangulations to pseudo-{A}nosov flows and back again.
\newblock In preparation.

\bibitem[SS19]{SS19}
Saul Schleimer and Henry Segerman.
\newblock From veering triangulations to link spaces and back again, 2019.
\newblock \href {https://arxiv.org/abs/1911.00006} {\path{arXiv:1911.00006}}.

\bibitem[SS20]{SS20}
Saul Schleimer and Henry Segerman.
\newblock Essential loops in taut ideal triangulations.
\newblock {\em Algebr. Geom. Topol.}, 20(1):487--501, 2020.
\newblock \href {https://doi.org/10.2140/agt.2020.20.487} {\path{doi:10.2140/agt.2020.20.487}}.

\bibitem[SS21]{SS21}
Saul Schleimer and Henry Segerman.
\newblock From loom spaces to veering triangulations, 2021.
\newblock \href {https://arxiv.org/abs/2108.10264} {\path{arXiv:2108.10264}}.

\bibitem[SS23]{SS23}
Saul Schleimer and Henry Segerman.
\newblock From veering triangulations to dynamic pairs, 2023.
\newblock \href {https://arxiv.org/abs/2305.08799} {\path{arXiv:2305.08799}}.

\bibitem[Tau10]{Tau10}
Clifford~Henry Taubes.
\newblock Embedded contact homology and {S}eiberg-{W}itten {F}loer cohomology {I}.
\newblock {\em Geom. Topol.}, 14(5):2497--2581, 2010.
\newblock \href {https://doi.org/10.2140/gt.2010.14.2497} {\path{doi:10.2140/gt.2010.14.2497}}.

\bibitem[Tsa23a]{Tsa23b}
Chi~Cheuk Tsang.
\newblock On the set of normalized dilatations of fully-punctured pseudo-{A}nosov maps, 2023.
\newblock \href {https://arxiv.org/abs/2306.10245} {\path{arXiv:2306.10245}}.

\bibitem[Tsa23b]{Tsa23}
Chi~Cheuk Tsang.
\newblock Veering branched surfaces, surgeries, and geodesic flows.
\newblock {\em New York J. Math.}, 29:1425--1495, 2023.

\bibitem[Tsa23c]{Tsathesis}
Chi~Cheuk Tsang.
\newblock {\em Veering Triangulations and Pseudo-{A}nosov Flows}.
\newblock PhD thesis, 2023.
\newblock URL: \url{https://www.proquest.com/dissertations-theses/veering-triangulations-pseudo-anosov-flows/docview/2869041415/se-2}.

\bibitem[Tsa24a]{Tsa24a}
Chi~Cheuk Tsang.
\newblock Constructing {B}irkhoff sections for pseudo-{A}nosov flows with controlled complexity.
\newblock {\em Ergodic Theory Dynam. Systems}, 44(8):2308--2360, 2024.
\newblock \href {https://doi.org/10.1017/etds.2023.105} {\path{doi:10.1017/etds.2023.105}}.

\bibitem[Tsa24b]{Tsa24}
Chi~Cheuk Tsang.
\newblock Horizontal {G}oodman surgery and almost equivalence of pseudo-{A}nosov flows, 2024.
\newblock URL: \url{https://arxiv.org/abs/2401.01847}, \href {https://arxiv.org/abs/2401.01847} {\path{arXiv:2401.01847}}.

\bibitem[Vit87]{Vit87}
Claude Viterbo.
\newblock Intersection de sous-vari\'{e}t\'{e}s lagrangiennes, fonctionnelles d'action et indice des syst\`emes hamiltoniens.
\newblock {\em Bull. Soc. Math. France}, 115(3):361--390, 1987.
\newblock URL: \url{http://www.numdam.org/item?id=BSMF_1987__115__361_0}.

\bibitem[Yu23]{Yu23}
Bin Yu.
\newblock Anosov flows on {D}ehn surgeries on the figure-eight knot.
\newblock {\em Duke Math. J.}, 172(11):2195--2240, 2023.
\newblock \href {https://doi.org/10.1215/00127094-2022-0079} {\path{doi:10.1215/00127094-2022-0079}}.

\bibitem[Zun]{Zun25}
Jonathan Zung.
\newblock Anosov flows and the pair of pants differential.
\newblock In preparation.

\end{thebibliography}
